\documentclass[12pt]{article}

\usepackage{exscale,relsize}
\usepackage{amsmath}
\usepackage{amsfonts}
\usepackage[colorlinks=true, linkcolor=blue]{hyperref}
\usepackage{amssymb}
\usepackage{calc}
\usepackage{theorem}
\usepackage{pifont}      % needed by dingautolist
\usepackage{array}
\usepackage{color}
\usepackage{enumerate}  % Collide with package enumitem???
\usepackage{bbm}
\usepackage{graphicx} 
\usepackage{subcaption}
\usepackage{caption}

\usepackage{algcompatible, algorithm}
\usepackage{mathtools}
\usepackage{soul}  % HILIGHTER
\usepackage{enumitem} % TIGHTER ITEMIZE, ENUMERATE
    \setlistdepth{20}
    \newlist{ditemize}{itemize}{6}
    \setlist[ditemize]{label={}, leftmargin=1em}

\newtheorem{theorem}{Theorem}[section]
\newtheorem{lemma}[theorem]{Lemma}
\newtheorem{fact}[theorem]{Fact}
\newtheorem{corollary}[theorem]{Corollary}
\newtheorem{proposition}[theorem]{Proposition}
\newtheorem{definition}[theorem]{Definition}

\theoremstyle{plain}{\theorembodyfont{\rmfamily}
\newtheorem{assumption}[theorem]{Assumption}}
\theoremstyle{plain}{\theorembodyfont{\rmfamily}
}
\theoremstyle{plain}{\theorembodyfont{\rmfamily}

\theoremstyle{plain}{\theorembodyfont{\rmfamily}
}
\theoremstyle{plain}{\theorembodyfont{\rmfamily}
\newtheorem{remark}[theorem]{Remark}}
\theoremstyle{plain}{\theorembodyfont{\rmfamily}
}

\numberwithin{equation}{section}

\newcommand{\RR}{\ensuremath{\mathbb R}}

\newcommand{\dom}{\ensuremath{\operatorname{dom}}}

\newcommand{\hprox}{\ensuremath{\operatorname{prox}}}

\newcommand{\inte}{\ensuremath{\operatorname{int}}}

\newcommand{\reli}{\ensuremath{\operatorname{ri}}}

\newcommand{\pluss}{{\hskip1pt \raise1pt\vbox{\hrule width6pt \vskip1pt
\hrule width6pt}\kern-4pt{\lower1pt\hbox{\vrule height6pt \kern1pt\vrule
height6pt}}\hskip5pt}}

\newcommand{\argmin}{\mathop{\rm argmin}\limits}

\newcommand{\N}{\ensuremath{\mathbb N}}
\newcommand{\Z}{\ensuremath{\mathbb Z}}
\DeclareMathOperator{\dist}{\mathop{dist}}
\DeclareMathOperator{\rng}{\mathop{rng}}
\DeclareMathOperator{\diam}{\mathop{diam}}

\title{{
    \fontfamily{ptm}\selectfont 
    A Near-Optimal Total Complexity for the Inexact Accelerated Proximal Gradient Method via Quadratic Growth
    }
}

\author{
    Hongda Li
    \thanks{
        Department of Mathematics, I.K. Barber Faculty of Science, The University of British Columbia, Kelowna, BC Canada V1V 1V7.
        E-mail: \texttt{alto@mail.ubc.ca}.
    } and
    Xianfu Wang
    \thanks{
        Department of Mathematics, I.K. Barber Faculty of Science, The University of British Columbia, Kelowna, BC Canada V1V 1V7.
        E-mail: \texttt{shawn.wang@ubc.ca}.
    }
}

\begin{document}

% TITLE, ABSTRACT ==============================================================
\date{\today}
\maketitle
\begin{abstract} 
    \noindent
    We consider the optimization problem $\min_{x\in\RR^n}{F(x):=f(x)+\omega(Ax)}$, where $f$ is an $L$-Lipschitz smooth function, and $\omega$ is a proper, lower semicontinuous, and convex function.
    We prove in this paper that when $\omega$ is a conic polyhedral function, the inexact accelerated proximal gradient method (IAPG), employed in a double-loop structure, achieves a total complexity of $\mathcal O(\ln(1/\varepsilon)/\sqrt{\varepsilon})$ measured by the total number of calls to the proximal operator of the convex conjugate $\omega^\star$ and the gradient of $f$ to achieve $\varepsilon$-optimality in function value.
    To the best of our knowledge, this improves upon the best-known complexity for IAPG.
    The key theoretical ingredient is a quadratic growth condition on the dual of the inexact proximal problem, which arises from the conic polyhedral structure of $\omega$ and implies linear convergence of the inner proximal gradient loop.
    To validate these findings, we conduct numerical experiments on a robust TV-$\ell_2$ signal recovery problem, demonstrating fast convergence. 
\end{abstract}


\section{Introduction}\label{sec:introduction}
    Nesterov's acceleration \cite{nesterov_method_1983} is a first-order method originally conceived to improve the convergence rate of the gradient descent method for convex functions with Lipschitz-continuous gradient.
    Since then, several major extensions of Nesterov's acceleration have been proposed in the literature; one prominent example is the accelerated proximal gradient (APG) method. 
    It adapts to nonsmooth objective functions, see for example, Beck and Teboulle \cite{beck_fast_2009}.
    APG arises in numerous problems in engineering, finance, imaging and signal processing.
    % \#TODO: Maybe add citations here too? 
    \par
    In the past decade, progress has been made in APG to extend its capabilities to composite optimization problems in which exact evaluation of the proximal operator is not available, necessitating inexact evaluation of the proximal operator.
    As a result, this new variant is referred to as the method of Inexact Accelerated Proximal Gradient (IAPG).
    In this paper, we improve the total complexity results of a double-loop IAPG method by exploiting a mild but favorable condition on the nonsmooth part of the objective.
    We show that if the nonsmooth part of the objective is a conic polyhedral function composed with a linear operator, then a near-optimal convergence rate is achievable.
    To demonstrate our theoretical results, we formulate a robust variant of TV-$\ell^2$.
    We use this formulation as a benchmark, demonstrating fast convergence and a favorable scaling of the inner-loop complexity relative to the outer-loop complexity.
    \par
    Before proceeding further, we clarify the phrase ``near-optimal total complexity" used in the title.
    Nesterov \cite[Theorem 2.1.7, Assumption 2.14]{nesterov_lectures_2018} established that any first-order algorithm satisfying a linear span assumption requires at least $\mathcal O(\varepsilon^{-1/2})$ number of gradient (or proximal gradient) evaluations to achieve $\varepsilon$-optimality, if minimizers exist.
    We show that the total complexity of the Inexact Accelerated Proximal Gradient (IAPG) method, measured by the total number of iterations of the inner and outer loops needed to achieve $\varepsilon$-optimality in function value, is bounded by $\mathcal O\left(\varepsilon^{-1/2}\ln(\varepsilon^{-1})\right)$ when $\omega$ is a conic polyhedral function.
    To the best of our knowledge, our theoretical results improve those of the literature \cite{bello-cruz_inexact_2020-1, schmidt_convergence_2011, villa_accelerated_2013}.
    \subsection{Problem formulation}
        Let $f:\RR^n \rightarrow \RR$ be a convex $L$-Lipschitz smooth function and let $A \in \RR^{m\times n}$ be a matrix.
        Let $\omega: \RR^m \rightarrow \overline \RR$ be a proper, closed and convex function. 
        We are interested in problems of the form:
        \begin{align}\label{expr:major-optimization-problem}
            \min_{x \in \RR^n} F(x) := f(x) + \omega(A x). 
        \end{align}
        We assume that the solution set is nonempty and let $\bar x$ denote a minimizer of $F$.
        Observe that \eqref{expr:major-optimization-problem} is an additively composite optimization problem whose nonsmooth part is $\omega(Ax)$.
        However, the Accelerated Proximal Gradient Method (APG) is not directly applicable to a general $A \in \RR^{m\times n}$, as $\hprox_{\lambda\omega\circ A}$ lacks a closed form in general.
        \par
        A wide array of significant real-world applications can be cast in the form of \eqref{expr:major-optimization-problem}.
        Examples include, but are not limited to, robust imaging applications \cite{ehrhardt_multicontrast_2016, joshi_mri_2009, sawatzky_em-tv_2013, xu_convex_2022, zhang_robust_2022}, most of which can be cast as Total Variation minimization problems, as surveyed in Scherzer et al. \cite[Chapter 3]{scherzer_variational_2009}.
        Recently, other non-standard regularizers such as input convex neural networks have been applied to imaging tasks, for example in Mukherjee et al. \cite{mukherjee_data-driven_2024}.
        Besides imaging tasks, problems in statistical inference \cite{christou_risk_2025,tang_fused_2016} appearing in finance and data science can be formulated into \eqref{expr:major-optimization-problem} as well.
    
    \subsection{Motivations}
        To motivate the use of IAPG on large-scale problems in the form of \eqref{expr:major-optimization-problem}, we consider the following robust TV-$\ell_2$ minimization problem where popular algorithms face a computational bottleneck:
        \begin{align}\label{eqn:robust-tvl2}
            \argmin_{x \in \RR^n} \left\lbrace
                \frac{1}{2}\dist\left(Cx - \tilde x \;|\; [-\lambda, \lambda]^n\right)^2 + \eta \Vert A x\Vert_1
            \right\rbrace. 
        \end{align}
        The above optimization problem fits into our formulation in \eqref{expr:major-optimization-problem} with the following components: $\omega(Ax) = \eta\Vert Ax \Vert_1$ (TV-$\ell_1$ regularization), $f(x) = \frac{1}{2}\dist\left(Cx - \tilde x \;|\; [-\lambda, \lambda]^n\right)^2$ (reconstruction fidelity) where $C$ is a box blur matrix with non-periodic boundary condition, $A$ is a first-order finite difference matrix, $\eta > 0$ is the regularization parameter, and $\tilde x$ is the observed signal.
        The fidelity term $f$ is a relaxation of the hard constraint $Cx - \tilde x \in [-\lambda,\lambda]^n$, obtained by replacing the indicator $\delta_{[-\lambda,\lambda]^n}$ with its Moreau envelope evaluated at the residual $Cx - \tilde x$; this renders $f$ smooth and insensitive to small deviations, imparting robustness to noise. 
        To the best of our knowledge, \eqref{eqn:robust-tvl2} has not yet been explicitly formulated in the literature.
        \par
        Both parts of the objective function are forms of PLQ functions \cite[Example 11.18]{rockafellar_variational_1998} which are well known in the literature.
        Furthermore, it fits naturally into the theoretical framework described in Aravkin et al. \cite{aravkin_sparserobust_2013}. 
        In contrast to the Interior Point approach suggested by their work, we consider a first-order method because in image processing, the matrices $C$ and $A$ are usually sparse and large.
        \footnote{
            A standard 1080p image with colors will present a blurring matrix $C$, and finite difference matrix $A$ of size: $6220800\times 6220800$, a size prohibitively large for second order methods. 
        }
        \par
        In the setting of first-order method, it is still challenging to compute the proximal operator of the fidelity term for $\lambda > 0$ and nontrivial choices of $C$, e.g., when $C$ is non-circulant or non-unitary.
        Furthermore, $\hprox_{\lambda\omega\circ A}(x)$ lacks a closed form when $A$ is nontrivial, e.g., when $A$ is not unitary.
        \par
        Well-known algorithms such as the Chambolle Pock algorithm \cite{chambolle_first-order_2011, chambolle_introduction_2016} (PDHG) solves the standard TV-$\ell_2$ problem.
        Applying their framework to \eqref{eqn:robust-tvl2} requires the exact proximal operator of $f$, which lacks a closed form for any nontrivial $C$.
        Alternatively, practitioners can employ an inexact solver for $\hprox_f$, but doing so risks losing the theoretical convergence guarantees of PDHG.
        \par
        Other methods such as the Bregman Splitting Method of Yin et al. \cite{yin_bregman_2008} could be applied.
        However, this method exhibits a slow theoretical convergence compared to PDHG, making it unsuitable for large-scale applications.
        Consequently, IAPG offers a compelling alternative: by removing the need for exact computation of $\hprox_{\lambda f}$ and $\hprox_{\lambda\omega\circ A}$, it enables efficient solution of large-scale problems where conventional proximal methods are prohibitive.
        This motivates the use of IAPG for optimizing \eqref{expr:major-optimization-problem}.

    \subsection{Literature reviews}
        In this section, we review key developments in the literature for addressing the optimization problem in \eqref{expr:major-optimization-problem} using the IAPG method.
        The study of inexact proximal operators traces back to Rockafellar \cite{rockafellar_monotone_1976}, whose inexactness conditions (A) and (B) remain foundational.
        \par
        More recently, Schmidt et al. \cite{schmidt_convergence_2011} and Villa et al. \cite{villa_accelerated_2013} independently utilized the $\epsilon$-subgradient to quantify the inexactness of the proximal operator within the accelerated proximal gradient algorithm.
        In addition to Schmidt et al. and Villa et al., Bello-Cruz et al. \cite{bello-cruz_inexact_2020-1} and Lin and Xu \cite{lin_reducing_2023} present formulations similar to ours.
        Bello-Cruz et al. employ an $\epsilon$-subgradient criterion for the inexact proximal problem in IAPG, similar to our approach; however, they consider only relative error without line search and provide no total complexity results for either the outer or inner loop.
        Lin and Xu \cite{lin_reducing_2023} study IAPG in a context different from ours, as they consider a different class of objective functions.
        Our work extends the framework of Villa et al. \cite{villa_accelerated_2013} in two significant directions: we accommodate backtracking line search and absolute error criteria, and we establish, for the first time, a total complexity bound of $\mathcal{O}(\varepsilon^{-1/2}\ln(\varepsilon^{-1}))$ that accounts for both the outer and inner loop iterations.
        \par
        Another significant line of research is the ``Catalyst" acceleration framework introduced by Lin et al. \cite{lin_catalyst_2018}.
        Unlike Schmidt et al. and Villa et al., this approach accelerates the proximal point method, building on the work of Güler \cite{guler_new_1992}.
        Instead of using the $\epsilon$-subgradient, Lin et al. quantify inexactness via optimality of the proximal problem and accelerate the proximal point method rather than the proximal gradient operator.
        Consequently, this requires an inexact proximal operator applied to the full objective $F$, together with warm-start conditions, to ensure convergence.
        \par
        Notably, the $\epsilon$-subgradient can also be employed for PDHG.
        See, for example, Rasch and Chambolle \cite{rasch_inexact_2020}.
        Their method applies to more general problem classes because both components of the objective can be nonsmooth with a linear composite structure.
        However, their total complexity is worse than ours because, in their analysis, the evaluation of the inexact proximal operator achieves only a sublinear convergence rate.
        See Rasch and Chambolle \cite[Corollary 3]{rasch_inexact_2020}.
        \par
        In nonconvex settings, new theoretical ideas are required.
        Multiple works employ relative error and the envelope function; see, for example, works by Khanh et al. \cite{khanh_inexact_2025}, and Calatroni and Chambolle \cite{calatroni_backtracking_2019}.
    \subsection{Our contributions}
        Our paper makes three substantial contributions to the theory and practice of the IAPG algorithm.
        \begin{enumerate}[nosep]
            \item We extend the theory of the inexact proximal gradient operator via $\epsilon$-subgradient theory.
            Specifically, our inexact proximal gradient inequality (Theorem \ref{thm:inxt-pg-ineq}) accommodates a backtracking line search and supports both relative and absolute error criteria.
            \item We establish a total complexity of $\mathcal O\left(\varepsilon^{-1/2}\ln(\varepsilon^{-1})\right)$ for IAPG in problems where $\omega$ is conic polyhedral.
            This improves upon all prior complexity results for IAPG \cite{lin_catalyst_2018, schmidt_convergence_2011, villa_accelerated_2013}, and is enabled by a quadratic growth condition for the dual of the inexact proximal problem (Theorem \ref{thm:minimizing-dual-pp}).
            \item We validate our theoretical results with numerical experiments on large-scale signal recovery tasks.
            In addition, we provide an open-source, high-performance Julia \cite{bezanson_julia_2017} implementation of IAPG, optimized for minimal memory overhead and C++/FORTRAN level speed.
        \end{enumerate}
        \textbf{The paper is organized as follows.}
        Section \ref{sec:fund-aux-rslts} establishes the foundations of the inexact proximal operator via $\epsilon$-subgradient theory, culminating in the inexact proximal gradient inequality that underpins the outer loop's $\mathcal{O}(1/k^2)$ convergence rate.
        Section \ref{sec:outerloop} defines the outer loop of IAPG and derives its $\mathcal{O}(1/k^2)$ convergence rate.
        We denote by $\epsilon$ the tolerance used for each inner loop call. 
        Section \ref{sec:innlp-lin-cnvg} establishes the linear convergence rate of the inner loop under a quadratic growth condition, yielding an $\mathcal{O}(\ln(\epsilon^{-1}))$ complexity per inner loop call.
        Section \ref{sec:total-cmplx} combines the outer and inner loop analyses to derive the total complexity bound of $\mathcal{O}(\varepsilon^{-1/2}\ln(\varepsilon^{-1}))$, and also establishes an $\mathcal{O}(\varepsilon^{-1}\ln(\varepsilon^{-1}))$ bound for convergence to stationarity.
        Section \ref{sec:alg-impl} presents concrete implementations of the inner and outer loops and verifies that they satisfy the required assumptions.
        Section \ref{sec:innlp-fxn} establishes that the total complexity results apply when $\omega$ is conic polyhedral.
        Finally, Section \ref{sec:numerics} presents two numerical experiments: the first verifies inner loop linear convergence, and the second applies IAPG to \eqref{eqn:robust-tvl2} to demonstrate efficiency on a large-scale problem.

\section{Preliminaries}\label{sec:fund-aux-rslts}
    The objective of this section is to study the inexact proximal operator via $\epsilon$-subgradient; these serve as the foundation for the theory of the inexact proximal gradient operator.
    \par
    The section begins by preparing the reader for our extensions of results in the literature (Theorem \ref{thm:inxt-pg-ineq} and Theorem \ref{thm:minimizing-dual-pp}) through the concept of $\epsilon$-subgradient (Definition \ref{def:esp-subgrad}) and inexact proximal point (Definition \ref{def:inxt-pp}).
    Their roles will are critical for ensuring globally bounded complexity for the inner loop.
    In Section \ref{ssec:inxt-pg-ineq}, we derive the inexact proximal gradient inequality in Theorem \ref{thm:inxt-pg-ineq}, which will be crucial for the convergence analysis of the outer loop of IAPG.
    In Section \ref{ssec:optz-inxt-pp-problem}, we present the proximal point problem, leading to our extension of Villa et al.'s \cite{villa_accelerated_2013} results in Theorem \ref{thm:minimizing-dual-pp}.
    \subsection{Notations and definitions}
        \textbf{Notations.}
        We denote $\overline \RR := \RR\cup \{-\infty, \infty\}$. 
        Let $g: \RR^n \rightarrow\overline \RR$. 
        We denote the Fenchel conjugate of $g$ by $g^\star$ which is defined as $g^\star(x) := \sup_{z\in \RR^n} \{\langle z, x\rangle - g(z)\}$.
        The domain of $g$ is $\dom (g):= \{x\in \RR^n: g(x) < \infty\}$. 
        For all $Q \subseteq \RR^n$, we define the affine hull of $Q$: 
        $$
        \text{affhull}(Q) :=
        \left\lbrace
            \theta_1x_1 + \theta_2x_2 + \cdots +\theta_Nx_N:
            \sum_{i = 1}^{N}\theta_i = 1, x_i \in Q \; \forall i \in \{1, \ldots, N\}, N \in \mathbb N
        \right\rbrace.
        $$
        With the above, we define the relative interior of a set $Q \subseteq \RR^n$ as: 
        \begin{align*}
            \reli(Q):= 
            \{
                x\in Q \left |  
                \exists\; \epsilon > 0 \text{ s.t. }
                \{z: \Vert z - x \Vert < \epsilon\}\cap \text{affhull}(Q)\subseteq Q
                \right.
            \}.
        \end{align*}
        We let $I: \RR^n \rightarrow \RR^n$ denote the identity operator.
        For a matrix $A \in \RR^{m\times n}$, $A^\dagger$ denotes its pseudoinverse, and $\rng(A) := \{Ax : x \in \RR^n\} \subseteq \RR^m$ denotes the range of $A$.
        Let $S \subseteq \RR^n$.
        We denote the projection onto the set $S$ by $\Pi_S$.
        It is defined by $\Pi_S(x) := \argmin_{z \in S} \Vert x - z \Vert$.
        Denote $\dist(x |S)$ to be the distance from $x$ to the set $S$, which is $\dist(x | S) := \min_{z \in S} \Vert z - x\Vert$.
        We define $\diam S := \sup_{x, y \in S} \Vert x - y\Vert$ to be the diameter. 
        Boldface $\mathbf 0$ denotes a vector of zeros in $\RR^n$.
        Denote $\Z_+ = \{0, 1, 2, \ldots\}$ for the set of indices starting at zero and $\N = \{1, 2, \ldots\}$ for indices excluding $0$. 
        \par
        Let $\RR^m, \RR^n$ be our ambient spaces. 
        We write $\Vert \cdot\Vert$ to be the Euclidean norm in $\RR^n$; we write $\Vert \cdot \Vert_1$ to be the $\ell^1$ norm in $\RR^n$ given by $\Vert x\Vert_1 := \sum_{i = 1}^{n} |x_i|$.
        We write $\Vert \cdot\Vert_\infty$ to be the infinity norm in $\RR^n$ given by $\Vert x\Vert_\infty := \max_{i = 1, \ldots, n} |x_i|$. 
        The proximal operator of a proper, closed and convex function $f:\RR^n \rightarrow \overline \RR$ is defined by: 
        \begin{align*}
            \hprox_{\lambda f}(x) := \argmin_{z \in \RR^n} \left\lbrace
                f(z) + \frac{1}{2\lambda} \Vert x - z\Vert^2
            \right\rbrace. 
        \end{align*}
        The indicator function of a set $C \subseteq \RR^n$ is the function defined by:
        \begin{align*}
            \delta_C(x) := \begin{cases}
                0 & \text{if } x \in C, 
                \\
                \infty & \text{otherwise. }
            \end{cases}
        \end{align*}
        For example, we can write $\delta_{\{x \in \RR^n: \Vert x\Vert_1 \le 1\}}$. 
        The word ``tolerance" represents the numerical value needed to exit a for loop structure in the algorithm. 
        We denote the inner loop tolerance by $\epsilon$, and the tolerance of the entire algorithm including the inner loop and outer loop by $\varepsilon$. 
        For example, $\mathcal O\left(\ln(\epsilon^{-1})\right)$ denotes the complexity of the inner loop and $\mathcal O\left(\varepsilon^{1/2}\ln(\varepsilon^{-1})\right)$ denotes the total complexity of the algorithm. 
        \par
        Finally, when presenting proofs, we use numerical subscripts: $\underset{(1)}\le, \underset{(2)}=$ which indicate that some intermediate results are invoked to justify the inequality or equality.
        These steps will be explained immediately after the chain of equalities/inequalities. 
        \par
        The definition below introduces $\epsilon$-gradient for proper functions. 
        It can be viewed as a perturbation of the usual definition of the Fenchel subgradient. 
        \begin{definition}[$\epsilon$-subgradient {\cite[(2.35)]{zalinescu_convex_2002}}]\;\label{def:esp-subgrad}\\
            Let $g: \RR^n \rightarrow \overline \RR$ be proper. 
            Let $\epsilon \ge 0$.
            Then the $\epsilon$-subgradient of $g$ at some $\bar x \in \dom g$ is given by:
            $$
            \begin{aligned}
                \partial g_\epsilon(\bar  x) :=
                \left\lbrace
                    v \in \RR^n \left| \;
                        \langle v, x - \bar  x\rangle \le
                        g(x) - g(\bar x) + \epsilon \;\forall x \in \RR^n
                    \right.
                \right\rbrace.
            \end{aligned}
            $$
            When $\bar x \not \in \dom g$, we set $\partial g_\epsilon(\bar x) = \emptyset$.
        \end{definition}
        \begin{remark}
            $\partial_\epsilon g$ is a multivalued operator. 
            It is not monotone in general even if $g$ is proper, closed and convex; when $\epsilon = 0$ it reduces to the Fenchel subdifferential $\partial g$ if $g$ is proper, closed, and convex.
        \end{remark}
        Next, we introduce results from the literature on the $\epsilon$-subgradient.
        % \#DONE: The assumption on f is unclear. Check the references and make sure it has convexity involved. 
        \begin{fact}[$\epsilon$-Fenchel inequality, {Zalinascu \cite[Theorem 2.4.2]{zalinescu_convex_2002}}]\label{fact:esp-fenchel-ineq}
            Let $\epsilon \ge 0$, and suppose that $g: \RR^n \rightarrow \overline \RR$ is a proper function. 
            Then:
            \begin{align}\label{fact:esp-fenchel-ineq:expr}
                x^* \in \partial_\epsilon f(\bar x)\iff f^\star(x^*) + f(\bar x) \le \langle x^*, \bar x\rangle + \epsilon \implies \bar x \in \partial_\epsilon f^\star(x^*)
            \end{align}
            The $\implies$ strengthens to $\iff$ when $f^{\star\star}(\bar x) = f(\bar x)$ (i.e., $f$ is proper, closed and convex), making all three conditions equivalent.
        \end{fact}
        The definition that follows defines the inexact evaluation of a proximal operator by $\epsilon$-subgradient of a proper, closed and convex function.
        \begin{definition}[The Inexact proximal operator]\label{def:inxt-pp}
            Let $x \in \RR^n$, $\epsilon \ge 0$, $\lambda > 0$.
            $\tilde x$ is an inexact evaluation of the proximal operator at $x$ if and only if:
            \begin{align*}
                \lambda^{-1}(x - \tilde x) \in \partial_{\epsilon} g(\tilde x).
            \end{align*}
            We denote this by $\tilde x \approx_\epsilon \hprox_{\lambda g}(x)$.
        \end{definition}
        \begin{remark}
            This definition is not new; see, e.g., Villa et al. \cite[Definition 2.1]{villa_accelerated_2013}.
            However, our $\epsilon$ differs from that of Villa et al.: our $\epsilon$ corresponds to their $\varepsilon^2/(2\lambda)$, so the two definitions are not directly comparable despite sharing the same conceptual form.
        \end{remark}
        Next, we introduce the resolvent identity.
        It still holds for $\epsilon$-subgradient, and is crucial for developing numerical algorithms that evaluate the proximal operator inexactly. 
        \begin{fact}[the resolvent identity, Rockafellar and Wets {\cite[Lemma 12.14]{rockafellar_variational_1998}}]\;\label{fact:resv-identity}\\
            Let $T: \RR^n \rightarrow 2^{\RR^n}$. Then:
            \begin{align}\label{fact:resv-identity:eqn}
                (I + T)^{-1} = (I - (I + T^{-1})^{-1}).
            \end{align}
        \end{fact}
        \begin{lemma}[inexact Moreau decomposition]\;\label{lemma:inxt-moreau-decomp}\hspace{-0.75em}
            Let $g: \RR^n \rightarrow \overline \RR$ be a closed, convex and proper function. 
            It has the equivalence
            \begin{align*}
                \tilde y \approx_\epsilon \hprox_{\lambda^{-1}g^\star}(\lambda^{-1}y)
                \iff 
                y - \lambda \tilde y \approx_\epsilon \hprox_{\lambda g}(y). 
            \end{align*}
        \end{lemma}
        \begin{proof}
            Consider $\tilde y \approx_\epsilon \hprox_{\lambda^{-1}g^\star}(\lambda^{-1}y)$: 
            \begin{align*}
                & \lambda^{-1} y - \tilde y \in \lambda^{-1} \partial_\epsilon g^\star(\tilde y)
                \\
                \iff & 
                \lambda^{-1} y \in \lambda^{-1} \partial_\epsilon g^\star(\tilde y) + \tilde y 
                = (I + \lambda^{-1}\partial_\epsilon g^\star)(\tilde y)
                \\
                \iff & 
                \tilde y \in (I + \lambda^{-1}\partial_\epsilon g^\star)^{-1}(\lambda^{-1}y)
                \underset{(1)}= \left(
                    I - (I + \partial_\epsilon g\circ(\lambda I))^{-1}
                \right)(\lambda^{-1}y)
                \\
                \iff & \lambda^{-1}y - \tilde y 
                \in (I + \partial_\epsilon g \circ(\lambda I))^{-1}(\lambda^{-1}y)
                \\
                \iff &
                \lambda^{-1}y \in 
                (I + \partial_\epsilon g \circ(\lambda I))(\lambda^{-1}y - \tilde y)
                = (\lambda^{-1}I + \partial_\epsilon g)(y - \lambda\tilde y)
                \\
                \iff & 
                \lambda^{-1}y - (\lambda^{-1}y - \tilde y) = \lambda^{-1}(y - (y - \lambda \tilde y))
                \in \partial_\epsilon g(y - \lambda \tilde y)
                \\
                \underset{\text{Def \ref{def:inxt-pp}}}\iff & 
                y - \lambda \tilde y \approx_\epsilon \hprox_{\lambda g}(y). 
            \end{align*}
            At (1) we apply Fact \ref{fact:resv-identity} with $T = \lambda^{-1}\partial_\epsilon g^\star$, giving $T^{-1} = (\lambda^{-1}\partial_\epsilon g^\star)^{-1} = \partial_\epsilon g\circ(\lambda I)$ by Fact \ref{fact:esp-fenchel-ineq}, which states that $(\partial_\epsilon g^\star)^{-1} = \partial_\epsilon g$ since $g$ is closed, convex and proper.
        \end{proof}
        \begin{definition}[Bregman Divergence of a differentiable function]\;\label{def:breg-div}
            Let $f: \RR^n \rightarrow \RR$ be a differentiable function.
            We define the Bregman divergence of $f$ by:
            \begin{align*}
                D_f: \RR^n \times \RR^n \rightarrow \RR & :
                (x, y) \mapsto f(x) - f(y) - \langle \nabla f(y), x - y\rangle.
            \end{align*}
        \end{definition}
        \begin{remark}
            By our definition here, $f$ is not necessarily a Legendre function, and it need not be in the scope of our paper.
        \end{remark}
        \begin{definition}[Lipschitz smoothness]\label{def:lipz-smooth}
            A convex, differentiable function $f: \RR^n \rightarrow \RR$ is $L$-Lipschitz smooth if there exists $L$ such that: 
            \begin{align*}
                (\forall x \in \RR^n)(\forall y \in \RR^n)\; 
                D_f(x, y) &\le \frac{L}{2}\Vert x - y\Vert^2. 
            \end{align*}
        \end{definition}
        \begin{remark}
            This is also known by the name ``Descent Lemma" in the literature, see for example Beck \cite[Lemma 5.7]{beck_first-order_2017}. 
            % \#DONE: Put a reference for the name "Descent Lemma" here. 
        \end{remark}
        \begin{fact}[Lipschitz smoothness equivalence {\cite[Theorem 18.15]{bauschke_convex_2017}}]\label{fact:lipz-smth-equiv}
            Let $f:\RR^n \rightarrow \RR$ be a convex, differentiable function. 
            The following are equivalent. 
            \begin{enumerate}[nosep]
                \item $f$ is $L$-Lipschitz smooth. 
                \item $\nabla f$ is an $L$-Lipschitz continuous mapping, i.e., $\Vert \nabla f(x) - \nabla f(y)\Vert \le L \Vert x - y\Vert$ for all $x, y \in \RR^n$.
            \end{enumerate}
        \end{fact}
        \begin{remark}
            This fact is from Bauschke and Combettes \cite{bauschke_convex_2017}, page 323. 
            Here, we consider Euclidean space $\RR^n$. 
        \end{remark}

    \subsection{Inexact proximal gradient inequality}\label{ssec:inxt-pg-ineq}
        In this section, we present the definition (Definition \ref{def:inxt-pg}) and characterizations (Lemma \ref{lemma:other-repr-inxt-pg}) of inexact proximal gradient operator along with their assumptions (Assumption \ref{ass:for-inxt-pg-ineq}) leading to the inexact proximal gradient inequality (Theorem \ref{thm:inxt-pg-ineq}).
        \begin{assumption}[for inexact proximal gradient]\;\label{ass:for-inxt-pg-ineq}\\
            Assume $(F, f, g, L)$ satisfy the following.
            \begin{enumerate}[nosep]
                \item $f: \RR^n \rightarrow \RR$ is a convex, $L$-Lipschitz smooth function (Definition \ref{def:lipz-smooth}) which we can evaluate $\nabla f$ exactly and efficiently. 
                \item $g: \RR^n \rightarrow \overline\RR$ is a proper, closed, and convex function whose exact proximal operator is unavailable.
                \item The overall objective is $F = f + g$.
            \end{enumerate}
        \end{assumption}
        \begin{definition}[exact proximal gradient]\;\label{def:exact-pg}\\
            Let $(F, f, g, L)$ satisfy Assumption \ref{ass:for-inxt-pg-ineq}.
            For all $\rho > 0$, $x^+ = T_\rho(x)$ is the exact proximal gradient operator if and only if
            \begin{align*}
                \mathbf 0 \in \nabla f(x) - \rho(x - x^+) + \partial g(x^+). 
            \end{align*}
        \end{definition}
        The following definition extends the proximal gradient operator to the inexact setting by applying the $\epsilon$-subgradient (Definition \ref{def:esp-subgrad}); it is crucial for algorithms in the outer loop of IAPG.
        \begin{definition}[inexact proximal gradient]\label{def:inxt-pg}
            Let $(F, f, g, L)$ satisfy Assumption \ref{ass:for-inxt-pg-ineq}.
            Let $\epsilon \ge 0, \rho > 0$.
            Then, $\tilde x \approx_\epsilon T_\rho(x)$ is an inexact proximal gradient if it satisfies the variational inequality: 
            \begin{align*}
                \mathbf 0 \in \nabla f(x) - \rho(x - \tilde x) + \partial_{\epsilon} g(\tilde x). 
            \end{align*}
        \end{definition}
        \begin{remark}
            The evaluation of $\nabla f$ at any points $x \in \RR^n$ is exact. 
        \end{remark}
        % We introduce the definition of the exact proximal gradient operator for algorithms in the inner loop of IAPG. 
        Note that setting $\epsilon = 0$ in Definition \ref{def:inxt-pg} recovers Definition \ref{def:exact-pg}.
        \par
        The next lemma shows that, the above definition of inexact proximal gradient using $\epsilon$-subgradient is equivalent to the composite of an inexact proximal point of the nonsmooth part $g$ on the gradient of the smooth part $f$, linking it back to Definition \ref{def:inxt-pp} in the previous section. 
        \begin{lemma}[other representations of inexact proximal gradient]\;\label{lemma:other-repr-inxt-pg}\\
            Let $(F, f, g, L)$ satisfy Assumption \ref{ass:for-inxt-pg-ineq}, $\epsilon \ge 0, \rho > 0$. Then for all $\tilde x \approx_\epsilon T_\rho(x)$, the following equivalent representations hold:
            \begin{align*}
                & (x - \rho^{-1}\nabla f(x)) - \tilde x 
                \in \rho^{-1} \partial_\epsilon g(\tilde x)
                \\
                \iff 
                & \tilde x \in (I + \rho^{-1}\partial_\epsilon g)^{-1}
                (x - \rho^{-1}\nabla f(x))
                \\
                \iff 
                & \tilde x \approx_\epsilon \hprox_{\rho^{-1} g}
                \left(x - \rho^{-1}\nabla f(x)\right)
            \end{align*}
        \end{lemma}
        \begin{proof}
            This is immediate. The first $\iff$ uses algebra commonly used for multivalued mappings, and the second $\iff$ takes the resolvent of $\partial_\epsilon g$, which by Definition \ref{def:inxt-pp} is the inexact proximal operator.
        \end{proof}
        \begin{lemma}[$\epsilon$-subgradient basic sum rule]\;\label{lemma:eps-subgrad-wsum-rule}\\
            Let $(F, f, g, L)$ satisfy Assumption \ref{ass:for-inxt-pg-ineq}, $\epsilon \ge 0$.
            Then:
            \begin{align*}
                (\forall x \in \RR^n)\; \partial_\epsilon g(x) + \nabla f(x) \subseteq \partial_\epsilon F(x). 
            \end{align*}
        \end{lemma}
        \begin{proof}
            Fix any $x \in \RR^n$, by Definition \ref{def:esp-subgrad} $\forall v \in \partial_\epsilon g(x)$ if and only if $\forall z \in \RR^n$: 
            \begin{align*}
                - &\epsilon \le g(z) - g(x) - \langle v, z - x\rangle, 
                \\
                &0 \le f(z) - f(x) - \langle \nabla f(x), z - x\rangle. 
            \end{align*}
            Adding the above two expressions yields $-\epsilon \le F(z) - F(x) - \langle \nabla f(x) + v, z - x\rangle$ which is $\nabla f(x) + v \in \partial_\epsilon F(x)$. 
        \end{proof}
        \par
        The following lemma states the fact that the $\epsilon$-subgradient of the objective function can be bounded by the residual of the inexact proximal gradient operator. 
        \begin{lemma}[The proximal gradient residual]\;\label{lemma:pg-residual}\\
            Let $(F, f, g, L)$ satisfy Assumption \ref{ass:for-inxt-pg-ineq}, $\epsilon \ge 0$.
            Let $\tilde x\approx_\epsilon T_\rho(x)$.
            Then:
            \begin{align*}
                \Vert x - \tilde x\Vert \ge (L + \rho)^{-1}\dist(\mathbf 0 | \partial_\epsilon F(\tilde x)). 
            \end{align*}
        \end{lemma}
        \begin{proof}
            Consider any $x \in \dom F, \epsilon\ge 0, \rho > 0$. 
            Let $\tilde x \approx_\epsilon T_\rho(x)$ (Definition \ref{def:inxt-pg}) so by definition it has: 
            \begin{align*}
                & \rho(x - \tilde x) - \nabla f(x) \in \partial_\epsilon g(\tilde x)
                \\
                \iff & 
                \rho(x - \tilde x) - \nabla f(x) + \nabla f(\tilde x) 
                \in \partial_\epsilon g(\tilde x) + \nabla f(\tilde x) \underset{(1)}\subseteq \partial_\epsilon F(\tilde x). 
            \end{align*}
            At (1), we applied Lemma \ref{lemma:eps-subgrad-wsum-rule}.
            Therefore:
            \begin{align*}
                \dist(\mathbf 0 | \partial_\epsilon F(\tilde x)) 
                &\le \Vert \rho(x - \tilde x) - \nabla f(x) + \nabla f(\tilde x)\Vert
                \\
                &\underset{(2)}\le \rho \Vert x - \tilde x\Vert + \Vert \nabla f(x) + \nabla f(\tilde x)\Vert 
                \\
                &\le (L + \rho)\Vert x - \tilde x\Vert. 
            \end{align*}
            At (2), we invoked Fact \ref{fact:lipz-smth-equiv}, which states $\nabla f$ is $L$-Lipschitz continuous, giving $\Vert \nabla f(x) - \nabla f(\tilde x)\Vert \le L \Vert x - \tilde x\Vert$.
        \end{proof}
        \par
        One of our main results of this section now follows. 
        The theorem below is an inexact variant of the proximal gradient inequality accommodating relative error, absolute error, and dynamic line search and backtracking. 
        By introducing a new relaxation parameter $\rho$, we accommodate the inexactness of the $\epsilon$-subgradient relative to $\Vert \tilde x - x\Vert^2$, where $\tilde x \approx_\epsilon T_{B+\rho}(x)$, and $B$ is the line search constant. 
        \begin{theorem}[inexact over-regularized proximal gradient inequality]\;\label{thm:inxt-pg-ineq}\\
            Let $(F, f, g, L)$ satisfy Assumption \ref{ass:for-inxt-pg-ineq} and denote $F = f + g$.
            Let $\approx_\epsilon T_\rho$ be given by Definition \ref{def:inxt-pg}.
            For all $\epsilon \ge 0, B \ge 0, \rho \ge 0$, consider any $\tilde x \approx_\epsilon T_{B + \rho}(x)$ such that
            $\tilde x, B$ satisfy the line search condition $D_f(\tilde x, x) \le \frac{B}{2}\Vert x - \tilde x\Vert^2$ ($D_f$ is given by Definition \ref{def:breg-div}).
            Then $\forall z \in \RR^n$: 
            \begin{align*}
                - \epsilon &\le 
                F(z) - F(\tilde x)
                + \frac{B + \rho}{2}\Vert x - z\Vert^2
                - \frac{B + \rho}{2}\Vert z - \tilde x\Vert^2
                - \frac{\rho}{2}\Vert \tilde x - x\Vert^2. 
            \end{align*}
        \end{theorem}
        \begin{proof}
            By Definition \ref{def:inxt-pg} write the variational inequality that describes $\tilde x \approx_\epsilon T_{B + \rho}(x)$ which is $\mathbf 0 \in \nabla f(x) - (B + \rho)(x - \tilde x) + \partial_{\epsilon} g(\tilde x)$.
            Applying Definition \ref{def:esp-subgrad}, so for all $z \in \RR^n$:
            \begin{align*}
                - \epsilon &\le 
                g(z) - g(\tilde x) - \langle - (B + \rho)(\tilde x - x) - \nabla f(x), z - \tilde x\rangle
                \\
                &= 
                g(z) - g(\tilde x) 
                + (B + \rho)\langle \tilde x - x, z - \tilde x\rangle
                + \langle \nabla f(x), z - \tilde x\rangle
                \\
                &\underset{(1)}{=} 
                g(z) + f(z) - g(\tilde x) - f(\tilde x)
                + (B + \rho)\langle \tilde x - x, z - \tilde x\rangle
                - D_f(z, x) + D_f(\tilde x, x)
                \\
                &\underset{(2)}{\le} 
                F(z) - F(\tilde x)
                + (B + \rho)\langle \tilde x - x, z - \tilde x\rangle
                + \frac{B}{2}\Vert \tilde x - x\Vert^2
                \\
                &=
                F(z) - F(\tilde x) + \frac{B + \rho}{2}\left(
                    \Vert x - z\Vert^2
                    - \Vert \tilde x - x\Vert^2
                    - \Vert z - \tilde x\Vert^2
                \right)
                + \frac{B}{2}\Vert \tilde x - x\Vert^2
                \\
                &= 
                F(z) - F(\tilde x)
                + \frac{B + \rho}{2}\Vert x - z\Vert^2
                - \frac{B + \rho}{2}\Vert z - \tilde x\Vert^2
                - \frac{\rho}{2}\Vert \tilde x - x\Vert^2. 
            \end{align*}
            At (1), we used the following: 
            \begin{align*}
                \langle \nabla f(x), z - \tilde x\rangle &= \langle \nabla f(x), z - x + x - \tilde x\rangle
                \\
                &= \langle \nabla f(x), z - x\rangle + \langle \nabla f(x), x - \tilde x\rangle
                \\
                &= -D_f(z, x) + f(z) - f(x) + D_f(\tilde x, x) - f(\tilde x) + f(x)
                \\
                &= -D_f(z, x) + f(z) + D_f(\tilde x, x) - f(\tilde x). 
            \end{align*}
            At (2), we used the fact that $f$ is convex hence $- D_f(z, x) \le 0$ always, and in the statement hypothesis we assumed that $B$ has $D_f(\tilde x, x) \le \frac{B}{2}\Vert \tilde x - x\Vert^2$.
            We also used $F = f + g$. 
        \end{proof}
        \begin{remark}
            When $\epsilon = 0, \rho = 0$, this reduces to the proximal gradient inequality exactly.
            The total perturbation admitted by the inequality is $\epsilon + \frac{\rho}{2}\Vert \tilde x - x\Vert^2$, decomposing into an absolute component $\epsilon$ and a relative component $\frac{\rho}{2}\Vert \tilde x - x\Vert^2$, where $\Vert \tilde x - x\Vert$ is a quantity that is large when $x$ is far from stationarity and vanishes at a fixed point of $\approx_\epsilon T_{B + \rho}$.
            This mixed error criterion automatically grants more tolerance for inexactness when $x$ is far from a stationary point, enabling faster convergence of the outer loop.
            \par
            The inequality differs from Schmidt et al.\ \cite[Lemma 2]{schmidt_convergence_2011} in that the gradient evaluation $\nabla f$ is exact and there is the additional over-relaxation parameter $\rho$.
            Compared to Villa et al.\ \cite{villa_accelerated_2013}, no equivalent result appears in their work, as they prefer Nesterov's estimating sequence, a preference we do not adopt.
        \end{remark}
        The following corollary is central to the convergence analysis of the inner loop of IAPG. 
        \begin{corollary}[the exact proximal gradient inequality]\label{corollary:exact-pgineq}
            Let $(F, f, g, L)$ satisfy Assumption \ref{ass:for-inxt-pg-ineq} and denote $F = f + g$.
            Let $\tau > 0$, $T_\tau$ be given by Definition \ref{def:exact-pg}.
            Consider any $x^+ = T_{\tau}(x)$ such that $x^+, \tau$ satisfy the line search condition $D_f(x^+, x) \le \frac{\tau}{2}\Vert x - x^+\Vert^2$ ($D_f$ is given by Definition \ref{def:breg-div}).
            Then $\forall z \in \RR^n$: 
            \begin{align*}
                0&\le 
                F(z) - F(x^+)
                + \frac{\tau}{2}\Vert x - z\Vert^2
                - \frac{\tau}{2}\Vert z - x^+\Vert^2. 
            \end{align*}
        \end{corollary}
        \begin{remark}
            When $\tau \ge L$, the line search condition $D_f(x^+, x) \le \frac{\tau}{2}\Vert x - x^+\Vert^2$ holds trivially by $L$-Lipschitz smoothness (Definition \ref{def:lipz-smooth}) of $f$.
        \end{remark}
        The above corollary is a special case of Theorem \ref{thm:inxt-pg-ineq} where $\rho = \epsilon = 0$. 

    \subsection{Primal-dual formulation of the inexact proximal point problem}\label{ssec:optz-inxt-pp-problem}
        In this section we discuss the consequence of assuming $g$ in Assumption \ref{ass:for-inxt-pg-ineq} satisfies $g(x) = \omega(Ax)$ where $\omega$ is globally Lipschitz convex function with an available proximal operator.
        Under this assumption, we formulate a proximal point problem in \eqref{eqn:primal-pp} leading to the major result (Theorem \ref{thm:minimizing-dual-pp}) which states that any sequence minimizing the Fenchel Rockafellar dual of the proximal point problem also minimizes the primal. 
        \begin{assumption}[linear composite of convex nonsmooth function]\;\label{ass:for-inxt-prox}\\
            Let $m, n\in \N$. 
            Assume  $(g, \omega, A, K_\omega)$ satisfy the following. 
            \begin{enumerate}[nosep]
                \item $A\in \RR^{m \times n}$ is a matrix. 
                \item\label{ass:for-inxt-prox:item2} $\omega: \RR^m \rightarrow \RR$ is proper, closed, and convex with an exact proximal operator $\hprox_{\lambda\omega^\star}$ for all $\lambda > 0$, known conjugate $\omega^\star$, and $\dom \omega = \RR^m$. 
                \item\label{ass:for-inxt-prox:item3} $g(x):= \omega(Ax)$ satisfying the constraint qualification $\rng A \cap \reli\dom \omega \neq\emptyset$. 
                \item\label{ass:for-inxt-prox:item4} $\omega$ is globally $K_\omega$-Lipschitz continuous. 
            \end{enumerate}
        \end{assumption}
        \begin{remark}
            Assumption \ref{ass:for-inxt-prox}\ref{ass:for-inxt-prox:item4} is equivalent to $\dom \omega^\star$ being a bounded set. 
            The item is also equivalent to $\partial \omega$ having a bounded range (Lemma \ref{lemma:lipz-cnvx-fxn}), i.e.: $\sup_{x \in \RR^m}\max_{v \in \partial \omega(x)} \Vert v\Vert = K_\omega < \infty$. 
            Assumption \ref{ass:for-inxt-prox}\ref{ass:for-inxt-prox:item3} follows from \ref{ass:for-inxt-prox:item2} since $\dom g = \RR^n$. 
        \end{remark}
        Let $(g, \omega, A, K_\omega)$ be given by Assumption \ref{ass:for-inxt-prox}. 
        Fix $y \in \RR^n, \lambda > 0$, to choose $\tilde x$ such that $\tilde x \approx_\epsilon \hprox_{\lambda g}(x)$ we first quantify the function inside the proximal operator: 
        \begin{align}\label{eqn:primal-pp}
            \Phi_\lambda(u) &:= \omega(Au) + \frac{1}{2\lambda} \Vert u - y\Vert^2. 
        \end{align}
        Observe that $\rng A \cap \reli\dom \omega \neq\emptyset$ since Assumption \ref{ass:for-inxt-prox} requires $\omega$ full domain.
        Therefore, we can use subgradient calculus for a minimizer $\bar u$ of $\Phi_\lambda$, which has: 
        \begin{align}\label{expr:primal-pp-sb}
            \mathbf 0 \in \partial \Phi_\lambda(\bar u) \iff 
            \lambda^{-1}(\bar u - y) + A^\top \partial \omega(A\bar u). 
        \end{align}
        The function $\Phi_\lambda$ is $\lambda^{-1}$-strongly convex due to its quadratic term and hence it must admit a unique minimizer.
        A well known result in the convex programming literature now follows. 
        \begin{fact}[Fenchel Rockafellar Duality {\cite[Proposition 15.22]{bauschke_convex_2017}}]\;\label{fact:fn-rck-duality}\\
            Let $f:\RR^n \rightarrow \overline \RR$, $g: \RR^m \rightarrow \overline \RR$ be closed convex and proper, $A \in \RR^{m \times n}$. 
            If ${\mathbf 0 \in \inte(\dom g - A \dom f)}$, then 
            \begin{align*}
                \inf_{u\in \RR^n}\left\lbrace
                    f(u) + g(Au)
                \right\rbrace 
                + \min_{v \in \RR^m}\left\lbrace
                    f^\star\circ(-A^\top)(v) + g^\star(v)
                \right\rbrace = 0. 
            \end{align*}
        \end{fact}
        \begin{remark}
            The theorem is not exactly the same as what is claimed in the original text by Bauschke and Combettes, because we are in a finite dimensional setting.
            To adapt the original theorem to finite dimension, we set $\mathcal H = \RR^n$ and used \cite[Proposition 6.12]{bauschke_convex_2017}. 
        \end{remark}
        Here, we are interested in the dual of the proximal problem written in the form $\Phi_\lambda = f + g\circ A$ where $f = u \mapsto \frac{1}{2\lambda}\Vert u - y\Vert^2$, and $g = \omega$.
        It has $f^\star(v) = \frac{1}{2\lambda}\Vert \lambda v + y\Vert^2 - \frac{1}{2\lambda}\Vert y\Vert^2$ (see Appendix \ref{lemma:chore1}). 
        Consequently, $f^\star \circ (- A^\top) = v\mapsto \frac{1}{2\lambda}\Vert -\lambda A^\top v  + y\Vert^2 - \frac{1}{2\lambda}\Vert y\Vert^2$. 
        And therefore by Fact \ref{fact:fn-rck-duality}, $\Phi_\lambda$ admits Fenchel Rockafellar dual (or simply the dual) in $\RR^m$: 
        \begin{align}\label{eqn:dual-pp}
            \Psi_\lambda(v) &:= f^\star\circ(-A^\top)(v) + g^\star(v) 
            = \frac{1}{2\lambda}\Vert \lambda A^\top v - y\Vert^2
            + \omega^\star(v) - \frac{1}{2\lambda}\Vert y\Vert^2. 
        \end{align}
        We define the duality gap
        \begin{align}\label{eqn:duality-gap-pp}
            \mathbf G_\lambda(u, v) &:= \Phi_\lambda(u) + \Psi_\lambda(v). 
        \end{align}
        Note that in this case the smooth part is quadratic and $\dom f = \RR^n$. 
        It follows that ${\mathbf 0 \in \inte(\dom g - A \dom f) = \inte(\dom g - \rng A)}$. 
        It holds because of $\rng A \cap \reli\dom g \neq\emptyset$ in Assumption \ref{ass:for-inxt-prox}.
        Therefore, strong duality holds and there exists $(\hat u, \hat v)$ such that $\mathbf G_\lambda(\hat u, \hat v) = 0 = \min_{u} \Phi_\lambda(u) + \min_v \Psi_\lambda(v)$. 
        \par
        The following result, taken from Villa et al. \cite{villa_accelerated_2013}, gives a sufficient condition for $\tilde x \approx_\epsilon \hprox_{\lambda g}(x)$.
        \begin{fact}[primal translate to dual {\cite[Proposition 2.2]{villa_accelerated_2013}}]\label{fact:primal-dual-trans}
            Let $(g, \omega, A)$ satisfy Assumption \ref{ass:for-inxt-prox}, $\epsilon \ge 0$, then 
            \begin{align*}
                \left(
                    \forall z \approx_\epsilon \hprox_{\lambda g}(y) 
                \right)(\exists v \in \dom \omega^\star): z = y - \lambda A^\top v. 
            \end{align*}
        \end{fact}
        \begin{lemma}[duality gap of inexact proximal problem {\cite[Proposition 2.3]{villa_accelerated_2013}}]\;\label{lemma:dlty-gap-inxt-pp}\\
            Let $(g, \omega, A)$ satisfy Assumption \ref{ass:for-inxt-prox}, for all $\epsilon \ge 0$, $v \in \RR^n$ consider the following conditions: 
            \begin{enumerate}[nosep]
                \item $\mathbf G_\lambda(y - \lambda A^\top v, v) \le \epsilon$. 
                \item $A^\top v \approx_\epsilon \hprox_{\lambda^{-1}g^\star}(\lambda^{-1}y)$. 
                \item $y - \lambda A^\top v \approx_{\epsilon} \hprox_{\lambda g}(y)$. 
            \end{enumerate}
            They have (i) $\implies$ (ii) $\iff$ (iii). 
            If in addition $\omega^\star(v) = g^\star\left(A^\top v\right)$, then all three conditions are equivalent. 
        \end{lemma}
        \begin{proof}
            We refer readers tof Villa et al. \cite[ Proposition 2.3]{villa_accelerated_2013} for the proof of (i) $\implies$ (iii), and the case when (i) $\iff$ (ii). 
            To show (ii) $\iff$ (iii) use Lemma \ref{lemma:inxt-moreau-decomp}. 
        \end{proof}
        \par
        The following theorem is enhanced from Villa et al. \cite[Theorem 5.1]{villa_accelerated_2013} and, it is our first major result. 
        It states that any sequence $v_j$ minimizing $\Psi_\lambda$ also minimizes the primal optimality gap. 
        This is crucial for showing the convergence results of the inner loop later on. 
        \begin{theorem}[minimizing the dual of the proximal problem]\;\label{thm:minimizing-dual-pp}\\
            Assume that we have $(g, \omega, A)$ given by Assumption \ref{ass:for-inxt-prox}. 
            Let the $\Phi_\lambda$ be given by \eqref{eqn:primal-pp}, and dual $\Psi_\lambda$ by \eqref{eqn:dual-pp}. 
            Let $\bar v$ be a minimizer of $\Psi_\lambda$. 
            Suppose that sequence $(v_j)_{j \in \Z_+}$ minimizes the dual $\Psi_\lambda$, i.e., $\lim_{j \rightarrow \infty} \Psi_\lambda(v_j) = \Psi_\lambda(\bar v)$. 
            Let $z_j = y - \lambda A^\top v_j$ for all $j \in \Z_+$. 
            Then, the following hold:
            \begin{enumerate}[nosep]
                \item\label{thm:minimizing-dual-pp:result1} If $\bar v$ is a minimizer of dual $\Psi_\lambda$, then $\bar z = y - \lambda A^\top \bar v$ is a minimizer of primal $\Phi_\lambda$. 
                \item\label{thm:minimizing-dual-pp:result2} It has $\Psi_\lambda(v_j) - \Psi_\lambda(\bar v) \ge \frac{1}{2 \lambda} \Vert z_j - \bar z\Vert^2$, and consequently $z_j \rightarrow \bar z$. 
                \item\label{thm:minimizing-dual-pp:result3} The primal optimality gap is bounded by dual by: 
                \begin{align*}
                    & \Phi_\lambda(z_j) - \Phi_\lambda(\bar z) 
                    \\
                    &\le 
                    \sqrt{\Psi_\lambda(v_j) - \Psi_\lambda(\bar v)}
                    \left(
                        2\sqrt{2\lambda}K_\omega\Vert A\Vert
                        + \sqrt{\Psi_\lambda(v_j) - \Psi_\lambda(\bar v)}
                    \right).
                \end{align*}
            \end{enumerate}
        \end{theorem}
        \begin{proof}    
            In preparations, we establish the following two intermediate results for the proof.  
            For all $v \in \RR^m$, it has the following identity holds:
            \begin{align}\label{thm:minimizing-dual-pp:pitem1}
                \begin{split}
                    & \frac{1}{2\lambda}\left\Vert \lambda A^\top v - y\right\Vert^2
                    - \frac{1}{2\lambda}\left\Vert \lambda A^\top \bar v - y\right\Vert^2 
                    + \langle A \bar z, v - \bar v\rangle 
                    \\
                    &= 
                    \frac{1}{2\lambda}
                    \left\Vert 
                        \lambda A^\top v - \lambda A^\top \bar v  + \lambda A^\top \bar v - y
                    \right\Vert^2
                    - \frac{1}{2\lambda}\left\Vert \lambda A^\top \bar v - y\right\Vert^2 
                    + \langle A \bar z, v - \bar v\rangle 
                    \\
                    &= 
                    \frac{1}{2\lambda} \left\Vert \lambda A^\top(v - \bar v)\right\Vert^2
                    + \frac{1}{\lambda}\left\langle \lambda A^\top(v - \bar v), \lambda A^\top \bar v - y \right\rangle
                    + \left\langle \bar z, A^\top(v - \bar v)\right\rangle 
                    \\
                    &\underset{\text{(1)}}{=} 
                    \frac{1}{2\lambda} \left\Vert \lambda A^\top(v - \bar v)\right\Vert^2
                    - \frac{1}{\lambda}\left\langle \lambda A^\top(v - \bar v), \bar z \right\rangle
                    + \left\langle \bar z, A^\top(v - \bar v)\right\rangle 
                    \\
                    &= \frac{1}{2\lambda} \left\Vert \lambda A^\top(v - \bar v)\right\Vert^2. 
                \end{split}
            \end{align}
            At (1), we substituted $\bar z = y - \lambda A^\top \bar v$. 
            To introduce our second intermediate result, consider that $\bar v$ is the minimizer on dual problem $\Psi_\lambda$. 
            Then, by Fenchel subgradient calculus, and definition of $\Psi_\lambda$ in \eqref{eqn:dual-pp}, we have the following sequence of equivalences: 
            \begin{align}\label{thm:minimizing-dual-pp:pitem2}
                \begin{split}
                    & \mathbf 0 \in \partial \Psi_\lambda(\bar v)
                    \\
                    \iff & \mathbf 0 \in 
                    A \left(\lambda A^\top \bar v - y\right) + \partial \omega^\star(\bar v)
                    \\
                    \iff &
                    A\bar z \in \partial \omega^\star(\bar v)
                    \\
                    \iff &
                    (\forall v \in \RR^m)\; \omega^\star(v) - \omega^\star(\bar v) 
                    \ge \langle A \bar z, v - \bar v\rangle. 
                \end{split}
            \end{align}
            \par
            We are now ready to prove \ref{thm:minimizing-dual-pp:result1}.
            From \eqref{thm:minimizing-dual-pp:pitem2}, by Fenchel identity that: 
            $$
                A \bar z \in \partial \omega^\star (\bar v) \iff \bar v \in \partial \omega(A \bar z).
            $$
            Multiplying $\lambda A^\top$ on both sides of $\partial \omega(A \bar z) \ni \bar v$ yields: 
            $$
                y - \bar z = \lambda A^\top \bar v \in \lambda A^\top\partial \omega (A \bar z). 
            $$
            Recall the optimality condition of $\Phi_\lambda$ from \eqref{expr:primal-pp-sb}. 
            With that in mind, re-arranging the above yields: $\mathbf 0 \in \bar z - y + \lambda A^\top \partial \omega (A \bar z) = \lambda \partial \Phi_\lambda(\bar z)$. 
            Therefore, by Fenchel subgradient calculus $\bar z = y - \lambda A^\top \bar v$ is a minimizer of $\Phi_\lambda$. 
            \par
            We are now prepared to prove \ref{thm:minimizing-dual-pp:result2}. 
            The definition of $\Psi_\lambda$ in \eqref{eqn:dual-pp} shows: 
            \begin{align*}
                \Psi_\lambda(v_j) - \Psi_\lambda(\bar v)
                &= \frac{1}{2\lambda}\left\Vert
                    \lambda A^\top v_j - y
                \right\Vert^2 
                - \frac{1}{2\lambda}\left\Vert
                    \lambda A^\top \bar v - y
                \right\Vert^2
                + \omega^\star(v_j) - \omega^\star(\bar v)
                \\
                &\underset{(2)}{\ge}
                \frac{1}{2\lambda}\left\Vert
                    \lambda A^\top v_j - y
                \right\Vert^2 
                - \frac{1}{2\lambda}\left\Vert
                    \lambda A^\top \bar v - y
                \right\Vert^2
                + \langle A\bar z, v_j - \bar v\rangle
                \\
                &\underset{(3)}{=} \frac{1}{2\lambda}\Vert \lambda A^\top(v_j - \bar v)\Vert^2
                \\
                &\underset{(4)}{=} \frac{1}{2\lambda}\Vert z_j - \bar z\Vert^2. 
            \end{align*}
            At (2) we applied $(\forall v \in \RR^m)\; \omega^\star(v) - \omega^\star(\bar v) \ge \langle A \bar z, v - \bar v\rangle$ from \eqref{thm:minimizing-dual-pp:pitem2}.
            At (3) we used the result from \eqref{thm:minimizing-dual-pp:pitem1}. 
            At (4), we substituted $z_j - \bar z = y - \lambda A^\top v_j - \left(y - \lambda A^\top\bar v\right) = \lambda A^\top(\bar v - v_j)$. 
            We assumed that $(v_j)_{j \in \Z_+}$ is a minimizing sequence of $\Psi_\lambda$, therefore the above result we derived implies: 
            \begin{align*}
                0 = \lim_{j \rightarrow \infty} \Psi_\lambda(v_j) - \Psi_{\lambda} (\bar v) \ge \lim_{j \rightarrow \infty}\Vert z_j - \bar z \Vert^2. 
            \end{align*}
            \par
            We now have everything we need to prove \ref{thm:minimizing-dual-pp:result3}. 
            Recall from Assumption \ref{ass:for-inxt-prox} the function $\omega$ is $K_\omega$-Lipschitz continuous.
            This fact will be useful throughout the derivations that follow.
            By definition of $\Phi_\lambda$ from \eqref{eqn:primal-pp}, for all $j \in \Z_+$: 
            {\allowdisplaybreaks
            \begin{align*}
                & \Phi_\lambda(z_j) - \Phi_\lambda(\bar z) 
                \\
                &= \omega(Az_j) - \omega(A \bar z)
                + \frac{1}{2\lambda}(\Vert z_j - y \Vert^2 - \Vert \bar z - y\Vert^2)
                \\
                &\le 
                K_\omega\Vert A\Vert \Vert z_j - \bar z\Vert 
                + \frac{1}{2\lambda}\left(
                    \Vert z_j - y\Vert + \Vert \bar z - y\Vert
                \right)\left(
                    \Vert z_j - y\Vert - \Vert \bar z - y\Vert
                \right)
                \\
                &\le
                K_\omega\Vert A\Vert \Vert z_j - \bar z\Vert 
                + \frac{1}{2\lambda}\left(
                    \Vert z_j - y\Vert + \Vert \bar z - y\Vert
                \right)\Vert z_j - \bar z\Vert
                \\
                &\le 
                K_\omega\Vert A\Vert \Vert z_j - \bar z\Vert 
                + \frac{1}{2\lambda}\left(
                    \Vert z_j - \bar z\Vert + 2\Vert \bar z - y\Vert
                \right)\Vert z_j - \bar z\Vert
                \\
                &= 
                \Vert z_j - \bar z\Vert
                \left(
                    K_\omega\Vert A\Vert
                    + \lambda^{-1}\Vert \bar z - y\Vert
                    + \frac{\Vert z_j - \bar z\Vert}{2\lambda}
                \right)
                \\ &\underset{\text{\ref{thm:minimizing-dual-pp:result2}}}\le 
                \sqrt{2\lambda\left(\Psi_\lambda(v_j) - \Psi_\lambda(\bar v)\right)}
                \left(
                    K_\omega\Vert A\Vert
                    + \lambda^{-1}\Vert \bar z - y\Vert
                    + \frac{\sqrt{2\lambda}}{2\lambda}\sqrt{\Psi_\lambda(v_j) - \Psi_\lambda(\bar v)}
                \right)
                \\
                &\underset{(5)}= 
                \sqrt{2\lambda\left(\Psi_\lambda(v_j) - \Psi_\lambda(\bar v)\right)}
                \left(
                    K_\omega \Vert A\Vert
                    + K_\omega \Vert A \Vert
                    + \frac{\sqrt{2\lambda}}{2\lambda}\sqrt{\Psi_\lambda(v_j) - \Psi_\lambda(\bar v)}
                \right)
                \\ 
                &= 
                \sqrt{\Psi_\lambda(v_j) - \Psi_\lambda(\bar v)}
                \left(
                    2\sqrt{2\lambda}K_\omega\Vert A\Vert
                    + \sqrt{\Psi_\lambda(v_j) - \Psi_\lambda(\bar v)}
                \right).
            \end{align*}
            }\noindent
            The first three chains of inequalities used the triangle inequality. 
            At (5), we used the fact that $\bar z$ is the minimizer of $\Phi_\lambda$. 
            Then, from \eqref{expr:primal-pp-sb} it has ${\mathbf 0 \in \partial \Phi_\lambda(\bar z) \iff \lambda^{-1}(y - \bar z) \in \partial (\omega\circ A)(\bar z)}$ which implies that $\lambda^{-1}\Vert y - \bar z \Vert \le \sup_{v \in \partial (\omega\circ A)(\bar z)}\Vert v\Vert$. 
            Then, we used the assumption that $\omega$ is $K_\omega$-Lipschitz continuous (Assumption \ref{ass:for-inxt-prox}): 
            \begin{align*}
                (\forall u_1, u_2 \in \RR^n)\; |\omega(A u_1) - \omega(A u_2)| 
                \le K_\omega\Vert Au_1 - Au_2 \Vert \le K_\omega \Vert A\Vert\Vert u_1 - u_2\Vert. 
            \end{align*}
            Therefore, $\omega \circ A$ is Lipschitz continuous with constant $K_\omega \Vert A\Vert$, combining the above with results from Appendix \ref{lemma:lipz-cnvx-fxn} produces: 
            \begin{align*}
                \lambda^{-1}\Vert \bar z - y\Vert \le \sup_{v \in \partial (\omega\circ A)(\bar z)}\Vert v\Vert \le K_\omega \left\Vert A\right\Vert. 
            \end{align*}
        \end{proof}
        \begin{remark}
            There are multiple ways to bound $\Vert \bar z - y\Vert$; the approach taken here integrates most naturally with the overall complexity analysis.
        \end{remark}

\section{Convergence, complexity of IAPG outer loop with line search}\label{sec:outerloop}

    This section derives the convergence rate of the outer loop.
    To start, Lemma \ref{lemma:inxt-apg-cnvg-prep1} establishes an essential inequality used throughout this section, and Definition \ref{def:inxt-apg} defines the algorithm of the outer loop.
    It is organized into four sections. 
    The first three subsections will prepare for IAPG outer loop convergence and the final subsection will present the IAPG convergence results and iteration complexity, covering the optimality gap, stationarity, and a termination criterion implying stationarity.
    \par
    Section \ref{ssec:rslt-valid-err} states the convergence rate of the outer loop under the weakest assumptions (Assumption \ref{ass:valid-err-schedule}) on the momentum sequence $(\alpha_k)_{k \in \Z_+}$, and the error sequence $(\epsilon_k)_{k \in \Z_+}$ which yields an upper bound on the optimality gap.
    These results underpin everything in the next three subsections.
    Following that, Section \ref{ssec:opt-momentum-seq} strengthens the assumptions of the error sequence and momentum sequence, forming the bedrock to derive the $\mathcal O(1/k^2)$ convergence rate of the IAPG outer loop.
    Section \ref{ssec:outerloop-err-shrink} addresses a remaining gap that is imperative for the analysis of the total complexity of IAPG.
    It establishes the fastest admissible rate of decay of $\epsilon_k$ to zero.
    This is vital for characterizing the total complexity of the algorithm in later sections because it links the iteration complexity of the IAPG outer loop with its inner loop.
    \par
    Finally, Section \ref{ssec:outerloop-cnvg} presents the major results. 
    It will show that if a minimizer exists for the objective function, then the function value converges to the minimum at a rate of $\mathcal O(1/k^2)$. 
    It also presents a termination criterion implying stationarity which converges at a rate of $\mathcal O (1/k)$.
    \begin{definition}[our inexact accelerated proximal gradient]\;\label{def:inxt-apg}\\
        Suppose that $(F, f, g, L)$ and sequences $(\alpha_k, B_k, \rho_k, \epsilon_k)_{k \in \Z_+}$ satisfy the following
        \begin{enumerate}[nosep]
            \item $(\alpha_k)_{k \in \Z_+}$ is a sequence such that $\alpha_k \in (0, 1]$ for all $k \in \Z_+$. 
            \item $(B_k)_{k \in \Z_+}$ has $B_k > 0\; \forall k\in \Z_+$, and it characterizes any potential line search, backtracking routine.
            \item $(\rho_k)_{k \in \Z_+}$ is a sequence such that $\rho_k \ge 0$, characterizing the over-relaxation of the proximal gradient operator. 
            \item $(\epsilon_k)_{k \in \Z_+}$ has $\epsilon_k > 0$ for all $k \in \Z_+$, and it characterizes the errors of inexact proximal evaluation.
            \item $(F, f, g, L)$ satisfy Assumption \ref{ass:for-inxt-pg-ineq}.
        \end{enumerate}
        Denote $L_k = B_k + \rho_k$ for short. 
        Let the inexact proximal gradient operator $\approx_\epsilon T_{L_k}$ be given by Definition \ref{def:inxt-pg}. 
        Given any initial condition $x_{-1}^\circ, x_{-1} \in \RR^n$, the algorithm generates the sequences $(y_k, x_k, x_k^\circ)_{k \in \Z_+}$ satisfying for all $k \in \Z_+$:
        \begin{align}
            & y_k = \alpha_k x_{k - 1}^\circ + (1 - \alpha_k)x_{k - 1}, \label{def:inxt-apg:yk}
            \\
            & x_k \approx_{\epsilon_k} T_{L_k}(y_k), \label{def:inxt-apg:xk}
            \\
            & D_f(x_k, y_k) \le \frac{B_k}{2}\Vert x_k - y_k\Vert^2, \label{def:inxt-apg:ls}
            \\
            & x_k^\circ = x_{k - 1} + \alpha_k^{-1}(x_k - x_{k - 1}). \label{def:inxt-apg:vk}
        \end{align}
    \end{definition}
    \begin{remark}
        The sequence $(B_k)_{k \ge 0}$ accomodates dynamic line search routines.
        For example, it can accommodate Calatroni and Chambolle's backtracking technique \cite{calatroni_backtracking_2019}.
    \end{remark}
    The following lemma is stated on its own to simplify the convergence proof later on in the section. 
    \begin{lemma}[APG convergence preparation]\; \label{lemma:inxt-apg-cnvg-prep1}\\
        Let $(F, f, g, L)$, $(y_k, x_k, x_k^\circ)_{k \in \Z_+}$, and $(\alpha_k, B_k, \rho_k, \epsilon_k)_{k \in \Z_+}$ be given by Definition \ref{def:inxt-apg}. 
        Denote $L_k := B_k + \rho_k$.
        Then, for any $\bar x \in \RR^n$ and initial guesses $x_{-1}, x_{-1}^\circ \in \RR^n$, the sequences satisfy for all $k \in \Z_+$ the inequality: 
        \begin{align*}
            & \frac{\rho_k}{2}\Vert x_k - y_k\Vert^2 - \epsilon_k
            \\
            &\le 
            (1 - \alpha_k)(F(x_{k - 1}) - F(\bar x)) + F(\bar x) - F(x_k) 
            + \frac{\alpha_k^2L_k}{2}\Vert \bar x - x_{k - 1}^\circ\Vert^2 
            - \frac{\alpha_k^2L_k}{2}\Vert \bar x - x_k^\circ\Vert^2.
        \end{align*}
    \end{lemma}
    \begin{proof}
        Two intermediate results are in order before we can prove the inequality. 
        Define $(\forall k \in \Z_+)\; \hat x_k := \alpha_k \bar x + (1 - \alpha_k)x_{k - 1}$. 
        The following equality holds for all $k \in \Z_+$:
        \begin{align}\label{eqn:inxt-apg-cnvg-prep1-a}\begin{split}
            \hat x_k - x_k &= 
            \alpha_k \bar x + (1 - \alpha_k)x_{k - 1} - x_k
            \\
            &= \alpha_k\bar x + (x_{k - 1} - x_k) - \alpha_kx_{k - 1}
            \\
            &\hspace{-0.35em}\underset{\text{\eqref{def:inxt-apg:vk}}}=\hspace{-0.3em} 
            \alpha_k \bar x - \alpha_k x_k^\circ. 
        \end{split}\end{align}
        The following equality also holds: 
        \begin{align}\label{eqn:inxt-apg-cnvg-prep1-b}\begin{split}
            \hat x_k - y_k &= 
            \alpha_k \bar x + (1 - \alpha_k)x_{k - 1} - y_k
            \\
            &\hspace{-0.35em}\underset{\text{\eqref{def:inxt-apg:yk}}}=\hspace{-0.3em} 
            \alpha_k \bar x - \alpha_k x_{k - 1}^\circ. 
        \end{split}\end{align}
        Recall that $L_k = B_k + \rho_k$. 
        Since $(f, g, L)$ satisfy Assumption \ref{ass:for-inxt-pg-ineq}, choosing $x = y_k$, $\tilde x = x_k \approx_{\epsilon_k} T_{L_k}(y_k)$, and $z = \hat x_k, \epsilon = \epsilon_k$, Theorem \ref{thm:inxt-pg-ineq} gives for all: 
        \begin{align*}
            & \frac{\rho_k}{2}\Vert x_k - y_k\Vert^2 - \epsilon_k
            \\
            &\le
            F(\hat x_k) - F(x_k) + \frac{L_k}{2}\Vert y_k - \hat x_k\Vert^2 - \frac{L_k}{2}\Vert \hat x_k - x_k\Vert^2
            \\
            &\underset{(1)}{\le} \alpha_k F(\bar x) + (1 - \alpha_k)F(x_{k - 1}) - F(x_k)
            + \frac{L_k}{2}\Vert y_k - \hat x_k\Vert^2 - \frac{L_k}{2}\Vert \hat x_k - x_k\Vert^2
            \\
            &\underset{(2)}{=} 
            (1 - \alpha_k)(F(x_{k - 1}) - F(\bar x)) + F(\bar x) - F(x_k) 
            + \frac{\alpha_k^2L_k}{2}\Vert \bar x - x_{k - 1}^\circ\Vert^2 
            - \frac{\alpha_k^2L_k}{2}\Vert \bar x - x_k^\circ\Vert^2.
        \end{align*}
        At (1) we used the fact that $F = f + g$, and hence $F$ is convex (Jensen's inequality).
        At (2) we used \eqref{eqn:inxt-apg-cnvg-prep1-a}, \eqref{eqn:inxt-apg-cnvg-prep1-b}. 
    \end{proof}
    \subsection{Results under a valid error schedule}\label{ssec:rslt-valid-err}
        This section establishes the groundwork for the convergence rate of the outer loop of IAPG, and it derives two intermediate results based on the weakest possible assumption (Assumption \ref{ass:valid-err-schedule}) on parameters in Definition \ref{def:inxt-apg} such that an upper bound exists (Proposition \ref{prop:inxt-apg-cnvg-generic}) for the optimality gap $F(x_k) - F(\bar x)$, and the termination criterion $\Vert x_k - y_k\Vert$ (Proposition \ref{prop:vk-gm}).
        \begin{assumption}[valid error schedule]\;\label{ass:valid-err-schedule}\\
            Let $(F, f, g, L)$, $(\alpha_k, B_k, \rho_k, \epsilon_k)_{k \in \Z_+}$, $(y_k, x_k, x_k^\circ)_{k \in \Z_+}$ satisfy Definition \ref{def:inxt-apg}.
            Let $(L_k)_{k \in \Z_+}$ be defined as $L_k := \rho_k + B_k$ for all $k \in \Z_+$. 
            Define $(\beta_k)_{k \in \Z_+}$ such that $\beta_0 = 1$ and for all $k \in \N$: 
            \begin{align}\label{ass:valid-err-schedule:betak} 
                \beta_k := \prod_{i = 1}^{k} \max\left(1 - \alpha_i, \frac{\alpha_i^2L_i}{\alpha_{i - 1}^2L_{i - 1}}\right). 
            \end{align}
            Fix the constants $\mathcal E_0 > 0, p > 0$. 
            % Let constants $\mathcal E_0, p$, and $(\beta_k)_{k \in \Z_+}, (\mathcal R_k(p))_{k \in \Z_+}$ satisfy the following. 
            Define the sequence $(\mathcal R_k)_{k \in \Z_+}$ with base case $\mathcal R_0(p) = \mathcal E_0$, and for all $k \in \Z_+$: 
            \begin{align}\label{ass:valid-err-schedule:rkp}
                \mathcal R_k(p) :=\mathcal E_0\left(1 + \sum_{l = 1}^{k}\frac{1}{l^p}\right). 
            \end{align}
            % \item\label{ass:valid-err-schedule:item1} Parameters $(\alpha_k, B_k, \rho_k, \epsilon_k)_{k \in \Z_+}, (F, f, g, L)$, iterates $(y_k, x_k, x_k^\circ)_{k \in \Z_+}$, and line search constants $(L_k)_{k \in \Z_+}$ all satisfy Definition \ref{def:inxt-apg}. 
            Let $(\epsilon_k)_{k \in \Z_+}, (\rho_k)_{k \in \Z_+}$ satisfy the base case ${\epsilon_0 \le \frac{\rho_0}{2}\Vert x_0 - y_0\Vert^2 + \mathcal E_0}$. 
            Assume inductively that it holds:
            \begin{align}\label{ass:valid-err-schedule:item2}
                \left(
                    \forall k \in \N
                \right):\;
                \frac{- \mathcal E_0\beta_k}{k^p} \le \frac{\rho_k}{2}\Vert x_k - y_k\Vert^2 - \epsilon_k. 
            \end{align}
        \end{assumption}
        \par
        The following proposition establishes $\forall \bar x \in \RR^n, \forall k \in \Z_+,\; \forall \alpha_k\in [0, 1)$, an upper bound on $F(x_k) - F(\bar x)$ in terms of $(\beta_k)_{k \in \Z_+}$ and $(\mathcal R_k(p))_{k \in Z_+}$.
        \begin{proposition}[convergence with valid error schedule]\;\label{prop:inxt-apg-cnvg-generic}\\
            Let $(F, f, g, L)$, $(\alpha_k, B_k, \rho_k, \epsilon_k)_{k \in \Z_+}$, $(y_k, x_k, x_k^\circ)_{k \in \Z_+}$, $\mathcal E_0, p$, and $(\beta_k)_{k \in \Z_+}, (\mathcal R_k)_{k \in \Z_+}$ be given by Assumption \ref{ass:valid-err-schedule}. 
            Fix any $\bar x \in \RR^n$, assume that $\alpha_0 = 1$, and for all $k \in \N:\alpha_k \in (0, 1)$.
            Then, for any initial guesses $x_{-1}, x_{-1}^\circ\in \RR^n$, the iterates $(y_k, x_k, x_k^\circ)_{k \in \Z_+}$ generated by an algorithm satisfying Definition \ref{def:inxt-apg} satisfy for all $k \in \Z_+$: 
            \begin{align*}
                F(x_k) - F(\bar x) + \frac{\alpha_k^2L_k}{2}\Vert \bar x - x_k^\circ\Vert^2
                \le 
                \beta_k \left(
                    \frac{L_0}{2}\Vert \bar x - x_{-1}^\circ\Vert^2
                    + \mathcal R_k(p)
                \right). 
            \end{align*}
        \end{proposition}
        \begin{proof} 
            The proof consists of two parts.
            The first part verifies recursively that the inequality is true for $k \in \N$. 
            The second part verifies the inequality is true for $k = 0$. 
            Apply Lemma \ref{lemma:inxt-apg-cnvg-prep1} with $k \in \N$:
            \begin{align}\label{prop:inxt-apg-cnvg-generic:pitem0}\begin{split}
                & \frac{\rho_k}{2}\Vert x_k - y_k\Vert^2 - \epsilon_k
                \\
                &\le (1 - \alpha_k)(F(x_{k - 1}) - F(\bar x)) + F(\bar x) - F(x_k) 
                \\&\quad
                    + \frac{\alpha_k^2L_k}{2}\Vert \bar x - x_{k - 1}^\circ\Vert^2 
                    - \frac{\alpha_k^2L_k}{2}\Vert \bar x - x_k^\circ\Vert^2 
                \\
                &\le 
                (1 - \alpha_k)(F(x_{k - 1}) - F(\bar x)) + F(\bar x) - F(x_k) 
                \\&\quad
                + \max\left(
                        1 - \alpha_k, \frac{\alpha_k^2L_k}{\alpha_{k - 1}^2L_{k - 1}}
                    \right)\frac{\alpha_{k - 1}^2L_{k - 1}}{2}\Vert \bar x - x_{k - 1}^\circ\Vert^2 
                - \frac{\alpha_k^2L_k}{2}\Vert \bar x - x_k^\circ\Vert^2. 
                \\
                &\le \max\left(
                    1 - \alpha_k, \frac{\alpha_k^2L_k}{\alpha_{k - 1}^2L_{k - 1}}
                \right)\left(
                    F(x_{k - 1}) - F(\bar x)
                    + \frac{\alpha_{k - 1}^2L_{k - 1}}{2}\Vert \bar x - x_{k - 1}^\circ\Vert^2 
                \right)
                \\&\quad 
                    + F(\bar x) - F(x_k) - \frac{\alpha_k^2L_k}{2}\Vert \bar x - x_k^\circ\Vert^2.
            \end{split}\end{align}
            Recall that we introduced $\beta_k$ in \eqref{ass:valid-err-schedule:betak} which had $\beta_k := \prod_{i = 1}^{k} \max\left(1 - \alpha_i, \alpha_i^2L_i\alpha_{i - 1}^{-2}L_{i - 1}^{-1}\right)$ for all $k \in \N$, and $\beta_0 = 1$.
            To simplify the notation we denote
            \begin{align*}
                \Lambda_k:= - F(\bar x) + F(x_k) + \frac{\alpha_k^2L_k}{2}\Vert \bar x - x_k^\circ\Vert^2. 
            \end{align*}
            Therefore, write $k$ in \eqref{ass:valid-err-schedule:item2} as $l$, and apply it to \eqref{prop:inxt-apg-cnvg-generic:pitem0}: 
            \begin{align}\label{prop:inxt-apg-cnvg-generic:pitem1}\begin{split}
                (\forall l \in \N)\quad 
                - \frac{\mathcal E_0 \beta_l}{l^p}
                &\le
                \frac{\rho_k}{2}\Vert x_l - y_l\Vert^2 - \epsilon_l 
                \le \frac{\beta_l}{\beta_{l - 1}} \Lambda_{l - 1} - \Lambda_l
                \\
                \underset{(1)}\implies  
                0 &\le 
                \frac{\mathcal E_0}{l^p} + \beta_{l - 1}^{-1}\Lambda_{l - 1} - \beta_l^{-1} \Lambda_l. 
            \end{split}\end{align}
            Note that at (1) we moved $\frac{\mathcal E_0 \beta_l}{l^p}$ to the RHS.
            We also divided by $\beta_l$ for all $l \in \N$ which is permissible because our assumption this proposition states that $\alpha_l \in (0, 1)$ for all $l \in \N$, meaning $\beta_l > 0$.
            For all $k \in \N$, telescoping the series in \eqref{prop:inxt-apg-cnvg-generic:pitem1} for $l = 1, 2, \ldots, k$ yields: 
            \begin{align*}
                0 \le 
                \beta_0^{-1}\Lambda_0 
                - \beta_k^{-1} \Lambda_k
                + \sum_{l = 1}^{k} \frac{\mathcal E_0}{l^p}
                \iff 
                \Lambda_k \le 
                \beta_k\left(
                    \beta_0^{-1}\Lambda_0 + \sum_{l = 1}^{k} \frac{\mathcal E_0}{l^p}
                \right). 
            \end{align*}
            Since $\beta_0 = 1$ (defined in Assumption \ref{ass:valid-err-schedule}), the above expression gives:
            \begin{align}\label{prop:inxt-apg-cnvg-generic:pitem3}\begin{split}
                (\forall k \in \N)\quad 
                & F(x_k) - F(\bar x) + \frac{\alpha_k^2L_k}{2}\Vert \bar x - x_k^\circ\Vert^2
                \\
                &\le 
                \beta_k \left(
                    F(x_0) - F(\bar x) + \frac{\alpha_0^2L_0}{2}\Vert \bar x - x_0^\circ\Vert^2
                    + \mathcal E_0 \sum_{l = 1}^{k} \frac{1}{l^p}
                \right). 
            \end{split}\end{align}
            We can upper bound $F(x_0) - F(\bar x) + \frac{\alpha_0^2L_0}{2}\Vert \bar x - x_0^\circ\Vert^2$ by considering results in Lemma \ref{lemma:inxt-apg-cnvg-prep1} with $k = 0$ and $\alpha_0 = 1$: 
            \begin{align}\label{prop:inxt-apg-cnvg-generic:pitem4}\begin{split}
                - & \mathcal E_0 
                \underset{\text{(2)}}\le 
                \frac{\rho_0}{2}\Vert x_0 - y_0\Vert^2 - \epsilon_0
                \le 
                F(\bar x) - F(x_0)
                + \frac{L_0}{2}\Vert \bar x - x_{-1}^\circ\Vert^2 - \frac{L_0}{2}\Vert \bar x - x_0^\circ\Vert^2
                \\
                \implies & 
                \mathcal E_0 
                \ge
                F(x_0) - F(\bar x) - \frac{L_0}{2}\Vert \bar x - x_{-1}^\circ\Vert^2
                + \frac{L_0}{2} \Vert \bar x - x_0^\circ\Vert^2
                \\
                \iff 
                & \mathcal E_0 + \frac{L_0}{2}\Vert \bar x - x_{-1}^\circ\Vert^2 
                \ge F(x_0) - F(\bar x) + \frac{L_0}{2} \Vert \bar x - x_0^\circ\Vert^2. 
            \end{split}\end{align}
            The inequality at (2) comes from \eqref{ass:valid-err-schedule:item2} in Assumption \ref{ass:valid-err-schedule}. 
            Substituting \eqref{prop:inxt-apg-cnvg-generic:pitem4} into RHS of \eqref{prop:inxt-apg-cnvg-generic:pitem3} yields:
            \begin{align*}
                (\forall k \in \N)\quad 
                & F(x_k) - F(\bar x) + \frac{\alpha_k^2L_k}{2}\Vert \bar x - x_k^\circ\Vert^2
                \le 
                \beta_k \left(
                    \frac{L_0}{2}\Vert \bar x - x_{-1}^\circ\Vert^2
                    + \mathcal E_0 + \sum_{l = 1}^{k} \frac{\mathcal E_0}{l^p}
                \right). 
            \end{align*}
            Finally, when $k = 0$, from \eqref{prop:inxt-apg-cnvg-generic:pitem4} we have: 
            \begin{align*}
                F(x_0) - F(\bar x)
                + \frac{L_0}{2} \Vert \bar x - x_0^\circ\Vert^2
                &\le
                \mathcal E_0 + \frac{L_0}{2}\Vert \bar x - x_{-1}^\circ\Vert^2.
            \end{align*}
            Combining cases when $k \in \N$ and $k = 0$, and recalling that $\alpha_0 = \beta_0 = 1$, we can use $(\mathcal R_k(p))_{k \in \Z_+}$ introduced in \eqref{ass:valid-err-schedule:rkp} for the RHS and write it as:
            \begin{align*}
                (\forall k \in \Z_+)\quad 
                & F(x_k) - F(\bar x) + \frac{\alpha_k^2L_k}{2}\Vert \bar x - x_k^\circ\Vert^2
                \le 
                \beta_k \left(
                    \frac{L_0}{2}\Vert \bar x - x_{-1}^\circ\Vert^2
                    + \mathcal R_k(p)
                \right). 
            \end{align*}
        \end{proof}
        \par
        The following proposition states a relation between the termination criterion $\Vert x_k - y_k \Vert$ and the sequence $\alpha_k$.
        It is crucial to derive the convergence to stationarity in later sections. 
        \begin{proposition}[the termination criterion]\;\label{prop:vk-gm}\\ 
            Let $(F, f, g, L)$, $(\alpha_k, B_k, \rho_k, \epsilon_k)_{k \in \Z_+}$, $(\beta_k)_{k \in \Z_+}$, $\mathcal E_0, p$ and $\mathcal R_k(p)$ be given by Assumption \ref{ass:valid-err-schedule}. 
            Assume in addition there exists $\bar x \in \RR^n$ which is a minimizer of $F$, $\alpha_0 = 1$ and $\alpha_k \in (0, 1)$ for all $k \in \N$. 
            Then, for the iterates $(y_k, x_k, x_k^\circ)_{k \in \Z_+}$ generated by an algorithm satisfying Definition \ref{def:inxt-apg}, the following hold for all $k \in \Z_+$:
            \begin{enumerate}[nosep]
                \item\label{prop:vk-gm:result1} $x_k^\circ - x_{k - 1}^\circ = \alpha_k^{-1}(x_k - y_k)$. 
                \item\label{prop:vk-gm:result2} $\Vert \bar x - x_k^\circ\Vert \le \Vert \bar x - x_{-1}^\circ\Vert + \sqrt{\frac{2\mathcal R_k(p)}{L_0}}$.
                \item\label{prop:vk-gm:result3} $\Vert x_k - y_k\Vert \le 2\alpha_k \left(\Vert \bar x - x_{-1}^\circ\Vert+ \sqrt{\frac{2\mathcal R_k(p)}{L_0}}\right)$. 
            \end{enumerate}
        \end{proposition}
        \begin{proof}
            We now show \ref{prop:vk-gm:result1}.
            From $y_k = \alpha_k x_{k - 1}^\circ + (1 - \alpha_k)x_{k - 1}$ which is from \eqref{def:inxt-apg:yk}, and $x_k^\circ = x_{k - 1} + \alpha_k^{-1}(x_k - x_{k - 1})$ which is from \eqref{def:inxt-apg:vk}; it has $\forall k \in \Z_+$: 
            {\allowdisplaybreaks
            \begin{align*}
                x_k^\circ - x_{k - 1}^\circ &=
                x_{k - 1} + \alpha_k^{-1}(x_k - x_{k - 1}) 
                - \alpha_k^{-1}(y_k - (1 - \alpha_k)x_{k - 1})
                \\
                &= (1 - \alpha_k^{-1})x_{k - 1} + \alpha_k^{-1}x_k - \alpha_k^{-1}y_k + (\alpha_k^{-1} - 1)x_{k - 1}
                \\
                &= \alpha_k^{-1}(x_k - y_k). 
            \end{align*}
            }
            We now show \ref{prop:vk-gm:result2}. 
            The hypotheses of Proposition \ref{prop:inxt-apg-cnvg-generic} hold, so $\forall k \in \Z_+$:
            \begin{align*}
                0 &\le \beta_k\left(
                    \frac{L_0}{2}\Vert \bar x - x_{-1}^\circ\Vert^2
                    + \mathcal R_k(p)
                \right) 
                - \frac{\alpha_k^2L_k}{2}\Vert \bar x - x_k^\circ\Vert^2
                - (F(x_k) - F(\bar x))
                \\
                \underset{(1)}\implies
                0 &\le 
                \frac{L_0}{2}\Vert \bar x - x_{-1}^\circ\Vert^2
                + \mathcal R_k(p)
                - \frac{\alpha_k^2L_k}{2\beta_k}\Vert \bar x - x_k^\circ\Vert^2
                \\
                &\underset{(2)}= 
                \frac{L_0}{2}\Vert \bar x - x_{-1}^\circ\Vert^2
                + \mathcal R_k(p)
                - \frac{L_0}{2}\Vert \bar x - x_k^\circ\Vert^2
                \\
                \iff
                \Vert \bar x - x_k^\circ\Vert 
                &\le \left( \Vert \bar x - x_{-1}^\circ\Vert^2 +
                    \frac{2\mathcal R_k(p)}{L_0}
                \right)^{1/2}
            \end{align*}
            We did two things at (1).
            Firstly, assumed $\bar x$ is the minimizer, so $- F(x_k) + F(\bar x) \le 0$. 
            Next, we have $\beta_k > 0$ always, so we divide both sides of the inequality by $\beta_k$.
            At (2), we used $\beta_k = \frac{\alpha_k^2 L_k}{\alpha_0^2 L_0}$ from \eqref{lemma:opt-mmntm-seq-result2}. 
            Since we assumed $\alpha_0 = 1$, it has $\beta_k = \alpha_k^2L_k/L_0$ too and the coefficient $\frac{\alpha_k^2L_k}{2 \beta_k} = \frac{L_0}{2}$. 
            Therefore, it follows that: $\Vert \bar x - x_k^\circ\Vert \le \Vert \bar x - x_{-1}^\circ\Vert + \sqrt{\frac{2\mathcal R_k(p)}{L_0}}$. 
            \par
            We now show \ref{prop:vk-gm:result3}. 
            Using \ref{prop:vk-gm:result1} we have equality $\forall k \in \Z_+$: 
            {\allowdisplaybreaks
            \begin{align*}
                \Vert x_k - y_k\Vert
                &= \alpha_k\Vert x_k^\circ - x_{k - 1}^\circ\Vert
                \\
                &\le \alpha_k \left(\Vert x_k^\circ - \bar x\Vert + \Vert \bar x - x_{k - 1}^\circ\Vert\right)
                \\
                &\underset{\text{\ref{prop:vk-gm:result2}}}\le 
                \alpha_k \left(
                    2\Vert \bar x - x_{-1}^\circ\Vert
                    + \sqrt{\frac{2\mathcal R_k(p)}{L_0}} + \sqrt{\frac{2\mathcal R_{k - 1}(p)}{L_0}}
                \right)
                \\
                &\le 
                2\alpha_k \left(
                    \Vert \bar x - x_{-1}^\circ\Vert
                    + \sqrt{\frac{2\mathcal R_k(p)}{L_0}}
                \right). 
            \end{align*}
            }
        \end{proof}
        \par 
        Now, to derive a convergence rate in terms of iteration $k$ of the outer loop, it remains to determine a specific sequence $\alpha_k$. 
        This will be the goal of the next section. 
    
    \subsection{Auxiliary results under an optimal momentum schedule}\label{ssec:opt-momentum-seq}
        In the remainder of the paper we shall assume the following.  
        \begin{assumption}[the optimal momentum sequence]\;\label{ass:opt-mmntm-seq}\\
            Let $(\alpha_k, B_k, \rho_k, \epsilon_k)_{k \in \Z_+}$, $(F, f, g, L)$, $(y_k, x_k, x_k^\circ)_{k \in \Z_+}$, $\mathcal E_0, p$, and $(L_k)_{k \in \Z_+}$, $(\beta_k)_{k \in \Z_+}, (\mathcal R_k(p))_{k \in \Z_+}$ be given by Assumption \ref{ass:valid-err-schedule}. 
            In addition, we assume:
            \begin{enumerate}[nosep]
                \item\label{ass:opt-mmntm-seq:item1} $(\alpha_k)_{k \in \Z_+}$ satisfies $\alpha_0 = 1$, and $(1 - \alpha_k) = \alpha_{k}^2L_k\alpha_{k - 1}^{-2}L_{k - 1}^{-1}$ for all $k \in \Z_+$, and $p > 1$.
                \item\label{ass:opt-mmntm-seq:item2} 
                $(L_k)_{k \in \Z_+}$ is bounded, i.e., there exists constants $L_{\max} \ge L_{\min} > 0$ such that $\{L_k\}_{k\in \Z_+} \subseteq [L_{\min}, L_{\max}]$. 
                \item\label{ass:opt-mmntm-seq:item3} For all $k \in \N$, $\epsilon_k$ satisfies $\epsilon_k = \frac{\mathcal E_0 \beta_k}{k^p} + \rho_k \frac{\Vert x_k - y_k\Vert^2}{2}$ with the base case $\epsilon_0 = \mathcal E_0$.
                Each $\epsilon_k$ is chosen to be the largest possible value. 
            \end{enumerate}
        \end{assumption}
        \begin{remark}
            There are two more observations which holds. 
            Firstly, item \ref{ass:opt-mmntm-seq:item2} states that the sequence $(L_k)_{k \in \Z_+}$ is bounded which is definitely true under Lipschitz smoothness for reasonable implementations of algorithmic line search.
            To illustrate, under the assumption $f$ is $L$-Lipschitz smooth (Assumption \ref{ass:for-inxt-pg-ineq}), if the Armijo line search produces $B_k \ge L$ after some point, then the sequence $(B_k)_{k \in \Z_+}$ will cease to increase afterwards.
            In this case, any bounded sequence of $\rho_k$ chosen by the practitioners will allow $\sup_{i \in \Z_+} L_i$ to be bounded.
            Otherwise, if $B_k \searrow 0$ for some fictitious line search and backtracking techniques (to the best of our knowledge, there doesn't exist any line search satisfying this property), then practitioner have the freedom to assign $\rho_k$ in a way such that $L_k$ is bounded below.
            \par
            Secondly, observe that item \ref{ass:opt-mmntm-seq:item1} in the above assumption confines the choice of $\alpha_k$ to be a unique sequence, and it also restricts $(\epsilon_k)_{k \in \Z_+}$ to be as large as possible.
            These two details are worth emphasizing because the former will inform our result in Lemma \ref{lemma:opt-mmntm-seq} which comes immediately after; and the latter will inform our results in Lemma \ref{lemma:err-schedule-lbnd}.
            The first result states that the specific choice of $(\alpha_k)_{k \in \Z_+}$ satisfying Nesterov's rule restricts $\alpha_k \in (0, 1)$ for all $k \in \Z_+$, enabling an upper bound of $\mathcal O(1/k^2)$ for $(\beta_k)_{k \in \Z_+}$.
            The latter result will leverage $\beta_k$ to inform a lower bound for the sequence $\epsilon_k$ which will be vital for the derivation of the total complexity of the algorithm.
        \end{remark}
        Under Assumption \ref{ass:opt-mmntm-seq}, the momentum sequence $(\alpha_k)_{k \in \Z_+}$ follows Nesterov's update rule. 
        This assumption is stronger than Assumption \ref{ass:valid-err-schedule} and enables two intermediate results.
        The first ensures that such a sequence still satisfies Assumption \ref{ass:valid-err-schedule} and hence results from the previous section are applicable.
        The second result is the $\mathcal O(1/k^2)$ upper bound (and lower bound) on the sequence $(\beta_k)_{k \in \Z_+}$. 
        Both of them are proved in Lemma \ref{lemma:opt-mmntm-seq}. 
        
        \begin{lemma}[the optimal momentum sequence is indeed valid and optimal]\;\label{lemma:opt-mmntm-seq}\\
            Let $(\alpha_k)_{k \in \Z_+}, (\beta_k)_{k \in \Z_+}$ be given by Assumption \ref{ass:opt-mmntm-seq}. 
            If we choose $\alpha_0 = 1$, then for all $k \in \N$:
            \begin{align}\label{lemma:opt-mmntm-seq-result1}
                \alpha_k &=
                \frac{L_{k - 1}}{2L_k}\left(
                    - \alpha_{k - 1}^2 + \left(
                        \alpha_{k - 1}^4 + 4 \alpha_{k - 1}^2\frac{L_k}{L_{k - 1}}
                    \right)^{1/2}
                \right) \in (0, 1)
            \end{align}
            By the base case $\beta_0 = 1$, the sequence $(\beta_k)_{k \in \Z_+}$ has $\forall k \in \N$:
            \begin{align}\label{lemma:opt-mmntm-seq-result2}
                \left(
                    1 + \alpha_0\sqrt{L_0}\sum_{i = 1}^{k}\sqrt{L_i^{-1}}
                \right)^{-2}
                \hspace{-1em}\le 
                \beta_k = \frac{\alpha_k^2L_k}{\alpha_0^2L_0}
                \le 
                \left(
                    1 + \frac{\alpha_0\sqrt{L_0}}{2}\sum_{i = 1}^{k}\sqrt{L_i^{-1}}
                \right)^{-2}. 
            \end{align}
        \end{lemma}
        \begin{proof}
            Firstly, we show \eqref{lemma:opt-mmntm-seq-result1}.
            We proceed by induction.
            Fix any $k \in \N$. 
            Assume inductively that $\alpha_{k - 1} \in (0, 1]$. 
            Obviously, the base case is satisfied with $\alpha_0 = 1 \in (0, 1]$. 
            \par
            We can solve for $\alpha_k$ in the recursive equality $(1 - \alpha_k) = \alpha_{k}^2L_k\alpha_{k - 1}^{-2}L_{k - 1}^{-1}$ from Assumption \ref{ass:opt-mmntm-seq}.
            To simplify notation, we write $\alpha_{k - 1}$ as $\alpha$, and $L_k/L_{k - 1}$ as $q$.
            Solving for $\alpha_{k}$, the quadratic equation admits one root that is strictly positive for all $k \in \N$: 
            \begin{align*}
                \alpha_k &= \frac{1}{2}\left(
                    - \frac{\alpha^2}{q} + \sqrt{
                        \frac{\alpha^4}{q^2} + \frac{4\alpha^2}{q}
                    }
                \right)
                \\
                &= \frac{\alpha^2}{2q}\left(
                    - 1 + \sqrt{1 + \frac{4q}{\alpha^2}}
                \right)
                \\
                &\underset{(1)}{<}
                \frac{\alpha^2}{2q}\left(
                    -1 + 1 + \frac{2q}{\alpha^2}
                \right)
                \\
                &= 1
            \end{align*}
            At (1) we bounded the radical using $\sqrt{a^2 + b} < a + b/(2a)$.
            We used the assumption that $\alpha_k > 0$, $L_k > 0, L_{k - 1} > 0$.
            This is true because we have $B_k > 0, \rho_k \ge 0$.
            % \begin{align*}
            %     1 + \frac{4q}{\alpha^2} &= 
            %     1 + \frac{4q}{\alpha^2} + \frac{4q^2}{\alpha^4} - \frac{4\alpha^2}{\alpha^2}
            %     \\
            %     &= \left(
            %         1 + \frac{2q}{\alpha^2}
            %     \right)^2 - \frac{4q^2}{\alpha^4} 
            %     \\
            %     &< \left(
            %         1 + \frac{2q}{\alpha^2}  
            %     \right)^2. 
            % \end{align*}
            Next, to see that $\alpha_k > 0$, recall the same fact that $L_k > 0$, and the inductive hypothesis $\alpha_{k - 1} \in (0, 1]$. 
            Therefore, $4q/\alpha^2 > 0$. 
            It follows that $\alpha_k = \frac{\alpha^2}{2q}\left(-1 + \sqrt{1 + 4q/\alpha^2}\right) > 0$ because the quantity inside the radical is strictly larger than $1$.
            Therefore, inductively it holds that $\alpha_k \in (0, 1)$ too. 
            \par
            We now show \eqref{lemma:opt-mmntm-seq-result2}.
            Using the assumption that $(\alpha_k)_{k \in \Z_+}$ satisfying $(1 - \alpha_k) = \alpha_k^2L_k \alpha_{k - 1}^{-2}L_{k - 1}^{-1}$ for all $k \in \Z_+$, we can simplify the definition of $\beta_k$ from \eqref{ass:valid-err-schedule:betak}, yielding:
            \begin{align*}
                \beta_k = \prod_{i = 1}^k \max\left(
                    1 - \alpha_i, \frac{\alpha_i^2L}{\alpha_{i - 1}^2L_{i - 1}}
                \right) 
                = \prod_{i = 1}^k(1 - \alpha_i) 
                = \prod_{i = 1}^k \frac{\alpha_i^2L_{i}}{\alpha_{i - 1}^2L_{i - 1}}
                = \frac{\alpha_k^{2}L_k}{\alpha_0^2L_0}. 
            \end{align*}
            The above equalities imply for all $k \in \N$: 
            \begin{enumerate}[nosep]
                \item[(a)] $\beta_k$ is monotone decreasing and $\beta_k > 0$ for all $k \in \Z_+$ because $\beta_k = \prod_{i = 1}^{k} (1 - \alpha_i)$ and, $\alpha_k \in (0, 1]$. 
                \item[(b)] The equalities $\frac{\beta_k}{\beta_{k - 1}} = (1 - \alpha_k) = \frac{\alpha_k^2L_k}{\alpha_{k - 1}^2 L_{k - 1}}$ hold for all $k \in \N$.
            \end{enumerate}
            Using the above observations, we can show the chain of equalities $\alpha_k^{2} = (1 - \beta_k/\beta_{k - 1})^2 = \beta_kL_0\alpha_0^2L_k^{-1}$ for all $k \in \Z_+$. 
            This is true because (b) has: 
            \begin{align}\label{eqn:opt-mmntm-seq-pitem1}\begin{split}
                (1 - \alpha_k) &= \beta_k/\beta_{k - 1}
                \\
                \iff 
                \alpha_k &= 1 - \beta_k / \beta_{k - 1}
                \\
                \implies 
                \alpha_k^2 &= (1 - \beta_k / \beta_{k - 1})^2. 
            \end{split}\end{align}
            Next, the recursive relation of $(\alpha_k)_{k \in \Z_+}$ gives
            \begin{align}\label{eqn:opt-mmntm-seq-pitem2}\begin{split}
                \alpha_k^2&= (1 - \alpha_k)\alpha_{k - 1}^2L_{k - 1}L_k^{-1}
                \\
                &= 
                (1 - \alpha_k)
                \left(
                    \frac{\alpha_{k - 1}^2L_{k - 1}}{\alpha_0^2L_0}
                \right)
                \frac{\alpha_0^2L_0}{L_k}
                \\
                &= 
                (\beta_k\beta_{k - 1}^{-1})\left(
                    \beta_{k - 1}
                \right)L_0\alpha_0^2L_k^{-1}
                \\
                &= \beta_kL_0\alpha_0^2L_k^{-1}.         
            \end{split}\end{align}
            Combining \eqref{eqn:opt-mmntm-seq-pitem1}, \eqref{eqn:opt-mmntm-seq-pitem2} and the fact that $\forall k \in \Z_+: \beta_k > 0$, it follows that $\forall \; i \ge 1$: 
            \begin{align*}
                L_0 \alpha_0^2 L_i^{-1} &= 
                \beta_i^{-1}\left(
                    1 - \frac{\beta_i}{\beta_{i - 1}}
                \right)^2
                \\
                &= 
                \beta_i \left(
                    \beta_i^{-1} - \beta_{i - 1}^{-1}
                \right)^2
                \\
                &=
                \beta_i \left(
                    \beta_i^{-1/2} - \beta_{i - 1}^{-1/2}
                \right)^2
                \left(
                    \beta_i^{-1/2} + \beta_{i - 1}^{-1/2}
                \right)^2
                \\
                &= \left(
                    \beta_i^{-1/2} - \beta_{i - 1}^{-1/2}
                \right)^2\left(
                    1 + \beta_i^{1/2}\beta_{i - 1}^{-1/2}
                \right)^2. 
            \end{align*}
            Since $\beta_i$ is monotone decreasing, $0 < \beta_i^{1/2}\beta_{i - 1}^{-1/2} \le 1$, giving:
            \begin{align*}
                \beta_i^{-1/2} - \beta_{i - 1}^{-1/2} 
                &\le \alpha_0\sqrt{\frac{L_0}{L_i}} 
                \le 2 \left(
                    \beta_i^{-1/2} - \beta_{i - 1}^{-1/2}
                \right).
            \end{align*}
            Telescope it for $i = 1, 2, \ldots, k$, using the fact $\beta_0 = 1$ yields the desired results: \eqref{lemma:opt-mmntm-seq-result2}.
        \end{proof}
        \begin{remark}
            This result is not entirely new; it has appeared in Güler \cite[Lemma 2.2]{guler_new_1992}.
            The difference here is the context. 
            We consider accelerated proximal gradient with line search instead of accelerated proximal point. 
            Nonetheless, the parameter $L_k$ is analogous to $\lambda_k$ in Güler's work.
        \end{remark}
    \subsection{One standalone auxiliary result for the total complexity}\label{ssec:outerloop-err-shrink}
        This section derives one key result which will facilitate the derivations of the total complexity of IAPG because it relates the complexity of the inner loop and the outer loop together by the sequence $(\beta_k)_{k \in \Z_+}$. 
        We will show in Lemma \ref{lemma:err-schedule-lbnd} that $\epsilon_k \searrow 0$ in Assumption \ref{ass:opt-mmntm-seq} shrinks no faster than $\mathcal O(k^{-2-p})$. 
        This is imperative because if the error sequence in Assumption \ref{ass:opt-mmntm-seq} approaches zero from above at a rate that cannot be characterized, then it is impossible to bound the complexity of the inner loop in relation to the outer loop.
        \begin{lemma}[error schedule lower bound]\;\label{lemma:err-schedule-lbnd}\\
            Let $(\alpha_k, B_k, \rho_k, \epsilon_k)_{k \in \Z_+}$, $\mathcal E_0, p$, and $(\beta_k)_{k \in \Z_+}$ be given by Assumption \ref{ass:opt-mmntm-seq}, and let $L_k := \rho_k + B_k$.
            Then, $\epsilon_k^{-1}$ has for all $k \in \Z_+$ the upper bound: 
            \begin{align*}
                \epsilon_k^{-1} \le \max\left(
                    \mathcal E_0^{-1}, \frac{4k^{2 + p}}{\mathcal E_0}
                \right)
                = \mathcal O(k^{2 + p}). 
            \end{align*}
            And when $k = 0$, it has naturally $\epsilon_0 \ge \mathcal E_0$. 
        \end{lemma}
        \begin{proof}
            From Assumption \ref{ass:opt-mmntm-seq} the largest valid error schedule is $\epsilon_k = \frac{\mathcal E_0 \beta_k}{k^p} + \rho_k \frac{\Vert x_k - y_k\Vert^2}{2}$ and therefore $\forall k \in \N$: 
            \begin{align*}
                \epsilon_k &\ge \frac{\mathcal E_0 \beta_k}{k^p}
                \\
                &\underset{(1)}{\ge} 
                \left(
                    1 + \alpha_0\sqrt{L_0}\sum_{i = 1}^{k}\sqrt{L_i^{-1}}
                \right)^{-2} 
                \frac{\mathcal E_0}{k^p}
                \\
                &\underset{(2)}{\ge}
                \frac{\mathcal E_0}{k^p}
                \left(
                    1 + k\sqrt{L_0}\sqrt{L_{\min}^{-1}}
                \right)^{-2} 
                \\
                &\ge \frac{\mathcal E_0}{k^p}\left(
                    \sqrt{\frac{L_0}{L_{\min}}} + k \sqrt{\frac{L_0}{L_{\min}}}
                \right)^{-2}
                \\
                &\ge \frac{\mathcal E_0}{k^p(1 + k)^2} \frac{L_0}{L_{\min}}
                \\
                &\ge \frac{\mathcal E_0}{4k^{2 + p}}\frac{L_0}{L_{\min}}
                \\
                &\ge \frac{\mathcal E_0}{4k^{2 + p}}
                \\
                &= \mathcal O(k^{-2-p}). 
            \end{align*}
            At (1), we used Lemma \ref{lemma:opt-mmntm-seq} and $\alpha_0 = 1$ (Assumption \ref{ass:opt-mmntm-seq}\ref{ass:opt-mmntm-seq:item1}). 
            At (2), we used that $L_{\min} \le L_i$ for all $i \in  \Z_+$. 
            Finally, for the case when $k = 0$, it had $\epsilon_0 = \mathcal E_0$ from Assumption \ref{ass:opt-mmntm-seq}\ref{ass:opt-mmntm-seq:item3} meaning that $\epsilon_0 \ge \mathcal E_0$.
            From here, we can obtain an upper bound on $\epsilon_k^{-1}$ for all $k \in \Z_+$: 
            \begin{align*}
                \epsilon_k^{-1} \le \max\left(
                    \mathcal E_0, \frac{4k^{2 + p}}{\mathcal E_0}
                \right). 
            \end{align*}
        \end{proof}

    \subsection{Convergence results of the outer loop}\label{ssec:outerloop-cnvg}
        This section contains three major results regarding the IAPG outer loop convergence. 
        Theorem \ref{thm:opt-cnvg-outr-loop} establishes that under Assumption \ref{ass:opt-mmntm-seq}, the sequence $(F(x_k))_{k \in \Z_+}$ converges to the minimum $F(\bar x)$ at a rate of $\mathcal O(1/k^2)$ where $\bar x$ is a minimizer of $F$.
        This is the optimal convergence rate of the optimality gap of the IAPG outer loop.
        Theorem \ref{thm:pg-rsdul-cnvg} provides the second result that $\Vert x_k - y_k\Vert$ converges to $0$ at a rate of $\mathcal O(1/k)$, establishing $\Vert x_k - y_k\Vert$ as a termination criterion, implying convergence to stationarity.
        Finally, our third result (Theorem \ref{thm:outerlp-itr-cmplx}) states the iteration complexity required to achieve $\varepsilon$ gap for: optimality gap, i.e., $F(x_k) - F(\bar x) \le \varepsilon$; the termination criteria, i.e., $\Vert x_k - y_k\Vert \le \varepsilon$; and the stationarity conditions, i.e., $\dist(\mathbf 0 | \partial_{\epsilon_k}F(x_k)) \le \varepsilon$.
        \begin{theorem}[$\mathcal O(1/k^2)$ outer loop function value convergence]\;\label{thm:opt-cnvg-outr-loop}\\
            Let $(f, g, L)$, $(\alpha_k, B_k, \rho_k, \epsilon_k)_{k \in \Z_+}$, $(y_k, x_k, x_k^\circ)_{k \in \Z_+}$, $\mathcal E_0, p$, and $(\beta_k)_{k \in \Z_+}, (\mathcal R_k(p))_{k \in \Z_+}$ be given by Assumption \ref{ass:opt-mmntm-seq}.
            For all $k \in \Z_+$ define $L_k = B_k + \rho_k$. 
            We consider sequence $(y_k, x_k, x_k^\circ)_{k \in \Z_+}$ generated by an algorithm given by Definition \ref{def:inxt-apg} for any initial guess $x_{-1}, x_{-1}^\circ \in \RR^n$. 
            Assume in addition that there exists $\bar x \in \RR^n$ that is a minimizer of $F = f + g$. 
            Then, for all $k \in \Z_+$:
            \begin{align*}
                & F(x_k) - F(\bar x) + \frac{\alpha_k^2L_k}{2}\Vert \bar x - x_k^\circ\Vert^2
                \le 
                \left(
                    1 + \frac{k\sqrt{L_0}}{2\sqrt{L_{\max}}}
                \right)^{-2}\left(
                    \frac{L_0}{2}\Vert \bar x - x_{-1}^\circ\Vert^2
                    + \mathcal R_k(p)
                \right). 
            \end{align*}
            In addition, $\mathcal R_k(p)$ evaluates to a convergent series;  hence, the above inequality establishes a convergence rate $\mathcal O(1/k^2)$ of the optimality gap.
        \end{theorem}
        \begin{proof} 
            Here, we operate under Assumption \ref{ass:opt-mmntm-seq}.
            Therefore, results from Lemma \ref{lemma:opt-mmntm-seq} apply because of the same set of assumptions.
            In addition, we have $\alpha_0 = 1$, and $L_k$ is bounded (Assumption \ref{ass:opt-mmntm-seq}\ref{ass:opt-mmntm-seq:item2}, \ref{ass:opt-mmntm-seq:item1}) and therefore:
            \begin{align*}
                \beta_k &\le 
                \left(
                    1 + \frac{\sqrt{L_0}}{2}\sum_{i = 1}^{k} \sqrt{L_i^{-1}}
                \right)^{-2}
                \hspace{-0.7em}\le 
                \left(
                    1 + \frac{k\sqrt{L_0}}{2\sqrt{L_{\max}}}
                \right)^{-2}. 
            \end{align*}
            Then, we apply Proposition \ref{prop:inxt-apg-cnvg-generic} because of three reasons.
            Firstly, Assumption \ref{ass:opt-mmntm-seq} includes everything in Assumption \ref{ass:valid-err-schedule}. 
            Secondly, we also assumed here $\alpha_0 = 1$ (Assumption \ref{ass:opt-mmntm-seq}\ref{ass:opt-mmntm-seq:item1}).
            Thirdly, we have $\alpha_k \in (0, 1)$ for all $k \in \N$ from Lemma \ref{lemma:opt-mmntm-seq}.
            Therefore, the result in Proposition \ref{prop:inxt-apg-cnvg-generic} applies, and the inequality strengthens into:
            \begin{align*}
                F(x_k) - F(\bar x) + \frac{\alpha_k^2L_k}{2}\Vert \bar x - x_k^\circ\Vert^2
                \le 
                \left(
                    1 + \frac{k\sqrt{L_0}}{2\sqrt{L_{\max}}}
                \right)^{-2}
                \left(
                    \frac{L_0}{2}\Vert \bar x - x_{-1}^\circ\Vert^2
                    + \mathcal R_k(p)
                \right). 
            \end{align*}
            Since $\bar x$ is the minimizer and $p > 1$ by Assumption \ref{ass:opt-mmntm-seq}\ref{ass:opt-mmntm-seq:item1}, we have $\mathcal R_k(p) < \mathcal R_\infty(p) < \infty$.
            Therefore, the above establishes that the function value converges to the minimum at a rate of $\mathcal O(1/k^2)$.
        \end{proof}
        \begin{remark}
            The seminal works of Villa et al. \cite{villa_accelerated_2013}, Schmidt et al. \cite{schmidt_convergence_2011} showed the convergence rate of IAPG a decade ago; however our results here differ in context, and they also extend these results found in the literature.
            Two aspects of our work are distinct from previous works.
            Firstly, we include a line search and backtracking, and secondly we include both relative and absolute error sequence for $\epsilon_k$. 
            \par
            In contrast to seminal works by Villa et al. \cite{villa_accelerated_2013}, Schmidt et al. \cite{schmidt_convergence_2011}, our outer loop convergence results include line search and the stepsize is chosen by the descent lemma. 
            Therefore, our results are applicable for algorithm that implements line search and back tracking, for example the technique proposed by Calatroni, Chambolle \cite{calatroni_backtracking_2019}. 
            In addition, Assumption \ref{ass:opt-mmntm-seq}\ref{ass:opt-mmntm-seq:item3} introduced $\epsilon_k$ which is a combination of relative error, and absolute errors.
            Therefore, our convergence result is more flexible. 
            \par
            Finally, our result here still remains new compared to more recent works concerning the IAPG method.
            Results by Bello-Cruz et al. \cite{bello-cruz_inexact_2020-1-1} differ from ours because they did not incorporate line search and absolute error.
            Furthermore, their work differs from ours because their theoretical focus differs, and our result here complements their result showing that both types of errors can be presented at the same time.
            In both cases, our outer loop result here differs from existing work in the literature, and it complements and extends prior works.
            \par
        \end{remark}
        Under the assumption that the objective function has a minimizer $\bar x$, our next theorem states that $\Vert x_k - y_k\Vert$ is a sufficient termination criterion for stationarity. 
        To be specific, it shows that $\Vert x_k - y_k\Vert \rightarrow 0$ at an $\mathcal O(1/k)$ rate, which implies that $\dist(\mathbf 0 | \partial_{\epsilon_k}F(x_k))$ converges at the same rate.
        \begin{theorem}[$\mathcal O(1/k)$ convergence to stationarity]\;\label{thm:pg-rsdul-cnvg}\\
            Let $(f, g, L)$, $(\alpha_k, B_k, \rho_k, \epsilon_k)_{k \in \Z_+}$, $\mathcal E_0, p$, $(\beta_k)_{k \in \Z_+}$ and $\mathcal R_k(p)$ be given by Assumption \ref{ass:opt-mmntm-seq}. 
            For all $k \in \Z_+$ define $L_k := B_k + \rho_k$. 
            Fix $\bar x \in \RR^n$ to be a minimizer of $F$. 
            Let initial guess $x_{-1}, x_{-1}^\circ \in \RR^n$ be arbitrary.
            An algorithm which satisfies Definition \ref{def:inxt-apg} generates iterates $(y_k, x_k, x_k^\circ)_{k \in \Z_+}$ such that $k \in \Z_+$:
            \begin{align}\label{thm:pg-rsdul-cnvg:rslt1}
                \begin{split}
                    & (L + L_k)^{-1}\dist(\mathbf 0 | \partial_{\epsilon_k}F(x_k))
                    \\
                    &\le \Vert x_k - y_k\Vert 
                    \\
                    &\le 
                    2\sqrt{\frac{L_0}{L_{\max}}}\left(
                        1 + \frac{k\sqrt{L_0}}{2\sqrt{L_{\max}}}
                    \right)^{-1}\left(\Vert \bar x - x_{- 1}^\circ\Vert+ \sqrt{\frac{2\mathcal R_k(p)}{L_0}}\right). 
                \end{split}
            \end{align}
            Hence, $\Vert x_k - y_k\Vert$ and $\dist (\mathbf 0 | \partial_{\epsilon_k}F(x_k))$ converge to zero at a rate of $\mathcal O(1/k)$.
        \end{theorem}
        \begin{proof}
            We have the following: 
            \begin{align*}
                \alpha_k
                &\underset{(1)}= \sqrt{\frac{\beta_kL_0}{L_k}}
                \underset{(2)}\le \sqrt{\frac{L_0}{L_k}}\left(
                    1 + \frac{\sqrt{L_0}}{2}\sum_{i = 1}^{k} \sqrt{L_i^{-1}}
                \right)^{-1} 
                \hspace{-1em}\underset{(3)}
                \le 
                \sqrt{\frac{L_0}{L_{\max}}}\left(
                    1 + \frac{k\sqrt{L_0}}{2\sqrt{L_{\max}}}
                \right)^{-1}. 
            \end{align*}
            At (1), by Assumption \ref{ass:opt-mmntm-seq} the result \eqref{lemma:opt-mmntm-seq-result2} in Lemma \ref{lemma:opt-mmntm-seq} applies, so $\beta_k = \frac{\alpha_k^2L_k}{\alpha^2_0L_0}$. 
            Since $\alpha_0 = 1$ (Assumption \ref{ass:opt-mmntm-seq}\ref{ass:opt-mmntm-seq:item1}), re-arranging gives $\alpha_k = \sqrt{\frac{\beta_kL_0}{L_k}}$. 
            At (2), we replace $\beta_k$ by its upper bound in \eqref{lemma:opt-mmntm-seq-result2}.
            At (3), we apply Assumption \ref{ass:opt-mmntm-seq}\ref{ass:opt-mmntm-seq:item2} which states that $(L_k)_{k \in \Z_+}$ is bounded above by $L_{\max}$.
            \par
            Next, we apply the result from Proposition \ref{prop:vk-gm}\ref{prop:vk-gm:result3} for three reasons.
            Firstly, here we assumed $\bar x \in \RR^n$ is a minimizer of $F$. 
            Secondly, here we assumed Assumption \ref{ass:opt-mmntm-seq} which included Assumption \ref{ass:valid-err-schedule}. 
            Thirdly, we have $\alpha_k \in (0, 1)$ for all $k \in \N$ by Lemma \ref{lemma:opt-mmntm-seq}. 
            Therefore, invoking this result and combining it with the previously derived inequality for $\alpha_k$ yields:
            \begin{align*}
                \Vert x_k - y_k\Vert &\le 
                2\alpha_k \left(\Vert \bar x - x_{-1}^\circ\Vert+ \sqrt{\frac{2\mathcal R_k(p)}{L_0}}\right)
                \\
                &\le 
                2\sqrt{\frac{L_0}{L_{\max}}}\left(
                    1 + \frac{k\sqrt{L_0}}{2\sqrt{L_{\max}}}
                \right)^{-1}
                \left(\Vert \bar x - x_{-1}^\circ\Vert+ \sqrt{\frac{2\mathcal R_k(p)}{L_0}}\right). 
            \end{align*}
            We have now justified the second inequality in \eqref{thm:pg-rsdul-cnvg:rslt1}. 
            To justify the first inequality, recall that $x_k \approx_{\epsilon_k} T_{L_k}(y_k)$ from \eqref{def:inxt-apg:xk}.
            Therefore, we can invoke Lemma \ref{lemma:pg-residual} with $\epsilon = \epsilon_k, x = y_k, \tilde x = x_k, \rho = L_k$ which yields: $\Vert x_k - y_k\Vert \ge (L + L_k)^{-1}\dist\left(\mathbf 0 | \partial_{\epsilon_k} F(x_k)\right)$. 
        \end{proof}
        \begin{remark}
            The convergence of $x_k$ to stationarity has been established for Accelerated Proximal Gradient in the literature.
            However, to the best of our knowledge, it is new for IAPG.
        \end{remark}
        \begin{theorem}[iterative complexity of the outer loop]\;\label{thm:outerlp-itr-cmplx}\\
            Let $(f, g, L)$, $(\alpha_k, B_k, \rho_k, \epsilon_k)_{k \in \Z_+}$, $\mathcal E_0, p$, $(\beta_k)_{k \in \Z_+}$ and $\mathcal R_k(p)$ be given by Assumption \ref{ass:opt-mmntm-seq}. 
            For all $k \in \Z_+$ define $L_k := B_k + \rho_k$. 
            Fix any $\bar x \in \RR^n$ to be a minimizer of $F$. 
            Consider iterates $(y_k, x_k, x_k^\circ)_{k \in \Z_+}$ generated by an algorithm satisfying Definition \ref{def:inxt-apg} for an arbitrary initial guess $x_{-1}, x_{-1}^\circ \in \RR^n$. 
            Define the following constants: 
            \begin{enumerate}[nosep,label=(\Roman*),font=\upshape]
                \item\label{thm:outerlp-itr-cmplx:cnst1} $C_1 := \frac{1}{2}\sqrt{\frac{L_0}{L_{\max}}}$,
                \item\label{thm:outerlp-itr-cmplx:cnst2} $C_2 := \frac{L_0}{2}\Vert \bar x - x_{-1}^\circ\Vert^2$,
                \item\label{thm:outerlp-itr-cmplx:cnst3} $C_3 := \Vert \bar x - x_{-1}^\circ\Vert + \sqrt{\frac{2 \mathcal R_\infty(p)}{L_0}}$. 
            \end{enumerate}
            Then, for all $\varepsilon > 0$ the following hold: 
            \begin{enumerate}[nosep]
                \item\label{thm:outerlp-itr-cmplx:rsult1} If $k\ge \sqrt{\frac{C_2}{\varepsilon C_1^2}}$, then $F(x_k) - F(\bar x) \le \varepsilon$.
                \item\label{thm:outerlp-itr-cmplx:rsult2} If $k \ge \frac{4C_1C_2C_3}{\varepsilon}$, then $\Vert x_k - y_k\Vert \le \varepsilon$. 
                \item\label{thm:outerlp-itr-cmplx:rsult3} If $k \ge \frac{4C_1C_3(L + L_{\max})}{\varepsilon C_2}$, then $\dist(\mathbf 0 | \partial_{\epsilon_k}F(x_k)) \le \varepsilon$. 
            \end{enumerate}
        \end{theorem}
        \begin{proof}
            To verify \ref{thm:outerlp-itr-cmplx:rsult1}, we apply Theorem \ref{thm:opt-cnvg-outr-loop} because it shares the same set of assumptions. 
            With $C_1, C_2$ in the theorem statement, the inequalities reads:
            \begin{align*}
                F(x_k) - F(\bar x) 
                \le (1 + C_1 k)^{-2}C_2 
                \le k^{-2}C_1^{-2}C_2 &\underset{(1)}
                \le \varepsilon.
            \end{align*}
            At (1), note that $k \ge \sqrt{\frac{C_2}{\varepsilon C_1^2}}$. 
            Therefore, it implies $k^{-2} \le \frac{\varepsilon C_1^2}{C_2}$, and substituting it validates the inequality. 
            \par
            Next, to show \ref{thm:outerlp-itr-cmplx:rsult2}, we invoke results from Theorem \ref{thm:pg-rsdul-cnvg} because it shares the same set of assumptions. 
            By the definitions of $C_1, C_2, C_3$ in the theorem statement, its results can be written as:
            \begin{align*}
                \Vert x_k - y_k\Vert 
                \le 
                4C_1(1 + kC_2)^{-1}C_3 
                \le 
                4C_1(kC_2)^{-1}C_3 \underset{(2)}\le \varepsilon.
            \end{align*}
            At (2), we used $k \ge \frac{4C_1C_2C_3}{\varepsilon}$ which implies $k^{-1} \le \frac{\varepsilon}{4C_1C_2C_3}$. 
            Therefore, substituting it verifies the inequality. 
            \par
            Finally, to show \ref{thm:outerlp-itr-cmplx:rsult3}, recall from \eqref{def:inxt-apg:xk} that $x_k \approx_{\epsilon_k} T_{L_k}(y_k)$.
            Therefore, we invoke \eqref{thm:pg-rsdul-cnvg:rslt1} from Theorem \ref{thm:pg-rsdul-cnvg} with $\rho = L_k, \epsilon = \epsilon_k, x = y_k$, and $\tilde x = x_k$ which yields:
            \begin{align*}
                \dist(\mathbf 0 | \partial_{\epsilon_k} F(x_k))
                &\le (L + L_k)\Vert x_k - y_k\Vert
                \\
                &\le 
                (L + L_k)4C_1(1 + kC_2)^{-1}C_3 
                \\
                &\le 4(L + L_{\max})C_1(kC_2)^{-1}C_3 
                \\
                &\underset{(3)}\le \varepsilon. 
            \end{align*}
            At (3), recall that $k \ge \frac{4C_1C_3(L + L_{\max})}{\varepsilon C_2}$ which implies $C_2k \ge \frac{4C_1C_3(L + L_{\max})}{\varepsilon}$, yielding $(kC_2)^{-1} \le \frac{\varepsilon}{4C_1C_3(L + L_{\max})}$. 
            Substituting it verifies the inequality. 
        \end{proof}

\section{Linear convergence rate of the inner loop}\label{sec:innlp-lin-cnvg}

    Continuing from Section \ref{ssec:optz-inxt-pp-problem}, our goal in this section is to show that for a fixed value of $y, \lambda$ we can find an element of $\approx_\epsilon \hprox_{\lambda\omega\circ A}(y)$ in complexity $\mathcal O(\ln(1/\epsilon))$ by incorporating the condition called ``quadratic growth" (Definition \ref{def:q-grwth}) into $\Psi_{\lambda}$.
    To achieve that, we divide it into three subsections.
    \par
    In Section \ref{ssec:lin-pgd} we establish, in general, the linear convergence rate of PGD.
    More specifically, we show that Proximal Gradient Descent (PGD) converges linearly when the objective function satisfies the quadratic growth condition.
    To prepare us for the complexity results of the inner loop, we introduce the algorithm needed to conduct the inner loop, and also introduce the assumptions on $\Phi_\lambda, \Psi_\lambda$ in Section \ref{ssec:conds-lin-cnvg-pp}.
    Finally, in Section \ref{ssec:lin-cnvg-inner-lp}, we derive our main results, which states that the total number of inner-loop iterations $j$ needed to obtain $z_j \approx_\epsilon \hprox_{\lambda \omega\circ A}(y)$ is bounded by $\mathcal O(\ln(\epsilon^{-1}))$.
    \subsection{Linear convergence of PGD}\label{ssec:lin-pgd}
        This subsection provides a set of conditions (Assumption \ref{ass:pgd-qg}) such that the PGD with line search (Definition \ref{def:pgd-with-qg}) has linear convergence rate. 
        Then, we will prove the major result (Theorem \ref{def:pgd-with-qg}) which states that the distance of the iterates generated by PGD to the set of minimizer converges linearly, and simultaneously the function value converges linearly to the minimum as well. 
        One of the key condition essential for our major result is called Quadratic Growth (Definition \ref{def:q-grwth}).
        The auxiliary result contributing directly to linear convergence in function value is in Lemma \ref{lemma:pg-env-upbnd}. 
        \par
        The following definition states the quadratic growth condition. 
        It states that the distance square to the set of minimizer is bounded by the optimality gap of the function by a Lipschitz relation. 
        \begin{definition}[quadratic growth condition]\;\label{def:q-grwth}\\
            Let $F: \RR^n \rightarrow \overline \RR$ be proper closed and convex. 
            Assume $S := \argmin_{x\in \RR^n}F(x)\neq \emptyset$. 
            Denote $F_{\min} = \min_{x \in \RR^n} F(x)$. 
            Then $F$ satisfies quadratic growth with constant $\kappa$ if $\exists \kappa > 0$: 
            \begin{align*}
                (\forall x \in \RR^n)\; F(x) \ge F_{\min} + \frac{\kappa}{2}\dist^2(x | S).
            \end{align*}
        \end{definition}
        \begin{assumption}[conditions for linear convergence of PGD]\;\label{ass:pgd-qg}\\
            The following assumption is about $(F, f, g, L, S, \kappa)$. 
            \begin{enumerate}[nosep]
                \item $F = f + g$ where $f$ is a convex differentiable $L$-Lipschitz smooth function (Definition \ref{def:lipz-smooth}), $g: \RR^n \rightarrow \overline \RR$ is a proper closed, and convex function. 
                \item Assume we can evaluate the exact proximal gradient operator $T_{\tau}$ of $f + g$ (Definition \ref{def:exact-pg}). 
                \item \label{ass:pgd-qg:item3} Assume that $S = \argmin_{x \in \RR^n}{F(x)} \neq \emptyset$, and hence $F$ admits minimum $F_{\min} = \min_{x \in \RR^n}F(x)$. 
                \item \label{ass:pgd-qg:item4} $F$ satisfies quadratic growth (Definition \ref{def:q-grwth}) for some $\kappa > 0$. 
            \end{enumerate} 
        \end{assumption}
        Under the above assumption, we derive the linear convergence of the Proximal Gradient Descent method. 
        This is crucial because PGD is a key part of optimizing the inexact proximal operator of the inner loop. 
        \begin{remark}
            The acronym PGD stands for Proximal Gradient Descent. 
        \end{remark}
        We state the following lemma which is useful when we derive the linear convergence rate of function values of PGD. 
        \begin{lemma}[proximal gradient envelope upper bound]\label{lemma:pg-env-upbnd}
            Let $F, f, g, L$ be given by Assumption \ref{ass:pgd-qg}. 
            Choose any $x \in \RR^n, \tau > 0$, and consider $x^+ = T_{\tau^{-1}}(x)$. 
            Then, it has: 
            \begin{align*}
                & g(x^+) + f(x) 
                + \langle \nabla f(x), x^+ - x\rangle
                + \frac{\tau}{2}\Vert x^+ - x\Vert^2
                \le 
                \min_{z \in \RR^n}\left\lbrace F(z) + \frac{\tau}{2}\Vert z - x\Vert^2 \right\rbrace. 
            \end{align*}
        \end{lemma}
        \begin{proof}
            Let $x^+ = T_{\tau^{-1}}(x)$.
            Then, by Definition \ref{def:exact-pg} it has: 
            \begin{align*}
                &\mathbf 0 \in \nabla f(x) - \tau(x - x^+) + \partial g(x^+)
                \\
                \underset{(1)}\iff &
                \mathbf 0 \in \partial\left(
                    z \mapsto g(z) + \langle \nabla f(x), z - x\rangle
                    + \frac{\tau}{2}\Vert x - z\Vert^2
                \right)(x^+)
                \\
                \underset{(2)}\iff &
                x^+ \in 
                \argmin_{x\in \RR^n}\left\lbrace
                    g(z) + \langle \nabla f(x), z - x\rangle + \frac{\tau}{2}\Vert z - x\Vert^2
                \right\rbrace. 
            \end{align*}
            At (1), we used the sum rule of subgradient; at (2) we used the subgradient calculus to deduce that $x^+$ is the minimizer of convex function $h(z) := g(z) + \langle \nabla f(x), z - x\rangle + \frac{\tau}{2}\Vert x - z\Vert^2$. 
            Therefore, substituting $x^+$ into $h(z) + f(x)$ yields: 
            \begin{align*}
                & g(x^+) + f(x) 
                + \langle \nabla f(x), x^+ - x\rangle
                + \frac{\tau}{2}\Vert x^+ - x\Vert^2
                \\
                &=
                \min_{z \in \RR^n}\left\lbrace
                    g(z) + f(x) 
                    + \langle \nabla f(x), z - x\rangle
                    + \frac{\tau}{2}\Vert z - x\Vert^2
                \right\rbrace
                \\
                &\underset{(3)}\le \min_{z \in \RR^n}\left\lbrace
                    g(z) + f(z) + \frac{\tau}{2}\Vert z - x\Vert^2
                \right\rbrace. 
            \end{align*}
            At (3), we used the fact that $f$ is convex which gives, for all $z\in \RR^n$ that $f(z) \ge f(x) + \langle \nabla f(x), z - x\rangle$.
        \end{proof}
        \par
        We define the proximal gradient descent method as follows.
        \begin{definition}[the proximal gradient descent]\;\label{def:pgd-with-qg}\\
            Let $(F, f, g, L, S, \kappa)$ satisfy Assumption \ref{ass:pgd-qg}. 
            Choose any initial guess $v_0 \in \RR^n$. 
            An algorithm is a proximal descent method if it generates iterates $(v_j)_{j \in \Z_+}$ satisfying for all $j \in \Z_+$:
            \begin{align*}
                & v_{j + 1} = \hprox_{\tau_j^{-1} g}(v_j - \tau_j^{-1}\nabla f(v_j)), 
                \\
                & D_f(v_{j + 1}, v_j) \le \frac{\tau_j}{2}\Vert v_{j + 1} - v_j\Vert^2.
            \end{align*}
            In addition, we assume that $(\tau_j)_{j \in \Z_+}$ is a bounded sequence, i.e., there exists $\bar \tau = \sup_{j \in \Z_+} \tau_j < \infty$.
        \end{definition}
        We now arrive at the first main result from the PGD introduced in Definition \ref{def:pgd-with-qg}.
        The following theorem shows that the iterates and the function value converge linearly. 
        In addition, they are bounded by the distance between the first initial guess to the set of minimizers. 
        \begin{theorem}[PGD converges linearly under quadratic growth]\;\label{thm:pgd-lin-cnvg}\\
            Let $(F, f, g, L, S, \kappa)$ satisfy Assumption \ref{ass:pgd-qg}.
            Suppose that iterates $(v_j)_{j \in \Z_+}$ and line search parameters $(\tau_j)_{j \in \Z_+}$ are given by Definition \ref{def:pgd-with-qg}. 
            Then the following are true: 
            \begin{enumerate}[nosep]
                \item \label{thm:pgd-lin-vng:item1} The iterates converge linearly: for all $j \in \Z_+$ it has
                \begin{align*}
                    \dist(v_{j} | S)^2 
                    \le
                    \left(
                        \prod_{n = 0}^{j - 1}\frac{1}{1 + \kappa/\tau_n}
                    \right)\dist(v_0 | S)^2
                    \le \left(
                        \frac{1}{1 + \kappa/\bar \tau}
                    \right)^{j}\dist(v_0 | S)^2. 
                \end{align*}
                Here, $\bar \tau = \sup_{j \in \Z_+} \tau_j < \infty$. 
                \item \label{thm:pgd-lin-vng:item2} The function value converges linearly: for all $j \in \Z_+$ it has
                \begin{align*}
                    F(v_{j + 1}) - F_{\min} &\le 
                    \frac{\tau_j}{2}\dist(v_j | S)^2 \le 
                    \frac{\tau_j}{2}\left(
                        \frac{1}{1 + \kappa/\bar \tau}
                    \right)^{j}\dist(v_0 | S)^2. 
                \end{align*}
            \end{enumerate}
        \end{theorem}
        \begin{proof}
            We now prove item \ref{thm:pgd-lin-vng:item1}.
            For all $j \in \Z_+$ by Definition \ref{def:pgd-with-qg}, we have $v_{j + 1} = \hprox_{\tau_j^{-1} g}(v_j - \tau_j^{-1}\nabla f(v_j))$.
            Therefore, using subgradient calculus it follows that: 
            \begin{align*}
                & v_{j + 1} \in \argmin_{z \in \RR^n}\left\lbrace
                    g(z) + \frac{\tau_j}{2}\Vert z - v_j + \tau_j^{-1}\nabla f(v_j)\Vert^2
                \right\rbrace
                \\
                \iff & 
                \mathbf 0 \in 
                \partial g(v_{j + 1}) + \tau_j(v_{j + 1} - v_j + \tau_j^{-1}\nabla f(v_j))
                \\
                \iff & \tau_jv_j - \nabla f(v_j) - \tau_jv_{j + 1} \in \partial g(v_{j + 1})
                \\
                \iff & \mathbf 0 \in \nabla f(v_j) - \tau_j(v_j - v_{j + 1}) + \partial g(v_{j + 1})
                \\
                \underset{\text{Def \ref{def:exact-pg}}}\iff\hspace{-0.1em} & 
                v_{j + 1} = T_{\tau_j}(v_j). 
            \end{align*}
            The above shows that $v_{j + 1} = T_{\tau_j}(v_j)$ where $T_{\tau_j}$ is from Definition \ref{def:exact-pg}. 
            Recall $S$ is the set of minimizer of $F$ (Assumption \ref{ass:pgd-qg}\ref{ass:pgd-qg:item3}).
            Therefore, we invoke Corollary \ref{corollary:exact-pgineq} with $x^+ = v_{j + 1}, x = v_j$ and $z = \Pi_S v_j$ which yields:
            \begin{align*}
                0 &\le F(\Pi_S(v_j)) - F(v_{j + 1}) + \frac{\tau_j}{2}\Vert v_j - \Pi_S v_j\Vert^2
                - \frac{\tau_j}{2}\Vert \Pi_S v_j - v_{j + 1}\Vert^2
                \\
                &\le F(\Pi_S(v_j)) - F(v_{j + 1}) 
                + \frac{\tau_j}{2}\Vert v_j - \Pi_S v_j\Vert^2
                - \frac{\tau_j}{2}\dist(v_{j + 1} | S)^2
                \\
                &= 
                F_{\min} - F(v_{j + 1}) 
                + \frac{\tau_j}{2}\dist(v_j | S)^2
                - \frac{\tau_j}{2}\dist(v_{j + 1} | S)^2
                \\
                &\underset{(1)}\le 
                - \frac{\kappa}{2}\dist(v_{j + 1} | S)^2
                + \frac{\tau_j}{2}\dist(v_j | S)^2
                - \frac{\tau_j}{2}\dist(v_{j + 1} | S)^2
                \\
                &\le 
                - \frac{\kappa + \tau_j}{2}\dist(v_{j + 1} | S)^2
                + \frac{\tau_j}{2}\dist(v_j | S)^2. 
            \end{align*}
            At (1), we invoked quadratic growth condition assumed in Assumption \ref{ass:pgd-qg}\ref{ass:pgd-qg:item4}. 
            Rearranging the above and algebraic simplifications yields: $\dist(v_{j + 1} | S)^2 \le \left(\frac{1}{1 + \kappa/\tau_j}\right)\dist(v_j | S)^2$.
            Unrolling it recursively, and using the fact $\bar \tau = \sup_{j \in \Z_+} \tau_j$ from Definition \ref{def:pgd-with-qg} yields:
            \begin{align*}
                \dist(v_{j + 1} | S)^2 &\le \left(
                    \prod_{n = 0}^j\frac{1}{1 + \kappa/\tau_n}
                \right)\dist(v_0 | S)^2 
                \le 
                \left(
                    \frac{1}{1 + \kappa/\bar \tau}
                \right)^{j + 1}\dist(v_0 | S)^2. 
            \end{align*}
            We note that the base case is when $j = -1$, and it is satisfied with $\dist(v_0 | S)^2 \le \dist(v_0 | S)^2$.
            \par
            We now prove item \ref{thm:pgd-lin-vng:item2}.
            To see the convergence of function value, consider for all $ j \in \Z_+$:
            \begin{align*}
                F(v_{j + 1}) &= f(v_{j + 1}) + g(v_{j + 1})
                \\
                &= g(v_{j + 1}) + f(v_{j}) + \langle \nabla f(v_j), v_{j + 1} - v_j\rangle
                + D_f(v_{j + 1}, v_j)
                \\
                &\underset{(2)}\le 
                g(v_{j + 1}) + f(v_{j}) + \langle \nabla f(v_j), v_{j + 1} - v_j\rangle
                + \frac{\tau_j}{2}\Vert v_{j + 1} - v_j\Vert^2
                \\
                &\underset{(3)} = \min_{z \in \RR^n}\left\lbrace
                    F(z) + \frac{\tau_j}{2}\Vert z - v_j\Vert^2
                \right\rbrace
                \\
                &\le F(\Pi_S v_j) + \frac{\tau_j}{2}\Vert \Pi_S v_j - v_j\Vert^2
                \\
                &= F_{\min} + \frac{\tau_j}{2}\dist(v_j | S)^2.
            \end{align*}
            At (2), we used Definition \ref{def:pgd-with-qg}, which states the line search conditions that state $D_f(v_{j + 1}, v_j) \le \frac{\tau_j}{2}\Vert v_{j + 1} - v_j\Vert^2$.
            At (3), we invoked Lemma \ref{lemma:pg-env-upbnd} because $v_{j + 1} = T_{\tau_j}(v_j)$ for all $j \in \Z_+$. 
            Using item \ref{thm:pgd-lin-vng:item1}, we obtain: 
            \begin{align*}
                F(v_{j + 1}) - F_{\min} &\le 
                \frac{\tau_j}{2}\dist(v_j | S)^2 \le 
                \frac{\tau_j}{2}\left(
                    \frac{1}{1 + \kappa/\bar \tau}
                \right)^{j}\dist(v_0 | S)^2. 
            \end{align*}
        \end{proof}
        \begin{remark}
            This is not a new result, and it has been established. 
            See for example, Necoara et al. \cite[Theorem 12]{necoara_linear_2019}. 
            The difference here is that the proof has been adapted into our context and assumptions for better exposition.
        \end{remark}

    \subsection{In preparations for linear convergence of the inner loop}\label{ssec:conds-lin-cnvg-pp}
        Continuing from Section \ref{ssec:optz-inxt-pp-problem}, in this subsection we define the algorithms used in the inner loop (Definition \ref{def:pgd-inner-lp}). 
        We characterize sufficient conditions that attain linear convergence for the algorithm in the inner loop in Assumption \ref{ass:lin-cnvg-for-pp} below. 
        \begin{assumption}[conditions for linear convergence of proximal problem]\;\label{ass:lin-cnvg-for-pp}\\
            Fix $y \in \RR^n$, $\lambda > 0$.
            Let $h_\lambda(x) := \frac{1}{2\lambda}\Vert \lambda x - y\Vert^2 - \frac{1}{2\lambda}\Vert y\Vert^2$. 
            Recall the dual objective $\Psi_\lambda(v) = h_\lambda(A^\top v) + \omega^\star(v)$, see \eqref{eqn:dual-pp}, and the primal objective $\Phi_\lambda(z) = \omega(A z) + \frac{1}{2\lambda}\Vert z - y\Vert^2$, see \eqref{eqn:primal-pp}.
            We assume that $(\omega, A, y, \lambda, h_\lambda, \Phi_\lambda, \Psi_\lambda, \kappa_\lambda)$ satisfy the following.
            \begin{enumerate}[nosep]
                \item\label{ass:lin-cnvg-for-pp:item2} $\omega, A$ satisfy Assumption \ref{ass:for-inxt-prox}, 
                meaing that $\omega$ is $L_\omega$-Lipschitz continuous. Equivalently $\omega^\star$ has bounded domain.
                \item\label{ass:lin-cnvg-for-pp:item3} There exists $\emptyset \neq S \subseteq \RR^m$ such that $(\Psi_\lambda, h_\lambda, \omega^\star, \lambda\Vert A^\top A\Vert, S, \kappa_\lambda)$ satisfy Assumption \ref{ass:pgd-qg}.     
            \end{enumerate}
        \end{assumption}
        \begin{remark}
            Recall from previous section and take note that Assumption \ref{ass:lin-cnvg-for-pp}\ref{ass:lin-cnvg-for-pp:item3} is saying that $\Psi_\lambda$ is a composite optimization problem with $\lambda \Vert A^\top A\Vert$-smooth part $h_\lambda\circ A^\top$, and nonsmooth part $\omega^\star$ satisfying the quadratic growth condition with $\kappa_\lambda$. 
        \end{remark}
        \par
        The following definition specifies the algorithm that achieves a linear convergence rate with the assumptions above.
        \begin{definition}[proximal gradient descent inner loop]\;\label{def:pgd-inner-lp}\\
            Let $\lambda > 0, \epsilon > 0$, and $(\omega, A, y, \lambda, h_\lambda, \Phi_\lambda, \Psi_\lambda, \kappa_\lambda)$ satisfy Assumption \ref{ass:lin-cnvg-for-pp}.
            Let initial guess $v_0\in \dom \omega^\star$ be feasible, let $z_0 = y - \lambda A^\top v_0$. 
            In Assumption \ref{ass:lin-cnvg-for-pp}, note that $\Psi_\lambda = h_\lambda\circ A^\top  + \omega^\star$ where $h_\lambda = x \mapsto \frac{1}{2\lambda}\Vert \lambda x - y\Vert^2$.
            We define an algorithm to be the inner loop algorithm if it generates iterates $(z_j, v_j)_{j \in \Z_+}$, and line search constants $(\tau_j)_{j \in \Z_+}$ such that for all $j \in \Z_+$: 
            \begin{align}
                & v_{j + 1} = \hprox_{\tau_j^{-1} \omega^\star}\left(
                    v_{j} - \tau_j^{-1}A\nabla h_\lambda(A^\top v_j)
                \right), \label{def:pgd-inner-lp-vj}
                \\
                & D_{h_\lambda}(v_{j + 1}, v_j) \le \frac{\tau_j}{2}\Vert v_{j + 1} - v_j\Vert^2, 
                \label{def:pgd-inner-lp-ls}
                \\
                & z_{j + 1} = y - \lambda A^\top v_{j + 1}. 
                \label{def:pgd-inner-lp-zj}
            \end{align}
            In addition, assume $(\tau_j)_{j\in \Z_+}$ is bounded by $\bar \tau_\lambda$.
            Recall $\mathbf G_\lambda$ from \eqref{eqn:duality-gap-pp}. 
            Finally, the algorithm outputs $z_j$ such that $j := \min\{t \in \Z_+ : \mathbf G_\lambda(z_t, v_t) \le \epsilon\}$. 
        \end{definition}
        \begin{remark}\label{remark:pgd-inner-lp}
            The value of $\mathbf G_\lambda(z_j, v_j)$ is easy to compute because it only requires access to matrix $A, A^\top$, and the function $\omega$. 
            In case when the proximal operator for $\omega^\star$ is nontrivial, we can apply the Moreau identity and calculate instead the proximal operator of $\omega$. 
            The gradient for $h_\lambda\circ A^\top$ is easy to compute, and it is: $\lambda AA^\top v - Ay$.
            The Bregman divergence for descent lemma is easy to compute too, and it is given by $D_{h_\lambda}(v_{j + 1}, v_j) = (\lambda/2)\Vert A^\top(v_{j + 1} - v_j)\Vert^2$.
        \end{remark}

    \subsection{Linear convergence of the inner loop}\label{ssec:lin-cnvg-inner-lp}
        Continuing from the previous subsections, in this subsection we are ready to present our major result
        which is the following theorem.
        The theorem states that a complexity of $\mathcal O(\ln(\epsilon^{-1}))$ is possible under Assumption \ref{ass:lin-cnvg-for-pp} for algorithm that satisfies Definition \ref{def:pgd-inner-lp}. 
        \begin{theorem}[linear convergence of the inner loop]\;\label{thm:inn-loop-lin-cnvg}\\
            Let the parameters $(\omega, A, y, \lambda, h_\lambda, \Phi_\lambda, \Psi_\lambda, \kappa_\lambda)$ of a proximal problem satisfy Assumption \ref{ass:lin-cnvg-for-pp}. 
            Let iterates $(z_j, v_j)_{j \in \Z_+}$, line search sequence $(\tau_j)_{j \in \Z_+}$, and its upper bound $\bar \tau_\lambda$ be given by Definition \ref{def:pgd-inner-lp}.
            Let $\bar v$ be a minimizer of $\Psi_\lambda$. 
            We denote the following quantities for short: 
            \begin{enumerate}[nosep,label=(\Roman*),font=\upshape]
                \item\label{thm:inn-loop-lin-cnvg:cnst-psi} $C_\Psi :=\diam(\dom \omega^\star)$, 
                \item $\Delta_{\lambda, j} = \Psi_\lambda(v_j) - \Psi_\lambda(\bar v)$, 
                \item\label{thm:inn-loop-lin-cnvg:cnst-lambda} 
                $
                    C_\lambda = C_\Psi\left(
                        2\sqrt{\lambda\bar \tau_\lambda}
                        K_\omega\Vert A\Vert
                        + \bar\tau_\lambda C_\Psi/2
                    \right). 
                $
            \end{enumerate}
            Then the following are true:
            \begin{enumerate}[nosep]
                \item\label{thm:inn-loop-lin-cnvg:item1} We have for all $j \in \Z+$, and $C_\Psi < \infty$, $\Delta_{\lambda, j}$ converges linearly to zero: 
                \begin{align*}
                    \Delta_{\lambda, j + 1}
                    &\le \frac{\bar \tau_\lambda}{2}\left(
                        \frac{1}{1 + \kappa_\lambda/\bar \tau_\lambda}
                    \right)^{j}C_\Psi^2. 
                \end{align*}
                \item\label{thm:inn-loop-lin-cnvg:item2} The duality gap $\mathbf G_\lambda$ from \eqref{eqn:duality-gap-pp} converges to zero linearly. 
                The following holds for all $j \in \N$:
                \begin{align*}
                    \mathbf G_\lambda(z_j, v_j)
                    \le
                    \left(
                    \frac{1}{1 + \kappa_\lambda/\bar \tau_\lambda}
                    \right)^{\frac{j - 1}{2}} C_\lambda. 
                \end{align*}
                \item\label{thm:inn-loop-lin-cnvg:item3} 
                For all $\epsilon > 0$, if $(j - 1)/2 \ge \max(0, \ln(C_\lambda/\epsilon)/\ln(1 + \kappa_\lambda/\bar \tau_\lambda))$, then $\mathbf G_\lambda(z_j, v_j) \le \epsilon$, and $z_j \approx_\epsilon \hprox_{\lambda \omega\circ A}(y)$.
            \end{enumerate}
        \end{theorem}
        \begin{proof}
            We now show item \ref{thm:inn-loop-lin-cnvg:item1}. 
            By Definition \ref{def:pgd-inner-lp}, for all $j \in \Z_+$, $(v_j)_{j \in \Z_+}$ satisfies $v_{j + 1} = \hprox_{\tau_j^{-1}\omega^\star}(v_j - \tau_j^{-1} A\nabla h_\lambda(A^\top v_j))$, and Definition \ref{def:pgd-with-qg}. 
            Here, it minimizes the smooth and nonsmooth composite objective $\Psi_\lambda(v) = h_\lambda(A^\top v) + \omega^\star(v)$ where $h_\lambda = v \mapsto \frac{1}{2\lambda}\left\Vert \lambda v - y\right\Vert^2$. 
            In Assumption \ref{ass:lin-cnvg-for-pp}\ref{ass:lin-cnvg-for-pp:item3}, we stated that $\Psi_\lambda$ satisfies Assumption \ref{ass:pgd-qg} which means two things. 
            First, it means $\Psi_\lambda$ admits a set of minimizers which we denote by $S$, and set $\bar v \in S$.
            Secondly, it means the results of convergence of PGD in Theorem \ref{thm:pgd-lin-cnvg}\ref{thm:pgd-lin-vng:item2} apply, and therefore for all $j \in \Z_+$:
            \begin{align*}
                \Psi_\lambda(v_{j + 1}) - \Psi_\lambda(\bar v) &\le 
                \frac{\tau_j}{2}\left(
                    \frac{1}{1 + \kappa_\lambda/\bar \tau_\lambda}
                \right)^{j}\dist(v_0 | S)^2
                \underset{(1)}\le 
                \frac{\bar \tau_\lambda}{2}\left(
                    \frac{1}{1 + \kappa_\lambda/\bar \tau_\lambda}
                \right)^{j}C_\Psi^2. 
            \end{align*}
            At (1), we used two results. 
            Firstly, we have $S \subseteq \dom \Psi_\lambda = \dom \omega^\star$.
            Then, from Definition \ref{def:pgd-inner-lp} we have $v_0 \in \dom\Psi_\lambda$.
            Therefore, $\dist(v_0 | S) \le C_\Psi = \diam(\dom \omega^\star)$.
            From Assumption \ref{ass:for-inxt-prox}\ref{ass:for-inxt-prox:item4}, it states that $\dom \omega^\star$ is bounded.
            Therefore, we have $\dist(v_0 | S) \le C_\Psi = \diam(\dom \omega^\star) < \infty$.
            Secondly, $\bar \tau_\lambda$ in Definition \ref{def:pgd-inner-lp} produces $\tau_j \le \sup_{j \in \Z_+} \tau_j = \bar \tau_\lambda$. 
            \par
            Now, we prove item \ref{thm:inn-loop-lin-cnvg:item2}. 
            For better notation, recall that $\Delta_{\lambda, j} = \Psi_\lambda(v_j) - \Psi_\lambda(\bar v)$ for all $j \in \Z_+$.
            From Theorem \ref{thm:minimizing-dual-pp}\ref{thm:minimizing-dual-pp:result1}, the primal $\Psi_\lambda$ admits the minimizer $\bar z = y - \lambda A^\top \bar v$.
            Recall duality gap $\mathbf G_\lambda$ from \eqref{eqn:duality-gap-pp} attains zero, i.e., $\mathbf G_\lambda(\bar z, \bar v) = 0 = \Phi_\lambda(\bar z) + \Psi_\lambda(\bar v)$.
            Therefore, we have for all $j \in \N$ the duality gap $\mathbf G_\lambda$: 
            {\allowdisplaybreaks 
            \begin{align*}
                \mathbf G_\lambda(z_j, v_j) &= \Phi_\lambda(z_j) + \Psi_\lambda(v_j) + 0
                \\
                &= \Phi_\lambda(z_j) - \Phi_\lambda(\bar z) + \Psi_\lambda(v_j) - \Psi_\lambda(\bar v)
                \\
                &= \Phi_\lambda(z_j) - \Phi_\lambda(\bar z) + \Delta_{\lambda, j}
                \\
                &\underset{(2)}\le 
                \sqrt{\Delta_{\lambda, j}}\left(
                    2\sqrt{2\lambda}K_\omega\Vert A\Vert
                    + \sqrt{\Delta_{\lambda, j}} 
                \right)
                + \Delta_{\lambda, j}
                \\
                &= 
                \sqrt{\Delta_{\lambda, j}}\left(
                    2\sqrt{2\lambda}K_\omega\Vert A\Vert
                    + 2\sqrt{\Delta_{\lambda, j}} 
                \right)
                \\
                &\underset{\text{\ref{thm:inn-loop-lin-cnvg:item1}}}\le
                \sqrt{\frac{\bar \tau_\lambda}{2}}\left(
                    \frac{1}{1 + \kappa_\lambda/\bar \tau_\lambda}
                \right)^{\frac{j - 1}{2}}C_\Psi\left(
                    2\sqrt{2\lambda}K_\omega\Vert A\Vert
                    + 2\sqrt{\Delta_{\lambda, j}} 
                \right)
                \\
                &\underset{(3)}\le
                \sqrt{\frac{\bar \tau_\lambda}{2}}
                \left(
                    \frac{1}{1 + \kappa_\lambda/\bar \tau_\lambda}
                \right)^{\frac{j - 1}{2}}C_\Psi\left(
                    2\sqrt{2\lambda}K_\omega\Vert A\Vert
                    + \frac{\sqrt{2\bar \tau_\lambda}C_\Psi}{2}
                \right)
                \\
                &= 
                \left(
                    \frac{1}{1 + \kappa_\lambda/\bar \tau_\lambda}
                \right)^{\frac{j - 1}{2}}
                C_\Psi\left(
                    2\sqrt{\lambda\bar \tau_\lambda}K_\omega\Vert A\Vert
                    + \frac{\bar\tau_\lambda C_\Psi}{2}
                \right). 
            \end{align*}
            }\noindent
            At (2), we invoked Theorem \ref{thm:minimizing-dual-pp}\ref{thm:minimizing-dual-pp:result3} on $\Phi_\lambda(z_j) - \Phi_\lambda(\bar z)$. 
            At (3) we invoked previous item \ref{thm:inn-loop-lin-cnvg:item1} to bound $2\sqrt{\Delta_{\lambda, j}}$ giving, for all $j\in \N$, the inequality:
            \begin{align*}
                \Delta_{\lambda, j} \le \frac{\overline \tau_\lambda}{2}\left(
                    \frac{1}{1 + \kappa_\lambda/\bar\tau_\lambda}
                \right)^{j - 1}C_\Psi^2 
                \le \frac{\bar \tau_\lambda}{2}C_\Psi^2. 
            \end{align*}
            Therefore, the above gives $2\sqrt{\Delta_{\lambda, j}}\le \frac{\sqrt{2\bar \tau_\lambda}C_\Psi}{2}$.
            \par 
            We now show \ref{thm:inn-loop-lin-cnvg:item3}.
            To argue $\mathbf G_\lambda(z_j, v_j) \le \epsilon$, we use item \ref{thm:inn-loop-lin-cnvg:item2}.
            Then, it suffices to show that $(j - 1)/2 \ge \max(0, \ln(C_\lambda/\epsilon)/\ln(1 + \kappa_\lambda/\bar \tau_\lambda))$ implies $C_\lambda \left(\frac{1}{1 + \kappa_\lambda/\bar \tau_\lambda}\right)^{\frac{j - 1}{2}} \le \epsilon$.
            Consider
            {\allowdisplaybreaks
            \begin{align*}
                C_\lambda\left(
                    \frac{1}{1 + \kappa_\lambda/\bar \tau_\lambda}
                \right)^{\frac{j - 1}{2}}
                &= C_\lambda \exp\left(
                    \frac{j - 1}{2}\ln \left(
                        \frac{1}{1 + \kappa_\lambda/\bar \tau_\lambda}
                    \right)
                \right)
                \\
                &= 
                C_\lambda \exp\left(
                    \frac{1 - j}{2}\ln \left(
                        1 + \kappa_\lambda/\bar \tau_\lambda
                    \right)
                \right)
                \\
                &\underset{(4)}\le
                C_\lambda \exp\left(
                    \min\left(
                        0, 
                        \frac{-\ln(C_\lambda/\epsilon)}{\ln(1 + \kappa_\lambda/\bar \tau_\lambda)}
                    \right)
                    \ln \left(
                        1 + \kappa_\lambda/\bar \tau_\lambda
                    \right)
                \right)
                \\
                &= 
                C_\lambda \exp\left(
                    \min\left(
                        0, -\ln(C_\lambda/\epsilon)
                    \right)
                \right)
                \\
                &= C_\lambda\min(1, \epsilon/C_\lambda) \le \epsilon. 
            \end{align*}
            }
            At (4), we used: 
            \begin{align*}
                \frac{j - 1}{2} \ge \max\left(0, 
                    \frac{\ln(C_\lambda/\epsilon)}{\ln(1 + \kappa_\lambda/\bar \tau_\lambda)}
                \right) \iff 
                \frac{1 - j}{2} \le 
                \min\left(
                    0, - \frac{\ln(C_\lambda/\epsilon)}{\ln(1 + \kappa_\lambda/\bar \tau_\lambda)}
                \right). 
            \end{align*}
            Then, we substitute the above upper bound for $(1 - j)/2$. 
            Since $\mathbf G_\lambda(z_j, v_j) \le \epsilon$, by Lemma \ref{lemma:dlty-gap-inxt-pp}, we have $z_j \approx_\epsilon \hprox_{\lambda \omega\circ A}(y)$.
        \end{proof}
        \begin{remark}
            To choose a feasible $v_0$ in practice, one may apply the $\hprox_{\omega^\star}$ operator.
        \end{remark}

\section{Total oracle complexity of the algorithm}\label{sec:total-cmplx}
    In this section, we present the main result of our paper.
    It states that the total number of iterations of the inner loop needed to achieve $F(x_k) - F(\bar x) \le \varepsilon$ is bounded by $\mathcal O(\varepsilon^{-1/2}\ln (\varepsilon^{-1}))$.
    Similarly, we also show that the total number of iterations of the inner loop needed to achieve $\dist(\mathbf 0 | \partial_{\epsilon_k}F(x_k)) \le \varepsilon$ is bounded by $\mathcal O(\varepsilon^{-1}\ln(\varepsilon^{-1}))$.
    \par
    To this end, Section \ref{ssec:g-bnded-ilp-cmplx} is dedicated to showing that the complexity of the inner loop is bounded globally for all iterations of the outer loop.
    Then, in Section \ref{ssec:innlp-total-cmplx} we present Theorem \ref{thm:inn-lp-overall-cmplx} which is the major result. 
    \subsection{Globally bounded inner loop complexity}\label{ssec:g-bnded-ilp-cmplx}
        Note that the inner loop solves a different optimization problem at each iteration of the outer loop.
        Therefore, even if the inner loop has a linear convergence rate, it does not mean that it converges at the same rate on each iteration of the outer loop.
        More precisely, observe that $C_\lambda$ in Theorem \ref{thm:inn-loop-lin-cnvg}\ref{thm:inn-loop-lin-cnvg:cnst-lambda} depends on $\lambda$ which changes depending on $L_k$ (Assumption \ref{ass:opt-mmntm-seq}) from the outer loop.
        \par
        In this section we address the concern and show that under two mild assumptions on the dual problem $\Psi_\lambda$, and the inner loop algorithm. 
        We derive inner loop linear convergence independent of parameter $\lambda$ from the outer loop.
        We refer to this property of the complexity of the inner loop as ``globally bounded". 
        \par
        Let $(F, f, g, L)$, $(\alpha_k, B_k, \rho_k,\epsilon_k)_{k \in \Z_+}$ satisfy Definition \ref{def:inxt-apg}. 
        Fix any $k \in \Z_+$ to be the iteration counter of the outer loop. 
        Let $(g, \omega, A)$ satisfy Assumption \ref{ass:for-inxt-prox}.
        Take $L_k = B_k + \rho_k$ as given by Assumption \ref{ass:opt-mmntm-seq}. 
        In this case, the inner loop finds $x \approx_{\epsilon_k} T_{L_k}(y_k)$ by evaluating the equivalent inexact proximal point problem (Lemma \ref{lemma:other-repr-inxt-pg}): 
        \begin{align*}
            x_k \approx_{\epsilon_k} \hprox_{L_k^{-1} g}(y_k - L_k^{-1}\nabla f(y_k)).
        \end{align*}
        Let $\lambda^{(k)} := L_k^{-1}, \tilde y_k := y_k - L_k^{-1}\nabla f(y_k)$.
        Then, the proximal problem $\Phi_{\lambda^{(k)}}$ from \eqref{eqn:primal-pp} is: 
        \begin{align}\label{eqn:primal-pp-k}
            \Phi_{\lambda^{(k)}}(u) &:= \frac{1}{2\lambda^{(k)}} \Vert u - \tilde y_k \Vert^2 + \omega(A u). 
        \end{align}
        From \eqref{eqn:dual-pp}, the dual becomes: 
        \begin{align}\label{eqn:dual-pp-k}
            \Psi_{\lambda^{(k)}}(v) &:=
            \frac{1}{2\lambda^{(k)}}
            \left\Vert \lambda^{(k)}A^\top v - \tilde y_k
            \right\Vert^2
            + \omega^\star(v) 
            - \frac{1}{2\lambda^{(k)}}\Vert \tilde y_k\Vert^2. 
        \end{align}
        Finally, the duality gap is: $\mathbf G_{\lambda^{(k)}}(u, v) = \Phi_{\lambda^{(k)}}(u) + \Psi_{\lambda^{(k)}}(v)$. 
        All the above forms the primal and dual of the proximal problem \eqref{eqn:primal-pp} defined in Section \ref{ssec:optz-inxt-pp-problem}.
        However, the parameters $\lambda, y$ are instead $\lambda^{(k)}, \tilde y_k$ which changes depending on $k \in \Z_+$. 
        \par
        To show that the complexity is bounded globally across all iterations, we assume the following. 
        \begin{assumption}[globally bounded inner loop complexity]\;\label{ass:inn-cmplx}\\ 
            Let $\rho_k, B_k, L_k$ be given by Definition \ref{def:inxt-apg}.
            Define $\lambda^{(k)} = L^{-1}_k$. \\
            For all $k \in \Z_+$, let iterates $\left(z_j^{(k)}, v_j^{(k)}\right)_{j \in \Z_+}$, and $\bar \tau_{\lambda^{(k)}}$ be given by Definition \ref{def:pgd-inner-lp}.
            We assume the parameters of the outer loop $\left(\epsilon_k, \lambda^{(k)}\right)$, and the parameters of the inner loop $\left(\omega, A, \tilde y_k, \lambda^{(k)}, h_{\lambda^{(k)}}, \Phi_{\lambda^{(k)}}, \Psi_{\lambda^{(k)}}, \kappa_{\lambda^{(k)}}, \bar \tau_{\lambda^{(k)}}\right)$ satisfy the following. 
            \begin{enumerate}[nosep]
                \item\label{ass:inn-cmplx:item1} $(\epsilon_k)_{k \in \Z_+}$ satisfies Assumption \ref{ass:opt-mmntm-seq}\ref{ass:opt-mmntm-seq:item3}, and $(L_k)_{k \in \Z_+}$ satisfies Assumption \ref{ass:opt-mmntm-seq}\ref{ass:opt-mmntm-seq:item2}.
                \item\label{ass:inn-cmplx:item2} For all $k \in \Z_+$, the set of parameters $\left(\omega, A, \tilde y_k, \lambda^{(k)}, h_{\lambda^{(k)}}, \Phi_{\lambda^{(k)}}, \Psi_{\lambda^{(k)}}, \kappa_{\lambda^{(k)}}\right)$ satisfies Assumption \ref{ass:lin-cnvg-for-pp}. 
                \item\label{ass:inn-cmplx:item3} There exists $\kappa_{\min} > 0$ such that $\inf_{k \in \Z_+}\kappa_{\lambda^{(k)}} > \kappa_{\min}$. 
                \item\label{ass:inn-cmplx:item4} There exists $\bar \tau_{\max} \in \RR$ such that $\sup_{k \in \Z_+}\bar \tau_{\lambda^{(k)}} \le \bar \tau_{\max}$. 
            \end{enumerate}
        \end{assumption}
        \begin{remark}
            In the above assumption, item \ref{ass:inn-cmplx:item1} implies $\{\lambda^{(k)}\}_{k\in \Z_+}\subseteq [L_{\max}^{-1}, L_{\min}^{-1}]$ where $0 < L_{\min} \le L_{\max}$;
            Item \ref{ass:inn-cmplx:item3} says that there exists $\kappa_{\min} > 0$ such that for all $\lambda \in \left[L_{\max}^{-1}, L_{\min}^{-1}\right]$, $\Psi_\lambda$ satisfies quadratic growth condition with $\kappa_{\lambda}$ such that $\kappa_\lambda \ge \kappa_{\min} > 0$;
            Finally, item \ref{ass:inn-cmplx:item4} says that the line search constant from Definition \ref{def:pgd-inner-lp} is bounded above.  
            Finally, $\bar \tau_{\max}$ exists by virtue of the fact that $h_\lambda$ is a Lipschitz-smooth function.
            However, the existence of $\kappa_{\min}$ is harder to verify in general. 
            Nonetheless, we show that $\kappa_{\min}$ exists for conic polyhedral $\omega$, see Section \ref{ssec:example-fxn-lin-cnvg-lp}.
        \end{remark}
        
        \begin{proposition}[inner loop complexity can be bounded globally]\;\label{prop:inner-lp-cmplx}\\
            Let $(\epsilon_k, \lambda_k, \kappa_{\lambda^{(k)}}, \bar \tau_{\lambda^{(k)}})_{k\in\Z_+}$, $\left(z_j^{(k)}, v_j^{(k)}\right)_{j \in \Z_+, k \in \Z_+}$ and $\left(\lambda^{(k)}, \epsilon_k\right)_{k \in \Z_+}$ satisfy Assumption \ref{ass:inn-cmplx}. 
            Denote $J_k \in \Z_+$ be the smallest integer such that $\mathbf G_{\lambda^{(k)}}\left(z_{J_k}^{(k)}, v_{J_k}^{(k)}\right) \le \epsilon_k$. 
            Take $C_{\Psi}$ from Theorem \ref{thm:inn-loop-lin-cnvg}\ref{thm:inn-loop-lin-cnvg:cnst-psi}, $K_\omega$ in Assumption \ref{ass:for-inxt-prox}\ref{ass:for-inxt-prox:item4}.
            Let $\kappa_{\min}, \bar \tau_{\max}$ be given by Assumption \ref{ass:inn-cmplx}\ref{ass:inn-cmplx:item3}, \ref{ass:inn-cmplx:item4}. 
            Define: 
            \begin{enumerate}[nosep, label=(\Roman*), font=\upshape]
                \item\label{prop:inner-lp-cmplx:cnst1} 
                $C_\lambda^{\max} := C_{\Psi}\left(2 \Vert A\Vert K_\omega\sqrt{\frac{\bar \tau_{\max}}{L_{\min}}} + (1/2)\bar \tau_{\max} C_{\Psi}\right)$,
                \item\label{prop:inner-lp-cmplx:cnst2} $C_4 := 1 + \frac{\kappa_{\min}}{\bar \tau_{\max}}$, 
                \item\label{prop:inner-lp-cmplx:cnst3} $C_5 := C_\lambda^{\max}C_4^{1/2}\mathcal E_0^{-1}$. 
            \end{enumerate}
            Then, the following holds for all $k \in \Z_+$:
            \begin{align}\label{prop:inner-lp-cmplx:rslt1}
                \bar J_k := \max_{i = 0, \ldots, k}J_i 
                &\le \max\left(
                    1, \frac{2\ln(C_5)}{\ln(C_4)}, \frac{
                        2(2 + p)\ln\left(k(4C_5)^{\frac{1}{2 + p}}\right)
                    }{\ln(C_4)}
                \right). 
            \end{align}
        \end{proposition}
        \begin{proof}
            Assumption \ref{ass:inn-cmplx} includes Assumption \ref{ass:lin-cnvg-for-pp}. 
            Therefore, we apply Theorem \ref{thm:inn-loop-lin-cnvg}\ref{thm:inn-loop-lin-cnvg:item3}, which means for all $k \in \Z_+$:
            \begin{align}\label{prop:inner-lp-cmplx-pitem1}\begin{split}
                \bar J_k &=
                \max_{i=0, \ldots, k}\max\left(
                    1, \frac{
                        2\ln\left(
                            \epsilon_i^{-1}C_{\lambda^{(i)}}
                        \right)
                        }{
                            \ln\left(
                                1 + \frac{\kappa_{\lambda^{(i)}}}{\bar \tau_{\lambda^{(i)}}}
                            \right)
                        }
                        + 1
                \right)
                \\
                &\underset{(1)}\le
                \max\left(
                    1, \max_{i=0, \ldots, k}
                    \left(
                        \frac{
                        2\ln\left(
                            \epsilon_i^{-1}
                            C_{\lambda^{(i)}}
                        \right)
                        }{\ln\left(C_4\right)}
                    \right)
                    + 1
                \right)
                \\
                &\underset{(2)}\le 
                \max\left(
                    1, \max_{i=0, \ldots, k}
                    \left(
                        \frac{
                        2\ln\left(
                            \epsilon_i^{-1}
                            C_{\lambda}^{\max}
                        \right)
                        }{\ln\left(C_4\right)}
                    \right) + 1
                \right)
                \\
                &= \max\left(
                    1, 
                    \frac{2\ln\left(\left(
                                \max_{i=0, \ldots, k}\epsilon_i^{-1}
                            \right)
                            C_{\lambda}^{\max}
                        \right) + 2\ln(\sqrt{C_4})
                        }{\ln\left(C_4\right)}
                    \right)
                \\
                &= 
                \max\left(
                    1, \frac{
                    2\ln\left(
                        C_{\lambda}^{\max}
                        \sqrt{C_4}\left(
                            \max_{i=0, \ldots, k}\epsilon_i^{-1}
                        \right)
                    \right)
                    }{\ln\left(C_4\right)}
                \right). 
            \end{split}\end{align} 
            At (1), we used Assumption \ref{ass:inn-cmplx}\ref{ass:inn-cmplx:item3}, \ref{ass:inn-cmplx:item4}. 
            This gives $C_4 = 1 + \frac{\kappa_{\min}}{\bar \tau_{\max}} \le \inf_{i \in \Z_+}1 + \kappa_{\lambda^{(i)}}/\bar \tau_{\lambda^{(i)}}$.
            At (2), recall $C_\lambda$ from Theorem \ref{thm:inn-loop-lin-cnvg}, $\lambda^{(i)} = L_i^{-1}$ from Assumption \ref{ass:inn-cmplx}.
            Then, it follows from Assumption \ref{ass:inn-cmplx}\ref{ass:inn-cmplx:item1} that $C_{\lambda^{(i)}}$ admits:
            \begin{align*}
                \sup_{i \in \Z_+}C_{\lambda^{(i)}} &= 
                \sup_{i \in \Z_+}\left\lbrace
                    C_{\Psi}\left(
                        2\sqrt{\lambda^{(i)}\bar \tau_{\lambda^{(i)}}}
                        K_\omega \Vert A\Vert + \bar \tau_{\lambda^{(i)}} C_{\Psi}/2
                    \right)
                \right\rbrace
                \\
                &\le 
                C_{\Psi}\left(
                    2\sqrt{\sup_{i \in \Z_+}L_i^{-1}\bar \tau_{\lambda^{(i)}}}
                    K_\omega \Vert A\Vert + \sup_{i \in \Z_+}\bar \tau_{\lambda^{(i)}} C_{\Psi}/2
                \right)
                \\
                &\le 
                C_{\Psi}\left(
                    2\sqrt{
                        L_{\min}^{-1}\bar \tau_{\max}
                    }K_\omega \Vert A\Vert 
                    + \frac{\bar \tau_{\max} C_{\Psi}}{2}
                \right)
                \\
                &= 
                C_{\Psi}\left(
                    2K_\omega \Vert A\Vert\sqrt{
                        \frac{\bar \tau_{\max}}{L_{\min}}
                    } + \frac{\bar \tau_{\max} C_{\Psi}}{2}
                \right) = C_{\lambda}^{\max}. 
            \end{align*}
            Next, we continue simplifying \eqref{prop:inner-lp-cmplx-pitem1} by bounding one of its terms, which gives::
            \begin{align}\label{prop:inner-lp-cmplx-pitem2}\begin{split}
                &2\ln\left(
                    C_\lambda^{\max}\sqrt{C_4}\left(
                        \max_{i=0, \ldots, k}\epsilon_i^{-1}
                    \right)
                \right)
                \\
                &\underset{(3)}\le 2\ln \left(
                    C_\lambda^{\max} \sqrt{C_4} \mathcal E_0^{-1}\max\left(1, 4k^{2+p}\right)
                \right)
                \\
                &\underset{(4)}= 2\max\left(
                    \ln\left(
                        C_\lambda^{\max} \sqrt{C_4} \mathcal E_0^{-1}
                    \right), 
                    \ln\left(
                        C_\lambda^{\max} \sqrt{C_4} \mathcal E_0^{-1}4k^{2+p}
                    \right)
                \right)
                \\
                &\underset{(5)}= 
                2\max\left(
                    \ln\left(C_5\right), \ln\left(C_54k^{2+p}\right)
                \right)
                \\
                &= 2\max\left(
                    \ln\left(
                        C_5\right), (2 + p)\ln\left(k(4C_5)^{\frac{1}{2 + p}}\right)
                \right). 
            \end{split}
            \end{align}
            At (3), we used results from Lemma \ref{lemma:err-schedule-lbnd} which states: $\epsilon_i^{-1} \le \mathcal E_0^{-1}\max(1, 4i^{2+p})$ for all $i \in \Z_+$, implying that $\max_{i=0, \ldots, k}\le \mathcal E_0^{-1}\max(1, 4k^{2+p})$.
            At (4), we take the $\max$ out of $\ln$, observe that when $k = 0$, $\max(1, 4k^{2 + p}) = 1$; for all $k \in \N$, $\max(1, 4k^{2 + p}) = 4k^{2 + p}$.
            At (5), we made the substitution $C_5 = C_\lambda^{\max}C_4^{1/2}\mathcal E_0^{-1}$. 
            Finally, substituting the result from \eqref{prop:inner-lp-cmplx-pitem2} into \eqref{prop:inner-lp-cmplx-pitem1}, yields the desired result:
            \begin{align*}
                \bar J_k &\le \max\left(
                    1, \frac{2\ln(C_5)}{\ln(C_4)}, 2(2 + p)\frac{
                        \ln\left(k(4C_5)^{\frac{1}{2 + p}}\right)
                    }{\ln(C_4)}
                \right). 
            \end{align*}
        \end{proof}
        \begin{remark}
            Here, we point out the fact that the upper bound of iterative complexity on the inner loop for the first $k$ iterations of the outer loop depend on $\bar \tau_{\max}$, $\kappa_{\min}$, and $L_{\min}$.
            In practice, $\bar \tau_{\max}, L_{\min}$ can be known quite easily given the implementations of line search procedures. 
            However, the value of $\kappa_{\min}$ would require more theoretical exploration; it would heavily depend on the class of functions to which $\omega$ belongs.
            For example, we show in Section \ref{ssec:example-fxn-lin-cnvg-lp} that $\kappa_{\min}$ exists if $\omega$ is a conic polyhedral function.
        \end{remark}
    
    \subsection{Overall complexity}\label{ssec:innlp-total-cmplx}
        This section presents the major result that the total oracle complexity of our IAPG algorithm as measured by an upper bound on the total number of uses of $\hprox_{\lambda\omega^\star}$ and $\nabla f(x)$, ignoring line search and backtracking is $\mathcal O(\varepsilon^{-1/2}\ln\varepsilon^{-1})$.
        The following theorem calculates an upper bound of $\mathcal O(\varepsilon^{-1/2}\ln(\varepsilon^{-1}))$ for the total number of iterations of the inner loop needed to achieve $F(x_k) - F(\bar x) \le \varepsilon$ and stationarity; and an upper bound of $\mathcal O(\varepsilon^{-1}\ln(\varepsilon^{-1}))$ for the number of iterations needed.
        \begin{theorem}[the bounds on total number of inner iteration of IAPG]\;\label{thm:inn-lp-overall-cmplx}\\
            Under Assumption \ref{ass:opt-mmntm-seq}, let $(x_k)_{k \in \Z_+}$, $(\epsilon_k)_{k \in \Z_+}$, and $(F, f, g, L)$ be given by Definition \ref{def:inxt-apg}.
            Suppose Assumption \ref{ass:inn-cmplx} holds, and keep $J_k$ as introduced in Definition \ref{def:pgd-inner-lp}.
            Take $C_5, C_4$ as defined in Proposition \ref{prop:inner-lp-cmplx}\ref{prop:inner-lp-cmplx:cnst2}, \ref{prop:inner-lp-cmplx:cnst3}, $C_1, C_2, C_3$ as defined in Theorem \ref{thm:outerlp-itr-cmplx}\ref{thm:outerlp-itr-cmplx:cnst1}, \ref{thm:outerlp-itr-cmplx:cnst2}, \ref{thm:outerlp-itr-cmplx:cnst3}.
            Let $\bar x$ be a minimizer of objective function $F$. 
            We define in addition, the following constants: 
            \begin{enumerate}[nosep, label=(\Roman*), font=\upshape]
                \item $C_6:= \left\lceil C_2^{1/2}C_1^{-1} \right\rceil^2(4C_5)^{\frac{2}{2 + p}}$. 
                \item $C_7 := \lceil 4C_1C_2C_3 \rceil (4C_5)^{\frac{1}{2 + p}}$.  
                \item $C_8 := \left\lceil 4C_1C_2^{-1}C_3(L + L_{\max}) \right\rceil(4C_5)^{\frac{1}{2 + p}}$. 
            \end{enumerate}
            Then, the following are true. 
            \begin{enumerate}[nosep]
                \item For all $\varepsilon > 0$, there exists $k\in \Z_+$ such that $F(x_k) - F(\bar x) \le \varepsilon$. 
                In this case, the total number of inner loop iterations is bounded by $\mathcal O(\varepsilon^{-1/2}\ln(\varepsilon^{-1}))$ for small enough $\varepsilon$, and more specifically:
                {\small
                \begin{align}\label{thm:inn-lp-overall-cmplx:rslt1}\begin{split}
                    \sum_{l = 0}^{k} J_l
                    &\le\left(
                        1 + \left\lceil C_2^{1/2}\varepsilon^{-1/2}C_1^{-1}\right\rceil
                    \right)
                    \max\left(
                        1, \frac{2\ln(C_5)}{\ln(C_4)}, \frac{
                            (2 + p)\ln\left(
                                \max(1, 4\varepsilon^{-1})C_6
                            \right)
                        }{\ln(C_4)}
                    \right).
                \end{split}\end{align}
                }
                \item For all $\varepsilon > 0$, there exists $k\in \Z_+$ such that $\Vert x_k - y_k\Vert \le \varepsilon$. 
                In this case, the total number of inner loop iterations is bounded by $\mathcal O(\varepsilon^{-1}\ln(\varepsilon^{-1}))$ for small enough $\varepsilon$, and more specifically: 
                {\footnotesize
                \begin{align}\label{thm:inn-lp-overall-cmplx:rslt2}\begin{split}
                    \sum_{l = 0}^{k} J_l
                    &\le \left(
                    1 + \left\lceil 
                        \frac{4C_1C_2C_3}{\varepsilon}
                    \right\rceil
                    \right)\max\left(
                        1, \frac{2\ln(C_5)}{\ln(C_4)}, \frac{
                            2(2 + p)\ln\left(
                                \max(1, 2\varepsilon^{-1})C_7
                            \right)
                        }{\ln(C_4)}
                    \right). 
                \end{split}\end{align}}
                \item For all $\varepsilon > 0$, there exists $k \in \Z_+$ such that $\dist(\mathbf 0 | \partial_{\epsilon_k}F(x_k)) \le \varepsilon$. 
                In this case, the total number of inner loop iterations is bounded by $\mathcal O (\varepsilon^{-1}\ln(\varepsilon^{-1}))$ for small enough $\varepsilon$, and more precisely: 
                {\footnotesize\begin{align}\label{thm:inn-lp-overall-cmplx:rslt3}
                    \sum_{l = 0}^{k} J_l &\le
                    \left(
                        1 + \left\lceil 
                            \frac{4C_1 C_3(L + L_{\max})}{\varepsilon C_2}
                        \right\rceil
                    \right)\max\left(
                        1, \frac{2\ln(C_5)}{\ln(C_4)}, 
                        \frac{2(2 + p)\ln\left(
                            C_8\max\left(1, 2\varepsilon^{-1}\right)
                        \right)}{\ln(C_4)}
                    \right). 
                \end{align}}
            \end{enumerate}
        \end{theorem}
        \begin{proof}
            Let $k\in \Z_+$, recall $\bar J_k = \max_{i = 0, \ldots, k} J_i$ as defined in Proposition \ref{prop:inner-lp-cmplx}.
            Therefore, the total number of iterations of the inner loop has an upper bound of the form: $\sum_{l = 0}^{k}J_l \le \sum_{l=0}^{k} \bar J_k \le (k + 1)\bar J_k$.
            \par
            To show the first result \eqref{thm:inn-lp-overall-cmplx:rslt1}, we apply Theorem \ref{thm:outerlp-itr-cmplx}\ref{thm:outerlp-itr-cmplx:rsult1} because we assumed Assumption \ref{ass:opt-mmntm-seq}, and the fact that the algorithm of the outer loop satisfies Definition \ref{def:inxt-apg}.
            Therefore, if $k = \left\lceil \sqrt{\frac{C_2}{\varepsilon C_1^2}} \right\rceil$, then $F(x_k) - F(\bar x) \le \varepsilon$.
            Given such $k$, we apply \eqref{prop:inner-lp-cmplx:rslt1} from Proposition \ref{prop:inner-lp-cmplx} because Assumption \ref{ass:inn-cmplx} is assumed.
            To this end, we first simplify the algebra by considering:
            \begin{align}\label{thm:inn-lp-overall-cmplx:pitem1}\begin{split}
                2\ln\left(k(4C_5)^{\frac{1}{2 + p}}\right)
                &\underset{(1)} = 
                2\ln\left(
                    \left\lceil C_2^{1/2}\varepsilon^{-1/2}C_1^{-1} \right\rceil
                    (4C_5)^{\frac{1}{2 + p}}
                \right)
                \\
                &\le \ln\left(
                    \left\lceil \varepsilon^{-1/2}\right\rceil^2
                    \left\lceil C_2^{1/2} C_1^{-1}\right\rceil^2
                    (4C_5)^{\frac{2}{2 + p}}
                \right)
                \\
                &\underset{(2)}\le \ln\left(
                    \max(1, 4\varepsilon^{-1})
                    \left\lceil C_2^{1/2} C_1^{-1}\right\rceil^2
                    (4C_5)^{\frac{2}{2 + p}}
                \right)
                \\
                &\underset{(3)}= \ln\left( \max(1, 4\varepsilon^{-1})C_6 \right).
            \end{split}\end{align}
            At (1), we substituted: $k = \left\lceil C_2^{1/2}\varepsilon^{-1/2}C_1^{-1}\right\rceil$. 
            At (2), we used $\left\lceil \varepsilon^{-1/2} \right\rceil^2 \le \max(1, 4\varepsilon^{-1})$.
            This is because when $\varepsilon \ge 1$, we have $\left\lceil \varepsilon^{-1/2} \right\rceil^2 = 1$, and when $\varepsilon \in (0, 1)$ it follows that: 
            \begin{align*}
                \left\lceil \varepsilon^{-1/2} 
                \right\rceil^2 \le 
                \left(
                     \varepsilon^{-1/2} + 1
                \right)^2 \le 
                \left(
                    2\varepsilon^{-1/2}
                \right)^2 \le 4 \varepsilon^{-1}. 
            \end{align*}
            Finally, at (3), we substituted $C_6 = \left\lceil C_2^{1/2} C_1^{-1}\right\rceil^2(4C_5)^{\frac{2}{2 + p}}$ to simplify and obtain the final expression. 
            Now, combining with previously derived result \eqref{prop:inner-lp-cmplx:rslt1} from Proposition \ref{prop:inner-lp-cmplx}, the total number of iterations satisfies:
            \begin{align*}
                \sum_{l = 0}^{k} \bar J_k 
                &\le 
                (k + 1) \bar J_k
                \\[-1em]
                &\le 
                \left(
                    1 + \left\lceil C_2^{1/2}\varepsilon^{-1/2}C_1^{-1}\right\rceil
                \right)\max\left(
                    1, \frac{2\ln(C_5)}{\ln(C_4)}, \frac{
                        2(2 + p)\ln\left(
                            k(4C_5)^{\frac{1}{2 + p}}   
                        \right)
                    }{\ln(C_4)}
                \right)
                \\
                &\hspace{-0.3em}\underset{\eqref{thm:inn-lp-overall-cmplx:pitem1}}\le 
                \left(
                    1 + \left\lceil C_2^{1/2}\varepsilon^{-1/2}C_1^{-1}\right\rceil
                \right)
                \max\left(
                    1, \frac{2\ln(C_5)}{\ln(C_4)}, \frac{
                        (2 + p)\ln\left(
                            \max(1, 4\varepsilon^{-1})C_6
                        \right)
                    }{\ln(C_4)}
                \right)
                \\
                &= \mathcal O(\varepsilon^{-1/2}\ln(\varepsilon^{-1})). 
            \end{align*}
            \par
            To show \eqref{thm:inn-lp-overall-cmplx:rslt2}, we apply Theorem \ref{thm:outerlp-itr-cmplx}\ref{thm:outerlp-itr-cmplx:rsult2} because Assumption \ref{ass:opt-mmntm-seq} is assumed.  
            The result states that if $
                k = \left\lceil
                    \frac{4C_1C_2C_3}{\varepsilon}
                \right\rceil
            $, it follows that $\Vert x_k - y_k\Vert \le \varepsilon$. 
            Given such $k$, to pave the way for the derivation, we simplify the algebra by considering: 
            {\allowbreak
            \begin{align}\label{thm:inn-lp-overall-cmplx:pitem2}\begin{split}
                k(4C_5)^{\frac{1}{2 + p}}
                &= \left\lceil
                    4\varepsilon^{-1} C_1C_2C_3
                \right\rceil
                (4C_5)^{\frac{1}{2 + p}}
                \\
                &\le
                \left\lceil
                    \varepsilon^{-1}
                \right\rceil
                \left\lceil
                    4C_1C_2C_3
                \right\rceil
                (4C_5)^{\frac{1}{2 + p}}
                \\
                &\underset{(4)}\le 
                \max(1, 2\varepsilon^{-1})
                \left\lceil
                    4C_1C_2C_3
                \right\rceil
                (4C_5)^{\frac{1}{2 + p}}
                \\
                &\underset{(5)}= \max(1, 2\varepsilon^{-1})C_7.
            \end{split}\end{align}}
            At (4), we apply $\left\lceil \varepsilon^{-1}\right\rceil \le \max(1, 2\varepsilon^{-1})$.
            This is true because for all $\varepsilon \in (0, 1)$ we have $\left\lceil\varepsilon^{-1} \right\rceil \le \varepsilon^{-1} + 1 \le 2 \varepsilon^{-1}$.
            Otherwise, if $\varepsilon \ge 1$, it follows that $\left\lceil \varepsilon^{-1}\right\rceil = 1$.
            Therefore, combining the two cases gives $\left\lceil \varepsilon^{-1} \right\rceil \le \max(1, 2 \varepsilon^{-1})$. 
            At (5), we substituted the constant $C_7$ defined in the theorem statement.
            Next, we apply \eqref{prop:inner-lp-cmplx:rslt1} in Proposition \ref{prop:inner-lp-cmplx} because Assumption \ref{ass:inn-cmplx} is assumed here.
            It follows that the total number of inner loop iterations has:
            \begin{align*}
                \sum_{l = 0}^{k}J_l &\le (k + 1)\bar J_k 
                \\[-1em]
                &\le 
                \left(
                    1 + \left\lceil 
                        4\varepsilon^{-1}C_1C_2C_3
                    \right\rceil
                \right)\max\left(
                    1, \frac{2\ln(C_5)}{\ln(C_4)}, \frac{
                        2(2 + p)\ln\left(
                            k(4C_5)^{\frac{1}{2 + p}}   
                        \right)
                    }{\ln(C_4)}
                \right)
                \\
                &\hspace{-0.3em}\underset{\eqref{thm:inn-lp-overall-cmplx:pitem2}}\le 
                \left(
                    1 + \left\lceil 
                        4\varepsilon^{-1}C_1C_2C_3
                    \right\rceil
                \right)\max\left(
                    1, \frac{2\ln(C_5)}{\ln(C_4)}, \frac{
                        2(2 + p)\ln\left(
                            \max(1, 2\varepsilon^{-1})C_7
                        \right)
                    }{\ln(C_4)}
                \right). 
            \end{align*}
            \par
            Results \eqref{thm:inn-lp-overall-cmplx:rslt3} can be shown similarly. 
            We use Theorem \ref{thm:outerlp-itr-cmplx}\ref{thm:outerlp-itr-cmplx:rsult3} which states that if $k =
            \left\lceil 4\varepsilon^{-1}C_1C_2^{-1}C_3(L + L_{\max}) \right\rceil$, then $\dist(\mathbf 0 | \partial_{\epsilon_k}F(x_k)) \le \varepsilon$.
            Given such $k$, we simplify
            \begin{align}\label{thm:inn-lp-overall-cmplx:pitem3}\begin{split}
                k(4C_5)^{\frac{1}{2 + p}} 
                &= 
                \left\lceil 
                    \frac{4C_1C_3(L + L_{\max})}{2C_2}
                \right\rceil(4C_5)^{\frac{1}{2 + p}}
                \\
                &\le 
                \left\lceil \varepsilon^{-1}\right\rceil 
                \left\lceil 4C_1C_2^{-1}C_3(L + L_{\max})\right\rceil (4C_5)^{\frac{1}{2 + p}}
                \\
                &=
                \left\lceil \varepsilon^{-1}\right\rceil C_8
                \\
                &\le 
                \max(1, 2 \varepsilon^{-1})C_8
            \end{split}
            \end{align}
            Applying \eqref{prop:inner-lp-cmplx:rslt1} from Proposition \ref{prop:inner-lp-cmplx}, we have:
            \begin{align*}
                \sum_{l = 0}^{k}J_l &\le (k + 1)\bar J_k 
                \\[-1em]
                &\le 
                \left(
                    1 + \left\lceil 
                        \frac{4C_1 C_3(L + L_{\max})}{\varepsilon C_2}
                    \right\rceil
                \right)\max\left(
                    1, \frac{2\ln(C_5)}{\ln(C_4)}, 
                    \frac{2(2 + p)\ln\left(k(4C_5)^{\frac{1}{2 + p}}\right)}{\ln(C_4)}
                \right)
                \\
                &\hspace{-0.5em}\underset{\eqref{thm:inn-lp-overall-cmplx:pitem3}}\le 
                \left(
                    1 + \left\lceil 
                        \frac{4C_1 C_3(L + L_{\max})}{\varepsilon C_2}
                    \right\rceil
                \right)\max\left(
                    1, \frac{2\ln(C_5)}{\ln(C_4)}, 
                    \frac{2(2 + p)\ln\left(
                        C_8\max\left(1, 2\varepsilon^{-1}\right)
                    \right)}{\ln(C_4)}
                \right). 
            \end{align*}
        \end{proof}
        \begin{remark}
            This result is new to the best of our knowledge.  
            This complexity had never been shown in the literature in a similar context.
        \end{remark}
        \begin{corollary}[total complexity of IAPG]\label{corollary:total-iapg-cmplx}
            Suppose that Assumptions \ref{ass:opt-mmntm-seq}, \ref{ass:inn-cmplx} are true. 
            Take the outer loop iterates $(x_k)_{k \in \Z_+}$, and objective function $(F, f, g, L)$ as given by Definition \ref{def:inxt-apg}. 
            Let $\bar x$ be a minimizer of $F$.
            Then, ignoring complexity involved for line search, the total number of uses of $\hprox_{\lambda \omega^\star}, \nabla f$ in the IAPG algorithm satisfies:
            \begin{enumerate}[nosep]
                \item\label{corollary:total-iapg-cmplx-rslt1} For sufficiently small $\varepsilon > 0$, there exists $k \in \Z_+$ such that $F(x_k) - F(\bar x) \le \varepsilon$, and the total number of calls on $\nabla f, \hprox_{\lambda \omega^\star}$ is bounded by $\mathcal O(\varepsilon^{-1/2}\ln(\varepsilon^{-1}))$. 
                \item\label{corollary:total-iapg-cmplx-rslt2} For sufficiently small $\varepsilon > 0$, there exists $k \in \Z_+$ such that $\dist(\mathbf 0 | \partial_{\epsilon_k} F(x_k)) \le \varepsilon$, and the total number of calls on $\nabla f, \hprox_{\lambda \omega^\star}$ is bounded by $\mathcal O(\varepsilon^{-1}\ln(\varepsilon^{-1}))$. 
            \end{enumerate}
        \end{corollary}
        \begin{proof}
            Let $\varepsilon > 0$ be sufficiently small (it is sufficient to have $\varepsilon \le 1$). 
            Then, in the setting of \ref{corollary:total-iapg-cmplx-rslt1}, the outer loop iterative complexity is bounded by $\mathcal O(\varepsilon^{-1/2})$ by Theorem \ref{thm:outerlp-itr-cmplx}\ref{thm:outerlp-itr-cmplx:rsult1}, and the total iterative complexity of the inner loop is bounded by $\mathcal O(\varepsilon^{-1/2}\ln(\varepsilon^{-1}))$ by \eqref{thm:inn-lp-overall-cmplx:rslt1} from Theorem \ref{thm:inn-lp-overall-cmplx}. 
            \par
            In the setting of \ref{corollary:total-iapg-cmplx-rslt2}, the total iterative complexity of the inner loop is bounded by $\mathcal O(\varepsilon^{-1}\ln(\varepsilon^{-1}))$ by \eqref{thm:inn-lp-overall-cmplx:rslt3} in Theorem \ref{thm:inn-lp-overall-cmplx}, and the iterative complexity of the outer loop is bounded by $\mathcal O(\varepsilon^{-1})$ by Theorem \ref{thm:outerlp-itr-cmplx}\ref{thm:outerlp-itr-cmplx:rsult3}.
            \par
            Under the assumption of no line search, each inner loop iteration uses $\hprox_{\lambda \omega^\star}$ exactly once. 
            Similarly, each iteration of the outer loop uses $\nabla f$ exactly once. 
            Therefore, in the setting of \ref{corollary:total-iapg-cmplx-rslt1}, the total number of uses of $\nabla f, \hprox_{\lambda \omega^\star}$ is bounded by the complexity: $\mathcal O(\varepsilon^{-1/2} \ln(\varepsilon^{-1}))$.
            In the setting of \ref{corollary:total-iapg-cmplx-rslt2}, the total number of uses of $\nabla f, \hprox_{\lambda \omega^\star}$ is bounded by the complexity: $\mathcal O(\varepsilon^{-1} \ln(\varepsilon^{-1}))$.
        \end{proof}

\section{Algorithm implementations}\label{sec:alg-impl}
    This section presents implementation details for the inner loop (Algorithm \ref{alg:inner-loop}), and outer loop (Algorithm \ref{alg:outer-loop}) of our IAPG algorithm.
    Propositions \ref{prop:inn-lp-alg-works}, \ref{prop:outer-lp-works} below show that the implementations comply with our theories.
    \subsection{Inner loop implementations}
        Algorithm \ref{alg:inner-loop} implements the logic that find an element in $\approx_\epsilon \hprox_{\lambda \omega^\star}$ for the IAPG, i.e., it is the pseudocode for the inner loop.
        \begin{algorithm}[H]
            \begin{algorithmic}[1]
                \footnotesize\small
                \STATE{
                    \textbf{PPPGD:}
                    \begin{tabular}{|ll}
                        $\omega : \RR^m \rightarrow \RR$ & Proper closed, and convex\\ 
                        $A\in \RR^{m \times n} $  &  Matrix \\
                        $z_0 \in \RR^n$ & Initial guess \\
                        $y_k \in \RR^n$ & Iterate from the outer loop \\
                        $y^+ \in \RR^n$ & Iterate from outer loop, it should be $y^+ := y_k - L_{k}^{-1}\nabla f(y_k)$  \\
                        $\epsilon^\circ$ & $\epsilon^\circ \ge 0$, Absolute error\\
                        $\rho $ & $\rho > 0$, Relative Error\\
                        $\lambda$ & $\lambda > 0$\\
                        $\tau_0 = \lambda\Vert A^\top A\Vert$ & step size inverted \\
                        $s \in \N$ & Line search constant shrinkage half-life
                    \end{tabular}\vspace{0.5em}
                }
                \STATE{$v_0 := \hprox_{\omega^\star}(z_0)$}
                \STATE{$\Phi_\lambda(z) := \omega(Az) + (1/2)\lambda^{-1}\Vert z - y^+\Vert^2$}
                \STATE{$\Psi_\lambda(v) := (\lambda/2)\Vert A^\top v\Vert^2 - \left\langle A^\top v, y^+\right\rangle + \omega^\star(v)$}
                \FOR{$j = 0, 2, \ldots, 2^{20} - 1$}
                    \IF{$\Phi_\lambda(z_j) + \Psi_\lambda(v_j) < \epsilon^\circ + (\rho/2)\Vert z_j - y_k\Vert^2$}
                    \label{alg:inner-loop-exit-cond}
                        \STATE{\textbf{break}}
                    \ENDIF
                    \WHILE{$\tau_j \le 2^{1023}$}\label{alg:inner-loop-start-ls}
                        \STATE{$
                            v_{j + 1} := 
                            \hprox_{\tau^{-1}_j\omega^\star}
                            \left(v_j - \tau_j^{-1}A(\lambda A^\top v_j - y)\right).
                        $}\label{alg:inner-loop-primal-update}
                        \IF{$\lambda \left\Vert A^\top (v_{j + 1} - v_{j})\right\Vert^2 \le \tau_j\Vert v_{j + 1} - v_j\Vert^2$}\label{alg:inner-loop-ls-cond}
                            \STATE{\textbf{break}}
                        \ENDIF
                        \STATE{$\tau_j := 2\tau_j$}
                    \ENDWHILE\label{alg:inner-loop-end-ls}
                    \IF{$\tau_j > 2^{1023}$}
                    \STATE{\textbf{return Line Search Error}}\label{alg:inner-loop-bstrk}
                    \ENDIF
                    \STATE{$\tau_{j + 1} := 2^{-1/s}\tau_j$} \label{alg:inner-loop-rlx}
                    \STATE{$z_{j + 1} := y^+ - \lambda A^\top v_{j + 1}$} \label{alg:inner-loop-dual-update}
                \ENDFOR
                \STATE{\textbf{return } $z_j$}
                \caption{Proximal Point Problem with PGD, inner loop}\label{alg:inner-loop}
            \end{algorithmic}
        \end{algorithm}
        The line search and backtracking procedures of Algorithm \ref{alg:inner-loop} accommodate numerical stability, flexibility, and best performance practices.
        The exit condition $\tau > 2^{1023}$ (line \ref{alg:inner-loop-start-ls}) safeguards against overflow in common floating-point standard.
        It exits when the line search constant overflows.
        Line \ref{alg:inner-loop-exit-cond} includes relative error $(\rho/2)\Vert z_j - y_k\Vert^2$ on the duality gap. 
        Line \ref{alg:inner-loop-ls-cond} implements the closed-form formula suggested in the remark of Definition \ref{remark:pgd-inner-lp} to expedite line search.
        Besides that, the constant $s = 4096$ is chosen here as the rate which $\tau_j$ at which decreases.
        Observe that at Line \ref{alg:inner-loop-rlx}, $\tau_{j+1}$ decreases to $2^{-1/s}\tau_j$. 
        Therefore, if line search is never triggered, $\tau$ would halve itself every $4096$ iterations. 
        Our choice here is made conservative for best stability to prevent frequent line searches. 
        Finally, $y^+, y_k$ are given by the outer loop because they are fixed for the inner loop.
        \par
        To apply the results of previous sections, specifically the inner loop complexity in Theorem \ref{thm:inn-loop-lin-cnvg}, Algorithm \ref{alg:inner-loop} satisfies Definition \ref{def:pgd-inner-lp}.
        The proposition below demonstrates it.
        \begin{proposition}[Algorithm \ref{alg:inner-loop} is an algorithm for the inner loop]\;\label{prop:inn-lp-alg-works}\\
            Let $\lambda > 0, \epsilon > 0$, and $(\omega, A, y_k, \lambda, h_\lambda, \Phi_\lambda, \Psi_\lambda, \kappa_\lambda)$ satisfy Assumption \ref{ass:lin-cnvg-for-pp}. 
            Let initial guess $v_0\in \dom \omega^\star$ be feasible, let $z_0 = y_k - \lambda A^\top v_0$. 
            Then, Algorithm \ref{alg:inner-loop} satisfies Definition \ref{def:pgd-inner-lp} with $\bar \tau_\lambda \le 2\lambda \Vert A^\top\Vert^2$. 
        \end{proposition}
        \begin{proof}
            We verify Algorithm \ref{alg:inner-loop} against Definition \ref{def:pgd-inner-lp}. 
            Update of $v_{j + 1} = \hprox_{\tau_j^{-1}\omega^\star}(v_j - \tau_j^{-1}\lambda A(A^\top v_j - y^+))$ at line \ref{alg:inner-loop-primal-update} implements $v_{j + 1} = \hprox_{\tau_j^{-1}\omega^\star}\left(v_j - \tau_j^{-1} A\nabla h_\lambda (A^\top v_j)\right)$ from \eqref{def:pgd-inner-lp-vj}. 
            Here, it uses $h_\lambda = x \rightarrow \frac{1}{2\lambda}\Vert \lambda x - y\Vert^2$ established in Assumption \ref{ass:lin-cnvg-for-pp}, from which one can readily verify $\nabla (h_\lambda\circ A^\top)(v_j) = A(\lambda A^\top v_j - y)$.
            \par
            Together Lines \ref{alg:inner-loop-start-ls}, \ref{alg:inner-loop-end-ls}, and \ref{alg:inner-loop-rlx} perform line search and backtracking for condition \eqref{def:pgd-inner-lp-ls}. 
            Line \ref{alg:inner-loop-start-ls} - \ref{alg:inner-loop-end-ls} performs Armijo line search by doubling $\tau_j$ (or equivalently halving the step size $\tau_j^{-1}$) when the condition failed, and accepting $v_{j + 1}$ when it succeeded.
            Then, line \ref{alg:inner-loop-rlx} implements backtracking on $\tau$.
            It shrinks $\tau_{j}$ by a factor of $1/(2^{1/s})$ for $\tau_{j + 1}$ in the next iteration. 
            Observe that line search condition (line \ref{alg:inner-loop-ls-cond}) is always satisfied for all $\tau \ge \lambda \Vert A^\top\Vert^2$. 
            It ensures that the doubling of $\tau_j$ always has $\tau_j \le 2\lambda\Vert A^\top\Vert^2$. 
            Therefore, $\sup_{j\in \Z_+} \tau_j :=\bar \tau_\lambda \le 2\lambda\Vert A^\top\Vert^2$. 
            \par
            Line \ref{alg:inner-loop-dual-update} updates $z_{j + 1}$ to satisfy \eqref{def:pgd-inner-lp-zj}. 
            Exit condition at line \ref{alg:inner-loop-exit-cond} ensures that termination occurs at the smallest $j$ such that $\mathbf G_\lambda(z_j, v_j) \le \epsilon$. 
        \end{proof}
        \begin{remark}
            We remark that $\Vert A^\top\Vert = \Vert A\Vert$. 
        \end{remark}

        Proposition \ref{prop:inn-lp-alg-works} is an enormous result because under Assumption \ref{ass:lin-cnvg-for-pp}, it ensures that the linear convergence claim (Theorem \ref{thm:inn-loop-lin-cnvg}) applies to Algorithm \ref{alg:inner-loop}. 

    \subsection{Outer loop implementations}
        Next, Algorithm \ref{alg:outer-loop} highlights the details for the outer loop implementation of IAPG. 
        \begin{algorithm}[H]
            \begin{algorithmic}[1]
                \footnotesize\small
                \STATE{
                    \textbf{IAPG: }
                    \begin{tabular}{|ll}    
                        $f: \RR^n \rightarrow \RR$ & Lipschitz smooth convex. \\
                        $\omega: \RR^m \rightarrow \RR$ & Proper, closed and convex. \\
                        $A \in \RR^{m \times n}$ &   \\
                        $x_{-1} \in \RR^n$ &  Initial guess. \\
                        $B_0 > 0$ & A valid Lipschitz smoothness estimate. \\
                        $s \in \N$ & Line search constant shrinkage half-life.  \\
                        $\rho > 0$ & Over-relaxation parameter.  \\ 
                        $p > 1$ &  \\ 
                        $r \in (0, 1]$ & Ratio between minimum and maximum line search constant.
                        \\
                        $\varepsilon$ & Tolerance on the stationarity. 
                    \end{tabular}\vspace{0.5em}
                }
                \STATE{$L_0 := (1 + \rho)B_0$}
                \STATE{$L_{\max} := L_0$}
                \STATE{$\alpha_0 := 1$}
                \STATE{$x_{-1}^\circ := x_{-1}$}
                \FOR{$k = 0, 1, 2, \ldots, N$}
                    \STATE{$y_k := \alpha_k x_{k - 1}^\circ + (1 - \alpha_k)x_{k - 1}$}
                    \label{alg:outer-loop-yk}
                    \STATE{$\rho_k := \rho B_k$}
                    \WHILE{$B_k \le 2^{1023}$}\label{alg:outer-loop-lsbtrk-start}
                        \STATE{$\epsilon_k^\circ := L_kL_0^{-1}\alpha_k^{2}\mathcal E_0 k^{-p}$ if $k > 0$ else $\mathcal E_0$}\label{alg:outer-loop-epsk}
                        \STATE{$y^+ := y_k - L_k^{-1}\nabla f(y_k)$}
                        \STATE{$x_k := \textbf{PPPGD}\left(\omega, A, z_0=y_k, y_k, y^+, \epsilon_k^\circ, \rho_k, L_k^{-1}\right)$}\label{alg:outer-loop-xk}
                        \IF{$f(x_k) - f(y_k) - \langle \nabla f(y_k), x_k - y_k\rangle \le B_k/2\Vert x_k - y_k\Vert^2$}
                        \label{alg:outer-loop-descent-cond}
                            \STATE{\textbf{break}}
                        \ENDIF
                        \STATE{$B_k := 2B_k$}\label{alg:outer-loop-2Bk}
                        \STATE{$\rho_k := \rho B_k$}
                        \STATE{$L_k := (1 + \rho)B_k$}\label{alg:outer-loop-Lk}
                        \STATE{$L_{\max} := \max(L_k, L_{\max})$}
                        \ENDWHILE\label{alg:outer-loop-lsbtrk-end}
                    \IF{$\Vert x_k - y_k\Vert \le \varepsilon$}\STATE{\textbf{break}}\ENDIF
                    \IF{$B_k > 2^{1023}$}
                        \STATE{\textbf{Return Line Search Error}}\label{alg:outer-loop-error-case}
                    \ENDIF
                    \STATE{$L_{k + 1} := \max\left(2^{-1/s}L_k, rL_{\max}\right)$}\label{alg:outer-loop-Lk-shrinks}
                    \STATE{$x_{k}^\circ := x_{k - 1} + \alpha_k^{-1}(x_k - x_{k - 1})$}
                    \label{alg:outer-loop-xk-circ}
                    \STATE{$\alpha_{k + 1} := (1/2)L_{k}L_{k + 1}^{-1}\left(-\alpha_{k}^2 + (\alpha_{k}^4 + 4\alpha_{k}^2L_{k + 1}L_{k}^{-1})^{1/2}\right)$}\label{alg:outer-loop-alphak}
                \ENDFOR
            \end{algorithmic}
            \caption{The inexact accelerated proximal gradient method in the outer loop}\label{alg:outer-loop}
        \end{algorithm}
        Algorithm \ref{alg:outer-loop} places safeguards on numerical stability. 
        Lines \ref{alg:outer-loop-lsbtrk-start}, \ref{alg:outer-loop-error-case} together ensure that the line search doesn't overflow floating point standard. 
        \begin{proposition}[Algorithm \ref{alg:outer-loop} is an algorithm for the outer loop]\;\label{prop:outer-lp-works}\\
            Let $(F, f, g, L)$ satisfy Assumption \ref{ass:for-inxt-pg-ineq}, $(g, \omega, A, K_\omega)$ satisfy Assumption \ref{ass:for-inxt-prox}. 
            Then, the following are true for Algorithm \ref{alg:outer-loop}. 
            \begin{enumerate}[nosep]
                \item\label{alg:outer-loop-rslt1} Iterates $(y_k, x_k, x_k^\circ)_{k \in \Z_+}$ and line search sequences $(B_k)_{k \in \Z_+}$ satisfy Definition \ref{def:inxt-apg}. 
                \item\label{alg:outer-loop-rslt2} Sequences $(\alpha_k)_{k \in \Z_+}, (L_k)_{k \in \Z_+} (\epsilon_k)_{k \in \Z_+}$, and constant $p, L_{\min}, L_{\max}$ satisfy Assumption \ref{ass:opt-mmntm-seq}, and we have $\frac{L_{\min}}{L_{\max}} \ge r$.
                Here, $r$ is from Algorithm \ref{alg:outer-loop}. 
            \end{enumerate}
        \end{proposition}
        \begin{proof}
            To verify \ref{alg:outer-loop-rslt1}, we need to verify \eqref{def:inxt-apg:yk}, \eqref{def:inxt-apg:xk}, \eqref{def:inxt-apg:vk} and \eqref{def:inxt-apg:ls} from Definition \ref{def:inxt-apg} by the implementations of Algorithm \ref{alg:outer-loop}. 
            Indeed, Line \ref{alg:outer-loop-yk} implements \eqref{def:inxt-apg:yk}, Line \ref{alg:outer-loop-xk} implements \eqref{def:inxt-apg:xk}, and Line \ref{alg:outer-loop-xk-circ} implements \eqref{def:inxt-apg:vk}. 
            The line search condition \eqref{def:inxt-apg:ls} is verified by Line \ref{alg:outer-loop-descent-cond}. 
            This is because $D_f(x_k, y_k) = f(x_k) - f(y_k) - \langle \nabla f(y_k), x_k - y_k\rangle$.
            \par
            Next, we verify \ref{alg:outer-loop-rslt2}. 
            To start, consider $(\epsilon_k)_{k\in \Z_+}$. 
            Recall that the tolerance has $\epsilon_k = \mathcal E_0 \beta_k k^{-p} + (\rho_k/2) \Vert x_k - y_k\Vert^2$ from Assumption \ref{ass:opt-mmntm-seq}\ref{ass:opt-mmntm-seq:item3}. 
            It can be seen from Line \ref{alg:outer-loop-epsk} that $\epsilon_k^\circ = L_kL_0^{-1} \mathcal E_0 k^{-p}$ when $k \in \N$ and $\epsilon_0^\circ = \mathcal E_0$ implements the first term representing the absolute tolerance. 
            The relative tolerance is implemented by Line \ref{alg:inner-loop-exit-cond} of Algorithm \ref{alg:inner-loop}. 
            Together, they form the tolerance $\epsilon_k$ which is the upper bound for the primal dual gap. 
            \par
            Next, Line \ref{alg:outer-loop-alphak} updates $(\alpha_k)_{k \in \Z_+}$ as specified in Assumption \ref{ass:opt-mmntm-seq}\ref{ass:opt-mmntm-seq:item1}. 
            Therefore, it satisfies Assumption \ref{ass:opt-mmntm-seq}. 
            The sequence $(L_k)_{k \in \Z_+}$ is managed by Line \ref{alg:outer-loop-Lk}, \ref{alg:outer-loop-Lk-shrinks}, and it forms part of the line search and back tracking routine. 
            First, we show that $L_k$ is bounded above. 
            Assumption \ref{ass:for-inxt-prox} states that $f$ is $L$-Lipschitz smooth so the line search condition (line \ref{alg:outer-loop-descent-cond}) holds for all $B_k \ge L$.
            By the doubling patterns on Line \ref{alg:outer-loop-2Bk}, we have $\sup_{k\in \Z_+}B_k \le 2L$.
            Therefore, $L_k \le 2(1 + \rho)L$ by Line \ref{alg:outer-loop-Lk}. 
            Finally, Line \ref{alg:outer-loop-Lk-shrinks} gives $L_k$ the lower bound $rL_{\max}$. 
            Therefore, we have $\frac{L_{\min}}{L_{\max}} \ge r$.
        \end{proof}
        \par
        Proposition \ref{prop:outer-lp-works} is major because it ensures that the established total complexity in Theorem \ref{thm:inn-lp-overall-cmplx} of IAPG applies.
        This is because under Assumption \ref{ass:opt-mmntm-seq} and Definition \ref{def:inxt-apg}, Theorem \ref{thm:outerlp-itr-cmplx} applies. 
        Moreover, Proposition \ref{prop:inn-lp-alg-works} establishes that the inner loop complies with the conditions needed for Theorem \ref{thm:inn-lp-overall-cmplx}.
        Therefore, the total complexity results of IAPG from Corollary \ref{corollary:total-iapg-cmplx} apply.

\section{Examples where inner loop has linear convergence}\label{sec:innlp-fxn}
    This section characterizes a class of functions of $\omega$ such that the convergence theories of IAPG from all prior sections apply.
    More specifically, we present the class of conic polyhedral functions, i.e., for every member $\omega$ of this class of functions, there exist $N \in \N$ and a finite collection of $\{w_i\}_{i = 1}^{N}\subseteq \RR^m$ such that $\omega(z) = \max_{i = 1, \ldots, N} \langle w_i, z\rangle$.
    Next, we show that if $\omega$ belongs to this class of functions, Assumptions \ref{ass:lin-cnvg-for-pp}, \ref{ass:inn-cmplx}\ref{ass:inn-cmplx:item3} are satisfied.
    Therefore, all theoretical results of the previous section apply.
    \par
    To this end, we divide this section into two subsections (Sections \ref{ssec:iloop-concrete-eb}, \ref{ssec:example-fxn-lin-cnvg-lp}).
    The first subsection presents existing results in the literature regarding the quadratic growth property of feasibility problems over a polyhedral domain.
    The second subsection establishes the fact that the convex conjugate $\omega^\star$ is the indicator function of a polytope (Lemma \ref{lemma:conic-polyfxn}), and hence it enables us to formulate the inner loop dual objective as a composite feasibility problem over a polytopic domain.
    Therefore, we can apply facts from the first subsection to show that if $\omega$ is conic polyhedral, then the dual objective $\Psi_\lambda$ satisfies all the assumptions.
    \subsection{Quadratic growth of polyhedral feasibility problem}\label{ssec:iloop-concrete-eb}
        In this section, we recall result from the literature which states that a feasibility problem with polyhedral constraints satisfies the quadratic growth condition.
        Fortunately, everything we need is in the work of Necoara et al. \cite{necoara_linear_2019}.
        We introduce quasi-strongly convex function (Definition \ref{def:q-scnvx}), along with several additional facts.
        The goal of this section is to present Fact \ref{fact:polyhedral-qscnvx-fxn}.
        It shows that a composite optimization problem in the form of $g\circ A + \delta_X$ where $X$ is a polyhedral set satisfies quadratic growth condition if $g$ is strongly convex and Lipschitz smooth.
        \begin{definition}[Quasi-strongly convex {\cite[Definition 1]{necoara_linear_2019}}]\;\label{def:q-scnvx}\\
            Let $f:\RR^n \rightarrow \RR$ be convex, differentiable, and $L$-Lipschitz smooth.
            Let $X \subseteq \RR^n$ and suppose that the set of minimizers $X^+=\argmin_{x \in X} f(x) \neq \emptyset$, and denote $f_{\min}$ to be the minimum of $f$ on $X$.
            It is quasi-strongly convex with constant $\kappa > 0$ on $X \subseteq \RR^n$ if there exists $\kappa > 0$ such that $\forall x \in X$, letting $\bar x = \Pi_{X^+} x$, we have
            \begin{align*}
                0 &\le f_{\min} - f(x) - \langle \nabla f(x), \bar x - x\rangle 
                - \frac{\kappa}{2}\Vert x - \bar x\Vert^2. 
            \end{align*}
        \end{definition}
        \begin{remark}
            This class of functions was introduced by Necoara et al. \cite{necoara_linear_2019}.
        \end{remark}
        \begin{fact}[quasi-strongly convex implies quadratic growth {\cite[Theorem 4]{necoara_linear_2019}}]\;\label{fact:qscnvx-q-growth}\\
            Let $f, X, \kappa$ be given by Definition \ref{def:q-scnvx}. 
            Then the function $F = f + \delta_X$ satisfies Definition \ref{def:q-grwth} (quadratic growth) with the same $\kappa$.
        \end{fact}
        The following classical result on the Hoffman error bound is paraphrased from Necoara et al.~\cite{necoara_linear_2019}.
        \begin{fact}[Hoffman error bound]\label{fact:hoffm-eb}
            Consider a nonempty polyhedral set $P = \{x \in \RR^n:Ax = b, Cx \le d\}$ defined via some $A \in \RR^{p \times n}, C \in \RR^{m\times n}, b \in \RR^p, d \in \RR^m$. 
            Then there exists a constant $\theta > 0$ depending only on $A$ and $C$: 
            \begin{align*}
                (\forall x \in \RR^n) \quad \dist(x | P) 
                \le \theta \dist\left(
                    (Ax - b, Cx - d) \;|\; 
                    \{\mathbf 0\}\times \RR^m_-
                \right). 
            \end{align*}
        \end{fact}
        \begin{remark}
            In the literature, estimating the smallest value of $\theta$ is an extensive research area
            An explicit formula is given in Necoara et al. \cite{necoara_linear_2019}. 
            Here, we will state only its existence without giving a precise expression. 
            \par
            We elaborate on the right-hand side of the inequality.
            Let $v \in \RR^m$ be a vector; we denote the projection of $v$ onto $\RR^m_+$ by $[v]_+$. 
            This applies $v_i \mapsto \max(v_i, 0)$ element-wise to vector $v$.
            The RHS can then be written as: 
            \begin{align*}
                \dist((Ax - b, Cx - d) \;|\; \{\mathbf 0\}\times \RR^m_-)
                &= \left\Vert
                    (Ax - b, Cx - d - \Pi_{\RR^m_-} (Cx - d))
                \right\Vert
                \\
                &= 
                \left\Vert
                    (Ax - b, \Pi_{\RR^m_+} (Cx - d))
                \right\Vert. 
            \end{align*}
            Since the distance is measured with respect to the $\ell^2$ norm, the following holds:
            \begin{align*}
                \dist((Ax - b, Cx - d) \;|\; \{\mathbf 0\}\times \RR^m_-)^2
                = \Vert Ax - b\Vert^2 + \Vert \Pi_{\RR^m_+}(Cx - d)\Vert^2. 
            \end{align*}
        \end{remark}
        \begin{fact}[quasi-strongly convex feasibility problem {\cite[Theorem 8]{necoara_linear_2019}}]\;\label{fact:polyhedral-qscnvx-fxn}\\
            Consider any $C \in \RR^{m\times n}, d \in \RR^m$ defining a nonempty polyhedral set $X = \{x : Cx \le d\}$.
            Let $h$ be $\sigma > 0$ strongly convex and $L$-Lipschitz smooth, and consider $f = h\circ A + \delta_X$ where $A \in \RR^{p \times n}$. 
            Then the following hold:
            \begin{enumerate}[nosep]
                \item The set of minimizers $X^+$ is nonempty, and it is a polyhedral set.
                \item The function $f$ is quasi-strongly convex (Definition \ref{def:q-scnvx}) with $\kappa = \sigma/\theta^2$ where $\theta$ is the Hoffman constant from Fact \ref{fact:hoffm-eb} for the polyhedral set $X^+$, and $\theta$ depends only on $A, C$. 
            \end{enumerate}
        \end{fact}

    \subsection{IAPG has near-optimal complexity for conic polyhedral regularizers}\label{ssec:example-fxn-lin-cnvg-lp}
        In this section, we use the theoretical results presented in the previous section to show that the dual objective $\Psi_\lambda$ from (\ref{eqn:dual-pp}) satisfies the quadratic growth condition under the assumption that $\omega$ is a conic polyhedral function.
        The following Lemma characterizes the structure of a conic polyhedral function $\omega$ and its convex conjugate.
        \begin{lemma}[convex conjugate of a max-affine function]\;\label{lemma:conic-polyfxn}\\
            Let $N\in \N$. 
            Choose $\{w_i\}_{i = 1}^N\subseteq \RR^m$. 
            Let $\mathbf \Delta^N = \{(\lambda_1, \ldots, \lambda_N) \in \RR^N: \sum_{i = 1}^{N}\lambda_i = 1\}\cap \RR^N_+$, the simplex. 
            Define $P$ to be the convex hull of the set of vectors $\{w_i\}_{i = 1}^N$, i.e., $P = \left\{\sum_{i = 1}^{N}\lambda_i w_i: (\lambda_1, \ldots, \lambda_N) \in \mathbf \Delta^N\right\}$.
            Define $\omega(v) = \max_{i = 1, \ldots, N} \langle w_i, v\rangle$.
            Then $\omega^\star = \delta_P$.
        \end{lemma}
        \begin{proof}
            We show that $\delta_P^\star(z) = \max_{i = 1, \ldots, N} \langle w_i, z\rangle = \omega (z)$.
            Since by definition $\omega$ is proper, closed and convex, $\omega^\star = \delta_P^{\star\star} = \delta_P$ by the biconjugate theorem.
            To demonstrate, consider:
            \begin{align*}
                \delta_P^\star(z) &=
                \sup_{v\in \RR^m}\left\lbrace
                    \langle z, v\rangle - \delta_P(v)
                \right\rbrace 
                \\
                &= 
                \sup_{v \in P}\left\lbrace
                    \langle z, v\rangle
                \right\rbrace 
                \\
                &= 
                \sup_{
                    \substack{
                        (\lambda_1, \ldots, \lambda_N)\\\in \mathbf \Delta^N
                    }
                }
                \left\lbrace
                    \left\langle z, \sum_{i = 1}^{N} \lambda_i w_i\right\rangle
                \right\rbrace 
                \\
                &= \sup_{
                    \substack{
                        (\lambda_1, \ldots, \lambda_N)\\\in \mathbf \Delta^N
                    }
                }
                \left\lbrace
                    \sum_{i = 1}^{N} \lambda_i \left\langle z,  w_i\right\rangle
                \right\rbrace 
                \\
                &= \max_{i = 1, \ldots, N} 
                \left\langle z,  w_i\right\rangle. 
            \end{align*}
        \end{proof}
        \par
        The following theorem is our main result.
        It shows that results from Section \ref{sec:total-cmplx} apply to conic polyhedral functions $\omega$.
        \begin{theorem}[near optimal complexity applies for conic polyhedral]\;\label{thm:conic-polyfxn-works}\\
            Let $\Psi_\lambda$ be given by \eqref{eqn:dual-pp}, i.e., $\Psi_\lambda = h_\lambda\circ A^\top + \omega^\star$ where $h_\lambda(v) = \frac{1}{2\lambda}\Vert \lambda v - y \Vert^2 - \frac{1}{2\lambda}\Vert y\Vert^2$. 
            Consider $\omega(z) = \max_{i = 1, \ldots, N} \langle w_i, z\rangle$ where $N \in \N$, and $\{w_i\}_{i = 1}^N \subseteq \RR^m$. 
            Then the following are true. 
            \begin{enumerate}[nosep]
                \item\label{thm:conic-polyfxn-works:item1} $\omega$ is $K_\omega = \max_{i = 1, \ldots, N}\Vert w_i\Vert$ $K_\omega$-Lipschitz continuous, and $\dom(\omega^\star)$ is a bounded set. Therefore, $\omega$ satisfies Assumption \ref{ass:for-inxt-prox}.
                \item\label{thm:conic-polyfxn-works:item2} $\Psi_\lambda$ satisfies Assumption \ref{ass:lin-cnvg-for-pp}\ref{ass:lin-cnvg-for-pp:item3} with $\kappa_\lambda = \frac{\lambda}{\theta^2}$ where $\theta$ is the Hoffman constant that only depends on matrix $A^\top$, and $\{w_i\}_{i = 1}^N$.
                \item\label{thm:conic-polyfxn-works:item3} If $\lambda$ is bounded below by $\lambda_{\min} > 0$, then $\kappa_\lambda$ is bounded below by $\frac{\lambda_{\min}}{\theta^2}$.
                Therefore, $\Psi_\lambda$ satisfies Assumption \ref{ass:inn-cmplx}\ref{ass:inn-cmplx:item3}.
            \end{enumerate}
        \end{theorem}
        \begin{proof}
            To verify \ref{thm:conic-polyfxn-works:item1}, by Lemma \ref{lemma:conic-polyfxn} it follows that $\omega^\star = \delta_P$ where
            \begin{align*}
                P = \left\{
                    \sum_{i = 1}^{N}\lambda_i w_i : (\lambda_1, \ldots, \lambda_N) \in \mathbf \Delta^N 
                \right\}. 
            \end{align*}
            $\dom \omega^\star = P$, and $P$ is always a bounded set. 
            Finally, $\omega$ is Lipschitz continuous with constant $K_\omega = \max_{i = 1, \ldots, N}\Vert w_i\Vert$ as follows directly from the definition of $\omega$.
            \par
            To verify \ref{thm:conic-polyfxn-works:item2}, we use Fact \ref{fact:polyhedral-qscnvx-fxn} to conclude that $\Psi_\lambda$ is quasi-strongly convex with $\kappa = \frac{\lambda}{\theta^2}$.
            Next, we will use Fact \ref{fact:qscnvx-q-growth} to conclude that $\Psi_\lambda$ satisfies the quadratic growth condition.
            Recall that $\Psi_\lambda = h_\lambda\circ A^\top + \omega^\star$ where $h_\lambda(v) = \frac{1}{2\lambda}\Vert \lambda v - y\Vert^2 - \frac{1}{2\lambda}\Vert y\Vert^2$ is a $\lambda$-strongly convex function.
            Therefore, $\Psi_\lambda$ satisfies Fact \ref{fact:polyhedral-qscnvx-fxn} with $h$ being $h_\lambda$, $\sigma = \lambda$, and $A$ being $A^\top$.
            Therefore, Fact \ref{fact:polyhedral-qscnvx-fxn} applies, and the quadratic growth constant is $\kappa_\lambda = \lambda/\theta^2$. 
            Here, $\theta$ is the Hoffman error bound constant (Fact \ref{fact:hoffm-eb}) defined via the inequality constraint system of polytope $P$, and the matrix $A^\top$.
            \par
            We now verify \ref{thm:conic-polyfxn-works:item3}. 
            From \ref{thm:conic-polyfxn-works:item2}, the quadratic growth constant of $\Psi_\lambda$ equals to $\kappa_\lambda$.
            Therefore, if $\lambda$ is bounded below by $\lambda_{\min} > 0$, $\kappa_\lambda$ is bounded below by $\frac{\lambda_{\min}}{\theta^2}$, i.e., $\kappa_\lambda \ge \frac{\lambda_{\min}}{\theta^2}$.
        \end{proof}

\section{Numerical experiments}\label{sec:numerics}
    This section presents the numerical experiments using IAPG.
    Section \ref{sec:verify-inn-cmplx} shows the linear convergence rate for a square sparse matrix $A$.
    Section \ref{sec:app-robust-recovery} presents our findings of IAPG applied to the robust signal recovery problem formulated in \eqref{eqn:robust-tvl2}.
    \subsection{Verify the complexity of the inner loop}\label{sec:verify-inn-cmplx}
        We present numerical experiments demonstrating Theorem \ref{thm:inn-loop-lin-cnvg} using Algorithm \ref{alg:inner-loop}.
        \par
        Let $m=128, n=128, \eta=2, \lambda =1$.
        Let $A \in \RR^{m\times n}$ be $A := H + I$, where $H$ is a sparse matrix whose entries are independently sampled with probability $1/\sqrt{mn}$ of being nonzero, with nonzero entries drawn uniformly from $[0, 1]$.
        We choose $\omega = \eta\Vert \cdot\Vert_1$, which is a conic polyhedral function, so Theorem \ref{thm:conic-polyfxn-works} applies and the inner loop converges linearly.
        The primal proximal problem is $\Phi_{\lambda}(u) = \eta\Vert (H + I)u \Vert_1 + \frac{1}{2\lambda}\Vert u - y\Vert^2$.
        \par
        Recall that in Definition \ref{def:pgd-inner-lp}, the inner loop (Algorithm \ref{alg:inner-loop}) performs PGD on the dual problem:
        \begin{align*}
            \Psi_\lambda(v) := \frac{\lambda}{2}\Vert A^\top v\Vert^2 - \langle A^\top v, y\rangle
            + \delta_{\{x :\Vert \eta^{-1} x \Vert_\infty \le 1\}}(v).
        \end{align*}
        Theorem \ref{thm:inn-loop-lin-cnvg} suggests that the number of iterations to achieve $\mathbf G_\lambda(z_j, v_j) \le \epsilon$ is bounded by $C\ln(\epsilon^{-1})$ for some finite constant $C$.
        \par
        For $i = 0,1, \ldots, 64$, we repeat the experiment 100 times with the following parameters.
        \begin{enumerate}[nosep]
            \item We set $\rho = 0$, and sample $y\in \RR^n$ uniformly from $\{x :\Vert \eta^{-1} x \Vert_\infty \le 1\} = \dom(\omega^\star)$.
            \item Since $\rho = 0$, $\epsilon$ contains only the absolute error, which we set to $\epsilon_i^\circ = 2^{-32 + i/4}$.
        \end{enumerate}
        Figure \ref{fig:5pt-stats-innlp} shows the five-number summary (minimum, Q1, median, Q3, maximum) of the iteration count $j$ at termination against $\epsilon_i^\circ$, over 100 trials per $i$.
        \begin{figure}[H]
            \center
            \includegraphics[width=38em]{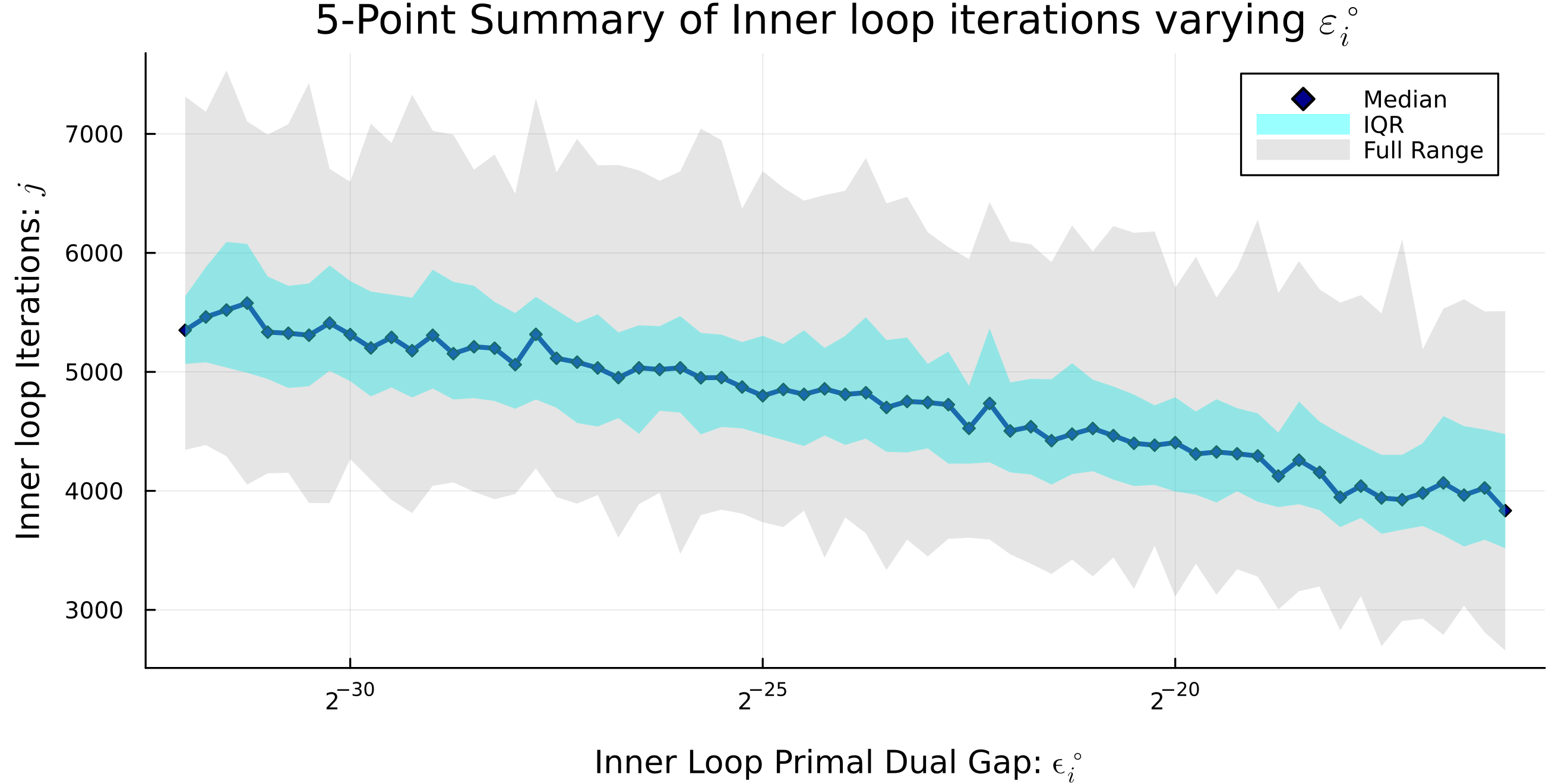}
            \caption{
                Five-number summary of the smallest inner loop iteration $j$ such that $\mathbf G_\lambda(z_j, v_j) \le \epsilon^\circ_i$, plotted against $\epsilon^\circ_i$. The linear growth of $j$ with $-\log_2(\epsilon_i^\circ)$ confirms the $\mathcal{O}(\ln(\epsilon^{-1}))$ bound.
            }
            \label{fig:5pt-stats-innlp}
        \end{figure}

    \subsection{Applications in robust signal recovery}\label{sec:app-robust-recovery}
        Let $\tilde x \in \RR^n$ denote an observed signal corrupted by noise after a linear transformation.
        We consider the following robust TV-$\ell_2$ formulation:
        \begin{align}\label{eqn:tvl2-non-traditional}
            \argmin_{x \in \RR^n} \left\lbrace
                \frac{1}{2}\dist\left(Cx - \tilde x \;|\; [-\lambda, \lambda]^n\right)^2 + \eta \Vert A x\Vert_1
            \right\rbrace.
        \end{align}
        Here, $C \in \RR^{n \times n}$ is a non-uniform box-blurring matrix and $A \in \RR^{(n-1)\times n}$ is the first-order forward difference matrix with non-circular boundary conditions.
        Specifically, $A$ is bi-diagonal with $A_{i, i} = -1$ and $A_{i, i + 1} = 1$ for all $i = 1, 2, \ldots, n - 1$.
        Matrix $C$ discretizes a non-uniform box-blur operation.
        More precisely, consider a signal $f:[0, \tau]\rightarrow \RR$; the non-uniform box-blur maps $f$ to $\tilde f: [0, \tau]\rightarrow \RR$:
        \begin{align*}
            \tilde f(t) = \int_{t - \min(t, l, \tau - t)}^{t + \min(t, l, \tau - t)}\frac{f(s)}{2\min(t, l, \tau - t)}ds.
        \end{align*}
        Here, $\tau \ge l >0$ represents the largest width of the window of the box-blur. 
        It can be seen as a box-blurring process whose kernel width shrinks near the boundary when the center is within distance $l$ of an endpoint.
        Consider the discretized ground truth signal $\bar x \in \RR^n$.
        For all $t \in \{1, \ldots, n\}$, define $w(t) := \min(t - 1, l, n - t)$. 
        Then, we can implement matrix $C\in \RR^{n\times n}$ by:
        \begin{align*}
            (\forall t \in \{1, \ldots, n\}):
            (C x)_t 
            &= 
            \sum_{i = t - w(t)}^{t + w(t)}
            \frac{x_i}{2w(t)}. 
        \end{align*}
        Consequently, $C\in \RR^{n\times n}$ is a square band matrix that is neither Toeplitz nor circular, hence challenging to numerically invert. 
        \par
        Our algorithm is well suited to the robust variant of the TV-$\ell_2$ problem in \eqref{eqn:tvl2-non-traditional}: it only requires the gradient\footnote{
            Let $E\subseteq \RR^n$, we have from the Moreau Envelope $\nabla\left(1/2\dist^2(x| E)\right) = x - \Pi_E x = x - \hprox_{\delta_E}x = \hprox_{\delta^\star_E}(x)$. We only need the proximal operator of support function $\delta_{E}^\star$ to compute the gradient. When $E = [-\lambda, \lambda]^n$, we have $\delta_E^\star = \Vert \cdot\Vert_1$.
        }
        of $\dist\left(Cx - \tilde x \;|\; [-\lambda, \lambda]^n\right)^2$.
        Note that the fidelity term imposes zero penalty on all $x$ satisfying $Cx - \tilde x \in [-\lambda, \lambda]^n$, analogous to the $\epsilon$-insensitive loss in the literature.
        The parameter $\lambda$ therefore controls the tolerance for the discrepancy between the observed signal $\tilde x$ and the blurred signal $Cx$: larger values of $\lambda$ accommodate more noise in the observations, making the formulation more robust to noise in $\tilde x$.
        \par
        Our numerical experiment is based on implementations described in Algorithm \ref{alg:inner-loop}, \ref{alg:outer-loop}.
        The discretized ground truth signal is $(\forall i = 0, \ldots, m)\; \bar x_i = \text{sign}\left(\sin\left(\frac{4\pi i}{m}\right)\right)$, which we corrupt as $\tilde x = C\bar x + 0.3z$, where $z \sim \mathcal N(0, I_m)$.
        The parameters are as follows:
        \begin{enumerate}[nosep]
            \item $n = 2048$, $m = n - 1$.
            \item The box-blurring window width is $l = 128$.
            \item The TV-$\ell_2$ regularization constant is $\eta = 2$, and $\lambda = 0.2$.
            \item For the algorithm (Algorithms \ref{alg:inner-loop}, \ref{alg:outer-loop}), we used $\mathcal E_0 = 64, p = 2$, and $\rho_k = B_k$, hence $L_k = 2B_k$. The algorithm exits when the outer loop detects $\Vert x_k - y_k\Vert \le 10^{-8}$. Other parameters are $r = 1/16$, $s = 4096$ for the inner loop, and $s = 1024$ for the outer loop.
        \end{enumerate}
        The experiment was run once, and Figure \ref{fig:signals-plots} shows the recovered signal, which is very close to the ground truth despite the heavy noise.
        A single run took approximately 12 minutes on a single CPU thread\footnote{Processor: Apple M3 Pro.}, with total inner loop iterations on the order of $2^{18}$.
        This large number of inner loop iterations is attributable to choosing $\eta = 2$, which causes $\hprox_{\lambda \omega^\star}$ to account for more of the computation.
        We empirically observed that the inner loop iteration count decreases significantly for smaller $\eta$, at the cost of weaker total variation penalization on $\eta\Vert Ax\Vert_1$, which degrades the quality of the recovered signal.
        To speed up performance in Julia \cite{bezanson_julia_2017}, we implemented the finite difference matrices $A, A^\top$ using a simple for-loop instead of Compressed Sparse Column (CSC) format, improving memory locality for the inner loop (this yields a tenfold speedup).
        \begin{figure}[H]
            \centering
            \subfloat{
                \includegraphics[width=18em]{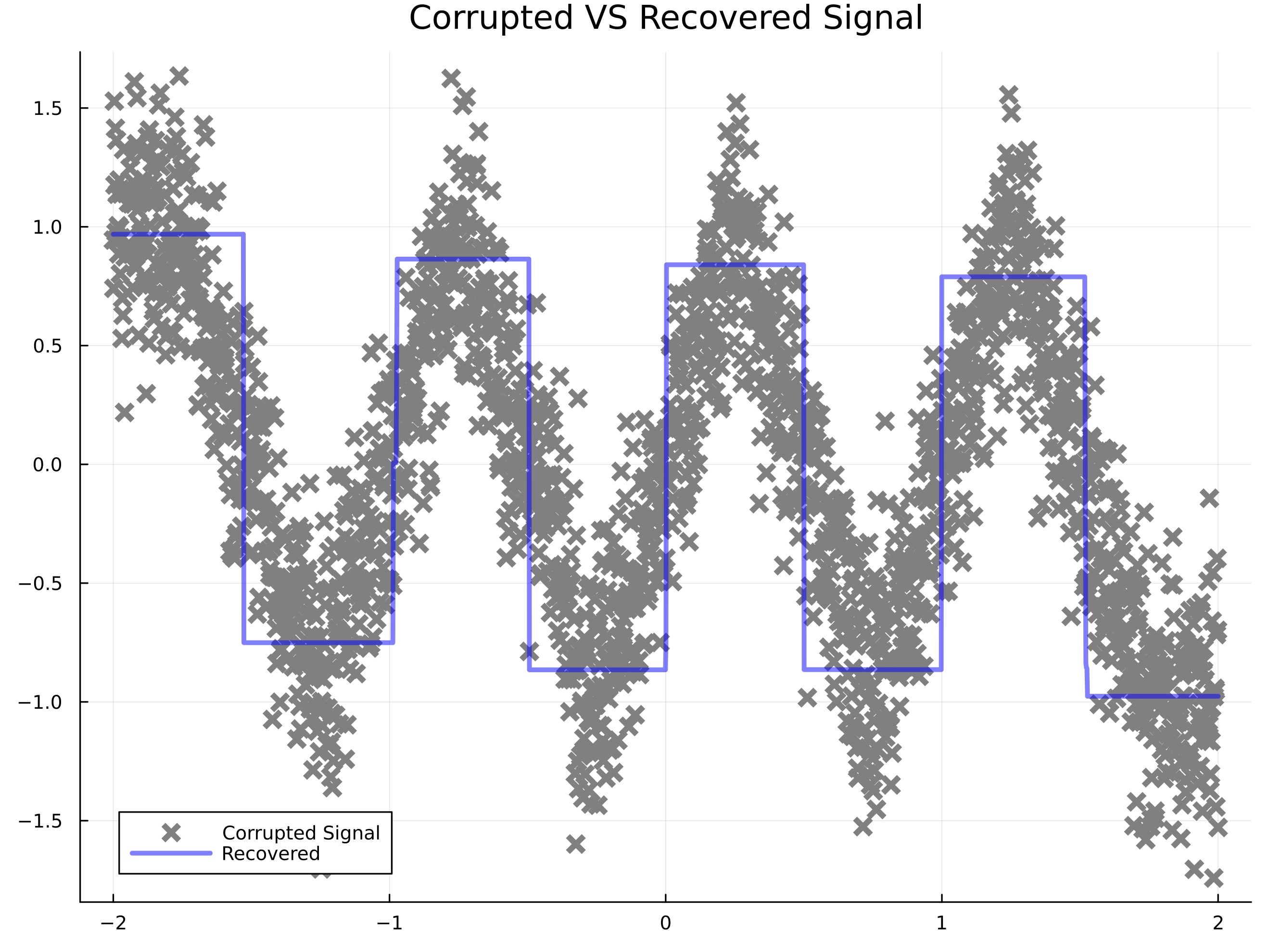}
            }
            \subfloat{
                \includegraphics[width=18em]{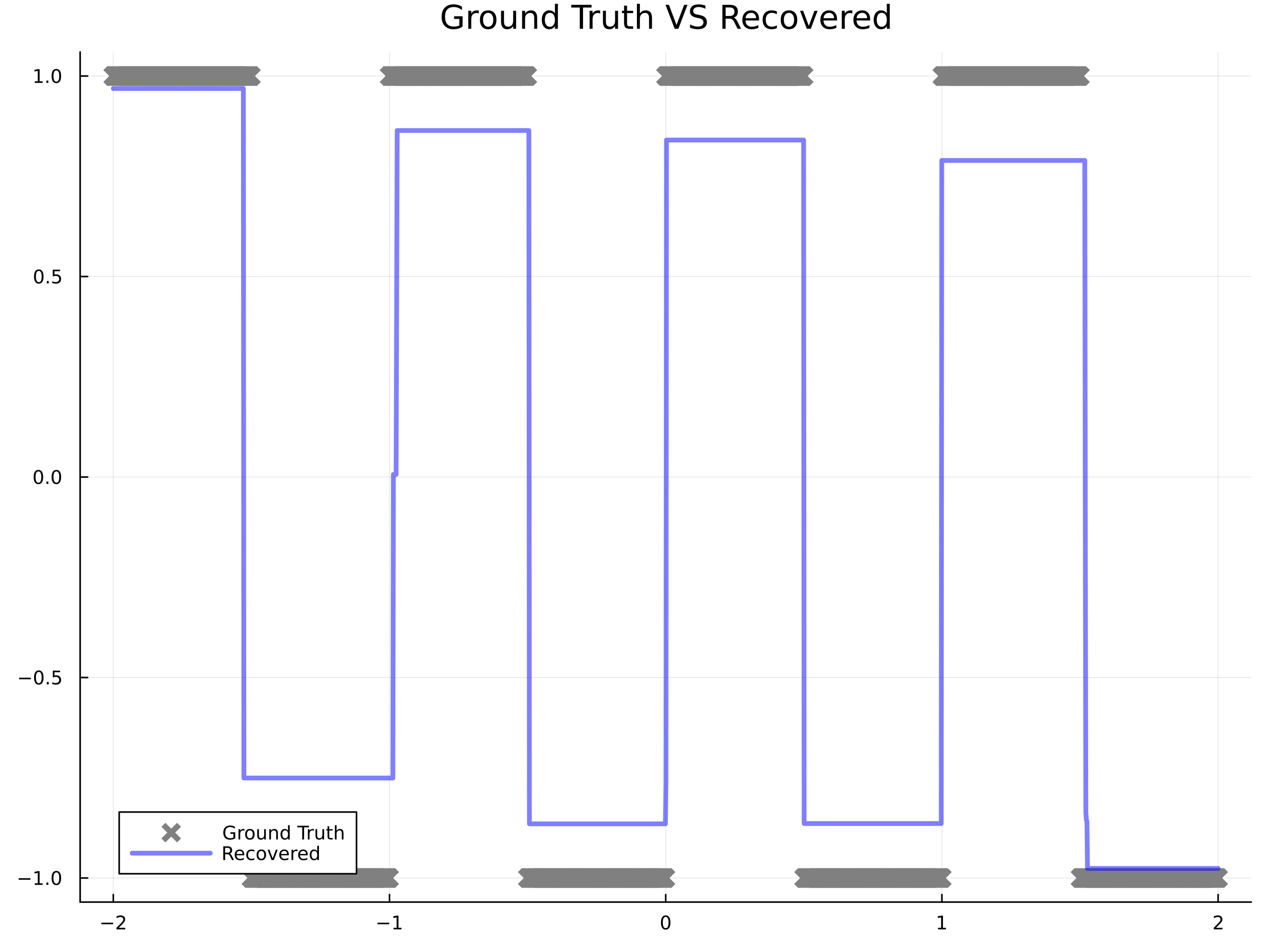}
            }
            \caption{Comparing the observed recovered signal with the observed signal $\tilde x$ and ground truth signal $\bar x$.}
            \label{fig:signals-plots}
        \end{figure}
        We now illustrate the convergence behavior of the algorithm on this experiment.
        The algorithm performs significantly better than the theoretical bound, and we discuss possible reasons for this favorable behavior.
        We track the following quantities at each outer iteration $k$:
        \begin{enumerate}[nosep]
            \item $J_k$, the total inner loop iterations until the tolerance $\epsilon_k$ is met.
            \item $\Vert x_k - y_k\Vert$, the stationarity residual, which upper bounds $\dist\left(\mathbf 0 |\partial_{\epsilon_k}F(x_k)\right)$ by Lemma \ref{lemma:pg-residual}.
            \item $\epsilon_k^\circ$, the absolute tolerance given to the outer loop.
        \end{enumerate}
        Our first set of results is shown in Figure \ref{fig:deblurring-set1}.
        Figure \ref{fig:deblurring-set1:a} illustrates a strong negative relationship between $J_k$ and $\ln(\epsilon_k^\circ)$: the inner loop iteration count $J_k$ grows proportionally to $\ln(1/\epsilon_k^\circ)$, confirming Theorem \ref{thm:inn-loop-lin-cnvg}\ref{thm:inn-loop-lin-cnvg:item3}.
        The first few outliers occur because the initial absolute tolerance is $\epsilon_0 = \mathcal E_0 = 64$; only afterwards does $\epsilon_k^\circ$ follow $\epsilon_k^\circ = L_k^{-1}L_0^{-1}\alpha_k^{2}\mathcal E_0 k^{-p}$.
        Figure \ref{fig:deblurring-set1:b} shows a strong linear relationship between $J_k$ and $\ln(k)$, verifying \eqref{prop:inner-lp-cmplx:rslt1} in Proposition \ref{prop:inner-lp-cmplx}, which states that the inner loop has a linear convergence rate.
        \begin{figure}[htp]
            \centering 
            \subfloat[$J_k$ and absolute tolerance: $\epsilon_k^\circ$]{
                \label{fig:deblurring-set1:a}
                \includegraphics[height=0.27\textheight]{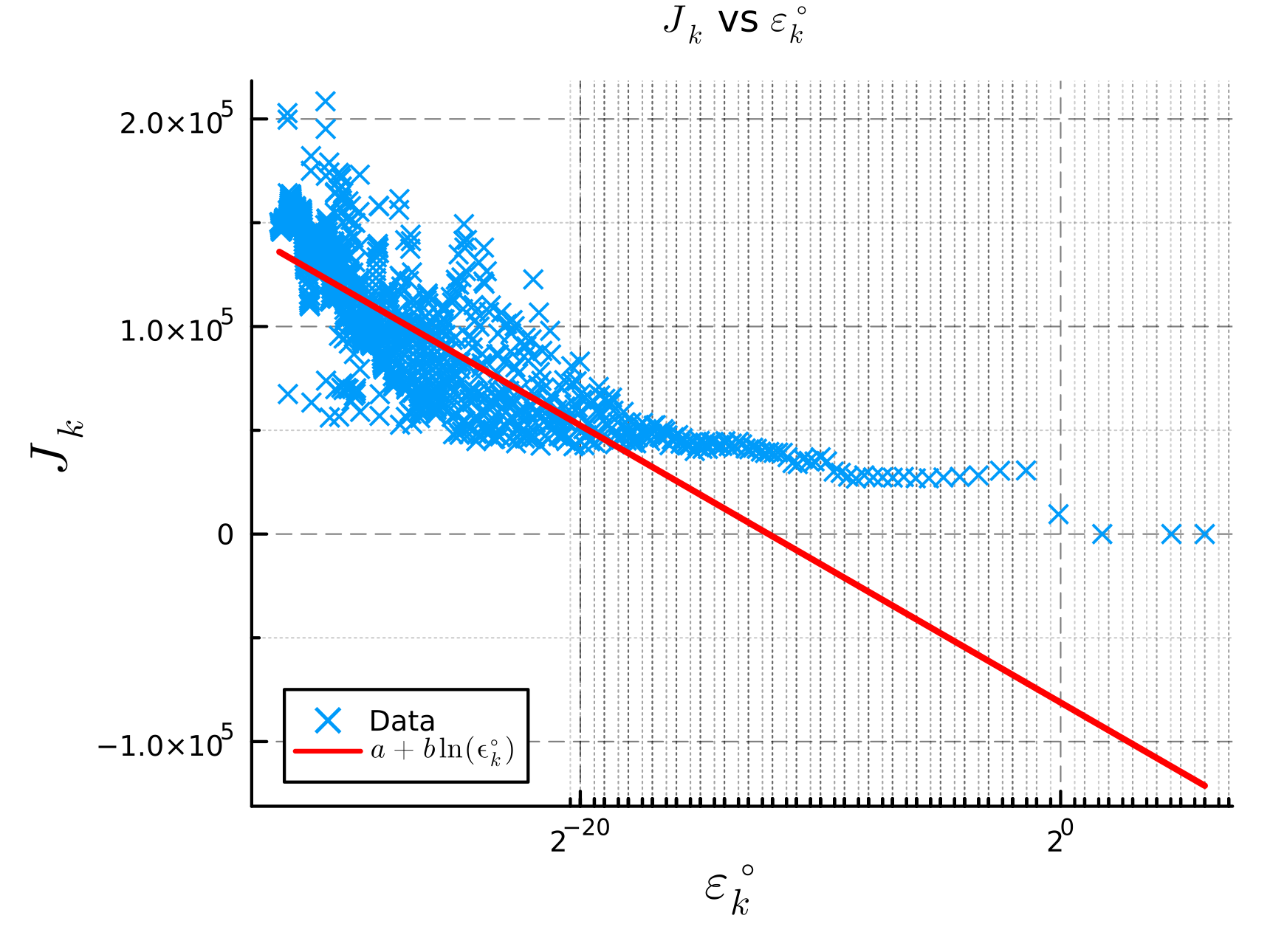}
            }\hfill
            \subfloat[$J_k$ and $k$]{
                \label{fig:deblurring-set1:b}
                \includegraphics[height=0.27\textheight]{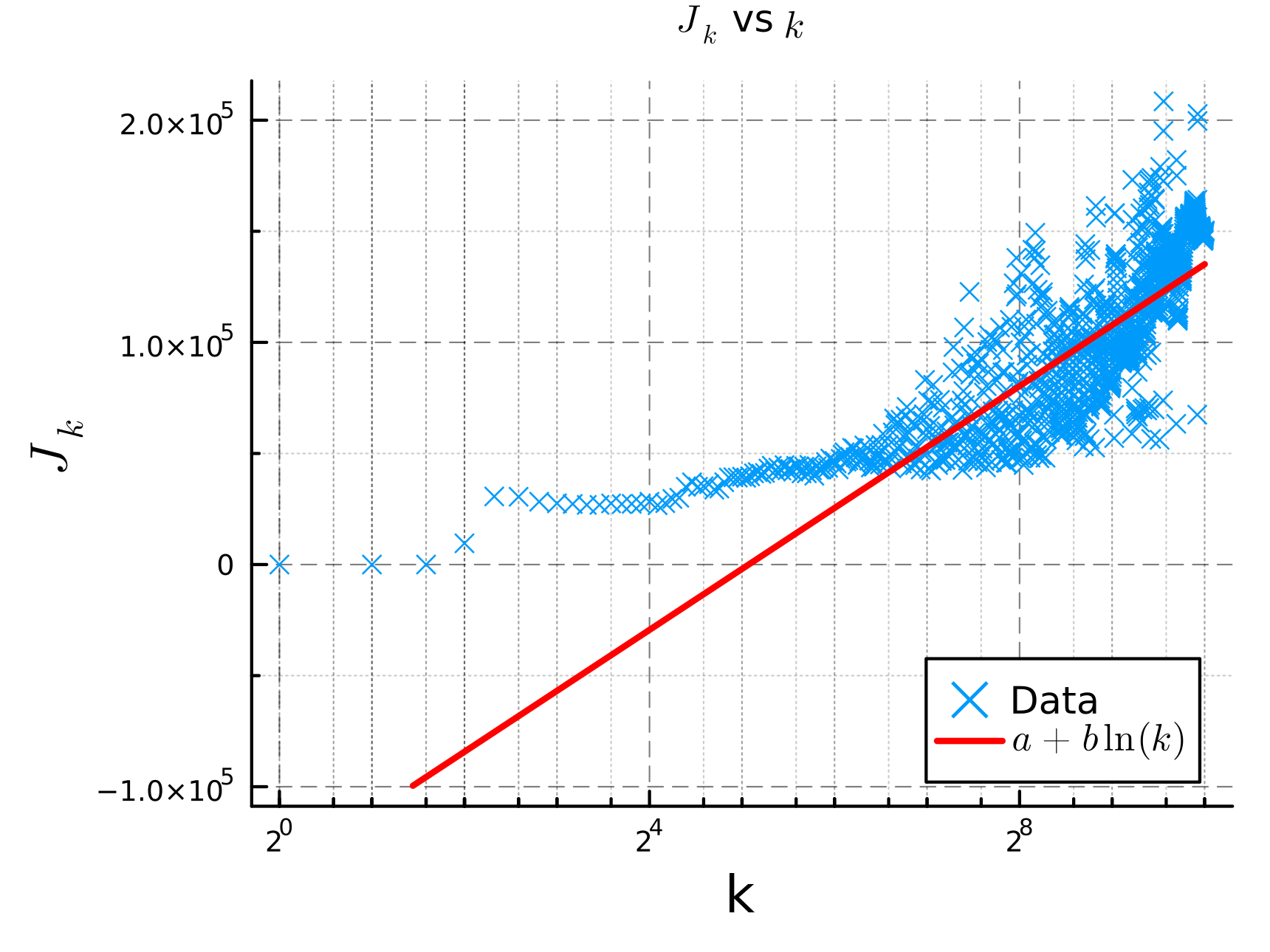}
            }
            \caption{(a): The model for the reference line is $y = a + b\ln(\epsilon_k^\circ)$ and the fitted values are: $a\approx -8.12\times 10^4, b\approx -9.63\times 10^4$. (b): The model for the reference line is $y = a + b \ln(k)$. The values are $a \approx -1.39\times 10^{5}, b \approx 3.96 \times 10^{4}$. }
            \label{fig:deblurring-set1}
        \end{figure}
        \par
        Our second set of results, shown in Figure \ref{fig:deblurring-set2}, is more revealing.
        Figure \ref{fig:deblurring-set2:a} shows the relation between the cumulative inner loop iterations $\sum_{i = 0}^{k} J_i$ and the residual $\Vert x_k - y_k\Vert$.
        On a log-log plot, we show that for positive constants $a, b, c, c_1$, the following holds for small enough $\Vert x_k - y_k\Vert$:
        \begin{align}\label{eqn:numerical-hypo}
            \Vert x_k - y_k\Vert \approx
                \frac{
                    c \max\left(
                        1, \left[\ln
                            \max\left(c_1, \sum_{i = 0}^{k} J_i
                            \right)
                        \right]^a
                    \right)
                }{
                    \max\left(c_1, \sum_{i = 0}^{k} J_i\right)^b
                }.
        \end{align}
        Taking the log on both sides of \eqref{eqn:numerical-hypo}:
        \begin{align*}
            \ln\Vert x_k - y_k\Vert \approx
            \ln c +
            a \max\left(0,
                    \ln\ln \max\left(c_1, \sum_{i = 0}^{k} J_i
                        \right)
            \right)
            - b \ln \max\left(c_1, \sum_{i = 0}^{k} J_i\right).
        \end{align*}
        We determine $c,c_1 a, b$ by multilinear regression, which yields the reference line in Figure \ref{fig:deblurring-set2:a}.
        Notably, $\Vert x_k - y_k\Vert$ decreases faster than $\mathcal O(\ln(k)/k)$ relative to $\sum_{i = 0}^{k}J_i$: we measured $b \approx 2.33$, a value much larger than $1$.
        Figure \ref{fig:deblurring-set2:b} plots the absolute error $\epsilon_k^\circ$ and the relative error $\frac{B_k}{2}\Vert x_k - y_k\Vert^2$ (by the choice $\rho_k = B_k$) on a log-log scale for each outer iteration.
        It is notable that the relative tolerance is an order of magnitude larger than the absolute tolerance, a consequence of the choice $\rho_k = B_k$.
        \begin{figure}[H]
            \centering
            \subfloat[$\sum_{i = 0}^{k}J_i$ and $\Vert x_k - y_k\Vert.$]{
                \label{fig:deblurring-set2:a}
                \includegraphics[height=0.48\textwidth]{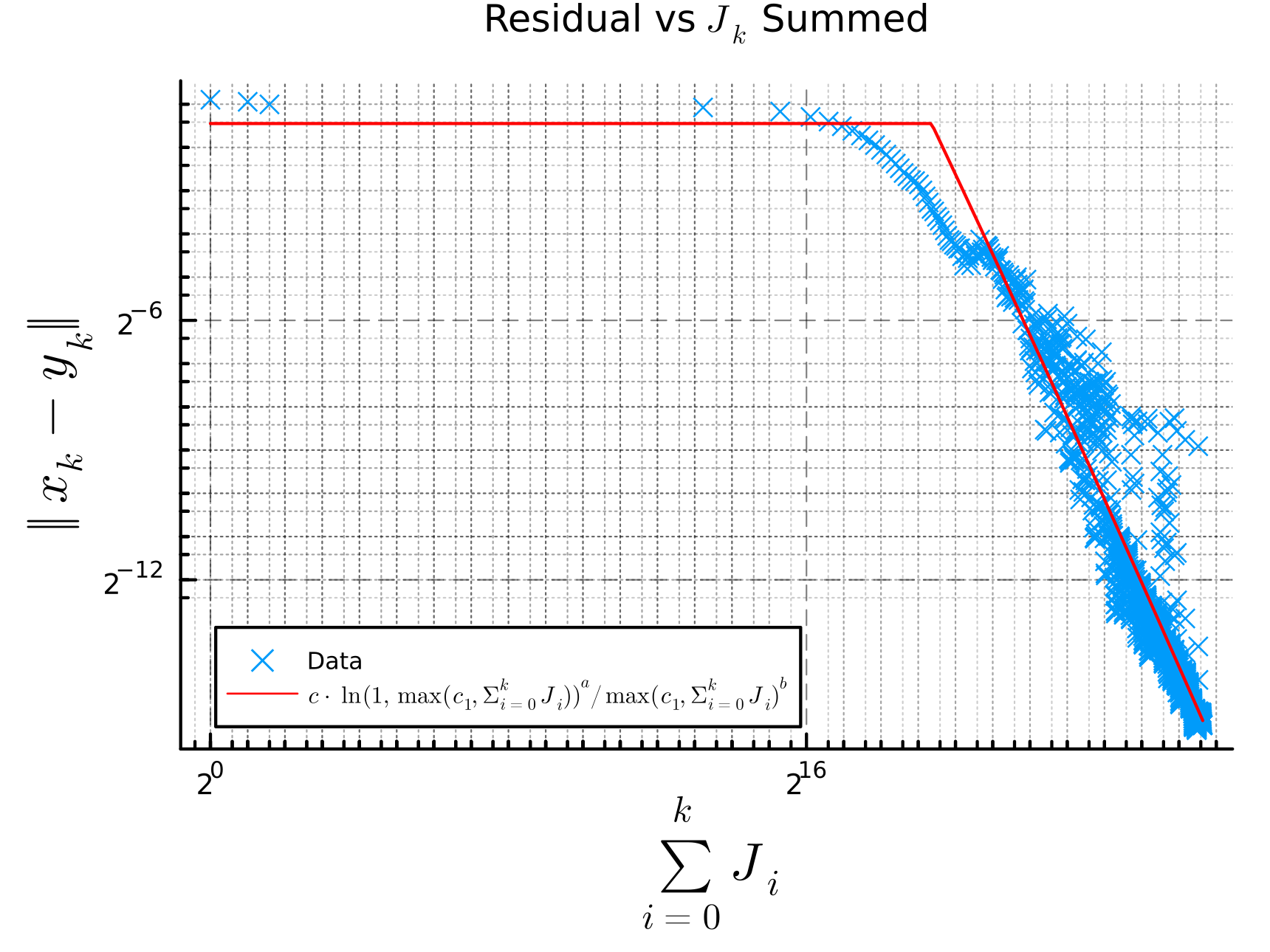}
            }\\
            \subfloat[$\epsilon^\circ_k$ and $(\rho_k/2)\Vert x_k - y_k\Vert^2$]{
                \includegraphics[height=0.45\textwidth]{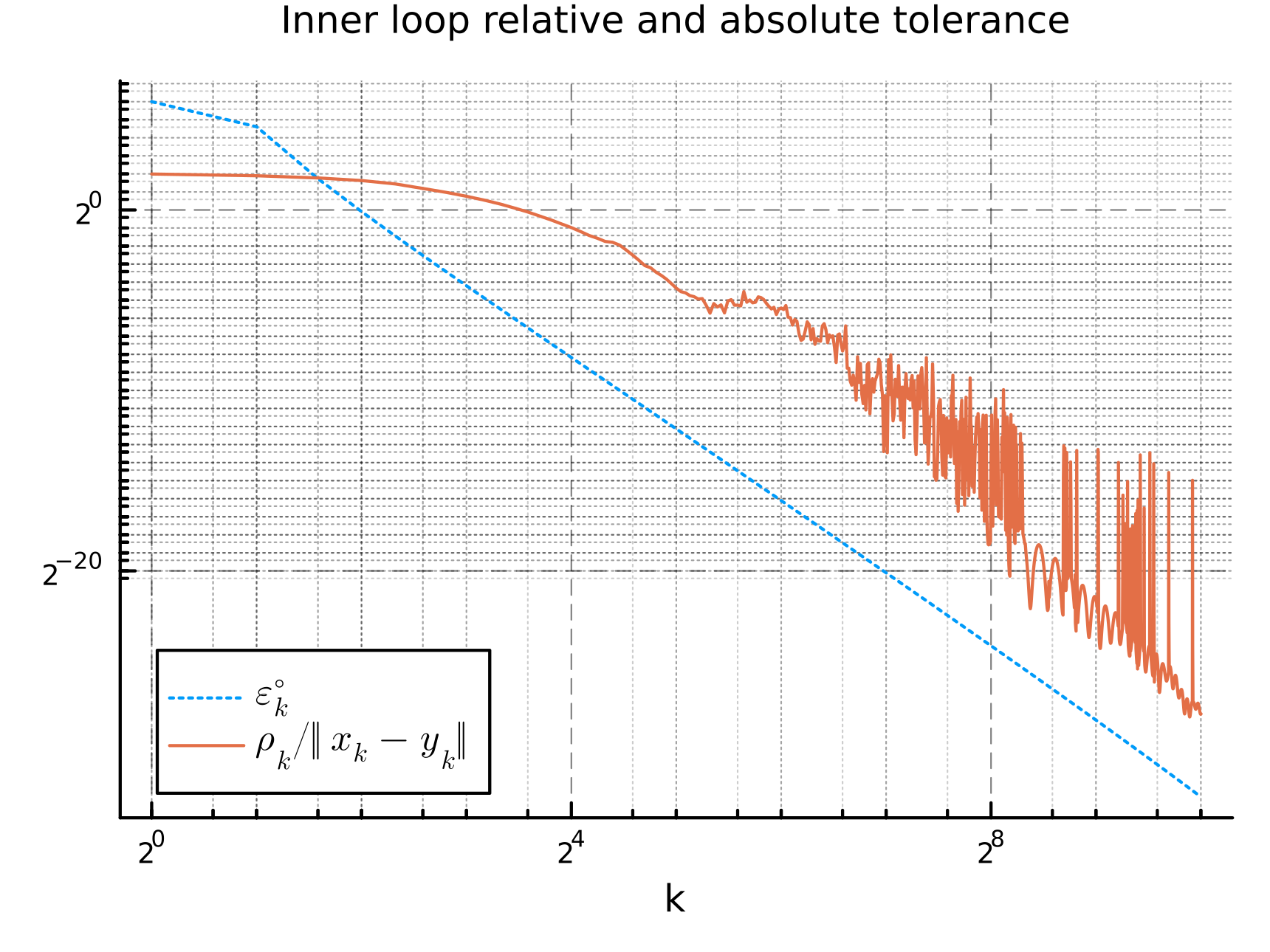}
                \label{fig:deblurring-set2:b}
            }
            \caption{
                (a): The model we fitted for the reference line is $y = \frac{c\max(1, \log(c_1, x)^{a})}{\max(c_1, x)^{b}}$ $\Vert x_k - y_k\Vert$ is on the y-axis, $\sum_{i = 0}^{k}J_i$ is on the x-axis. 
                The best fitted value is $c \approx 2.41\times 10^{5}, c_1\approx 6.72\times 10^{5}, a\approx 6.89, b\approx 2.33$. 
                (b): The figure shows relative error $\frac{\rho_k}{2}\Vert x_k - y_k\Vert^2$, and absolute error $\epsilon_k^\circ$, illustrating that relative error is relatively bigger than absolute error for the choice $\rho_k = B_k$. 
            }
            \label{fig:deblurring-set2}
        \end{figure}
        \par
        In summary, these experiments validate our theoretical contributions: our variant of IAPG enables a first-order solution to the robust TV-$\ell_2$ formulation \eqref{eqn:tvl2-non-traditional} that leverages the composite additive structure.
        Moreover, the empirical convergence exceeded our expectations, with the inner loop $J_k$ scaling at $\mathcal O(\ln(k))$ and the inner and outer loop together scaling below $\mathcal O(\ln(k)/k)$, pointing toward strong potential for even larger-scale problems where such favorable scaling is critical.
        This is a practical advantage that, to the best of our knowledge, has no precedent in the literature.

\section{Conclusions and future works}\label{sec:future-works}
% \#TODO: "future works" -> "Future Work". "Works" (plural) is non-standard; the conventional heading is "Conclusions and Future Work".
    In this paper, we study the convergence of the Inexact Accelerated Proximal Gradient (IAPG) method for \eqref{expr:major-optimization-problem}, showing that an error bound condition on the dual of the inexact proximal point problem yields faster global convergence.
    Following Villa et al. \cite{villa_accelerated_2013}, we extend their results to show that the inner loop achieves a linear convergence rate when $\omega$ is conic polyhedral.
    % \#TODO: "Following Villa et al. ..., we extend their results" — dangling construction: grammatically, the subject "we" is following Villa et al., which is fine, but "extend their results" suggests Villa et al.'s results are doing the following. Clearer: "Building on Villa et al. \cite{villa_accelerated_2013}, we extend their results to show..."
    More precisely, when $\omega$ is conic polyhedral, the dual of the inexact proximal point problem satisfies a quadratic growth condition, which yields linear convergence of the inner loop.
    By incorporating global Lipschitz continuity of $\omega$, we further show that this linear convergence rate holds uniformly over all initial points supplied by the outer loop.
    Together, these results yield a total complexity of $\mathcal{O}(\varepsilon^{-1/2}\ln(\varepsilon^{-1}))$ on the total number of evaluations of $\nabla f$ and $\hprox_{\lambda \omega^\star}$, improving upon all prior complexity results for IAPG.
    To validate our theoretical results, we formulate a robust TV-$\ell_2$ problem with a non-uniform box blur matrix and a TV penalization term with a large regularization multiplier.
    % \#TODO: "box blur matrix" -> "box-blur matrix". Hyphenate compound modifier.
    Numerical evidence confirms that the complexity, measured by the number of evaluations of $\hprox_{\lambda\omega^\star}$, scales in accordance with our theoretical predictions.
    \par
    Despite the advances made in this work, several open problems remain.
    \begin{enumerate}
        \item Can the proximal point problem \eqref{eqn:primal-pp} be solved stochastically, and what convergence guarantees would carry over?
        \item Our experiments show that the inner loop accounts for the majority of total iterations. Can the dual problem \eqref{eqn:dual-pp} be parallelized with minimal overhead, or does incorporating a proximal Quasi-Newton method or preconditioning reduce the inner loop iteration count?
        \item Can the three-operator splitting problem $\min_{x\in \RR^n}\{f(x) + \delta_C(x) + \omega(Ax)\}$ where $\delta_C$ is the indicator function of a convex set $C \subseteq \RR^n$, be addressed within our framework, and what regularity conditions on $\omega$ and $C$ would ensure efficient solution of the inexact proximal point problem?
        \item Would combining our results with those of Rasch and Chambolle \cite{rasch_inexact_2020} yield a total complexity of $\mathcal{O}(\ln(\varepsilon^{-1})\varepsilon^{-1})$ for convergence of the duality gap?
        \item Can our results be combined with adaptive restarts from Alamo et al. \cite{alamo_gradient_2019} or Hessian damping from Attouch \cite{attouch_first-order_2022} to further improve convergence?
    \end{enumerate}

\section*{Acknowledgements}
    The research of HL and XW was partially supported by the NSERC Discovery Grant of Canada.

\appendix
\section{Necessary intermediate results}
    \begin{lemma}[That conjugate for the dual of proximal problem]\;\label{lemma:chore1}\\
        Let $f:\RR^n \rightarrow \overline \RR: u\mapsto \frac{1}{2\lambda}\Vert u - v\Vert^2$. 
        Then its conjugate is given by
        \begin{align*}
            f^\star(v) = \frac{1}{2\lambda}\Vert \lambda v + y\Vert^2 - \frac{1}{2\lambda}\Vert y\Vert^2. 
        \end{align*}
    \end{lemma}
    \begin{proof}
        Recall the following properties for any closed, proper convex function $f: \RR^n \rightarrow \overline \RR$. 
        Let $a \in \RR^n$ be any vector, let $\alpha > 0, c \in\RR$. 
        Then, we introduce these three properties of conjugating a convex function: 
        \begin{enumerate}[nosep]
            \item $(\alpha f)^\star = \alpha f^\star \circ (\alpha^{-1}I)$. 
            \item $(f + c)^\star(y) = f^\star(y) - c$. 
            \item $\left(x \mapsto f(x) + \langle x, a\rangle\right)^\star(y) = f^\star(y - a)$. 
        \end{enumerate}
        From here we have: 
        \begin{align*}
            f^\star(v) 
            &= 
            \left(
                u \mapsto \lambda^{-1}\left(\frac{1}{2}\Vert u\Vert^2 - \langle u, y\rangle\right) 
                + \frac{1}{2\lambda} \Vert y\Vert^2
            \right)^\star(v)
            \\
            &= \left(
                u \mapsto \lambda^{-1}\left(\frac{1}{2}\Vert u\Vert^2 - \langle u, y\rangle\right) 
            \right)^\star(v)
            - \frac{1}{2\lambda} \Vert y\Vert^2
            \\
            &= 
            \left[
                \lambda^{-1}\left(
                u \mapsto \left(\frac{1}{2}\Vert u\Vert^2 - \langle u, y\rangle\right) 
            \right)^\star\circ(\lambda I)
            \right](v)
            - \frac{1}{2\lambda} \Vert y\Vert^2
            \\
            &= \left[
                \lambda^{-1}\left(
                u \mapsto \left(
                    \frac{\Vert \cdot\Vert^2}{2}
                \right)^\star(u + y)
            \right)\circ(\lambda I)
            \right](v)
            - \frac{1}{2\lambda} \Vert y\Vert^2
            \\
            &= \left[
                \lambda^{-1}\left(
                    u \mapsto \frac{\Vert u + y\Vert^2}{2}
                \right)\circ(\lambda I)
            \right](v)
            - \frac{1}{2\lambda} \Vert y\Vert^2
            \\
            &= \lambda^{-1}\left(
                \frac{1}{2}\Vert \lambda v + y\Vert^2
            \right) 
            - \frac{1}{2\lambda} \Vert y\Vert^2. 
        \end{align*}
    \end{proof}

    \begin{lemma}[Lipschitz constant of convex function]\;\label{lemma:lipz-cnvx-fxn}
        Let $f: \RR^n \rightarrow \overline \RR$ be a closed, convex, proper function. 
        Let $\partial f$ be its convex subdifferential. 
        Then, 
        \begin{enumerate}[nosep]
            \item for all $x\in \RR^n, y\in \RR^n$ it has: $|f(x) - f(y)| \le \left(\sup_{x\in \dom \partial f} \dist (\partial f(x) \;|\; \mathbf 0)\right)\Vert y - x\Vert$. 
            \item If in addition, the function is $K$ Lipschitz continuous globally on $\RR^n$, then: $(\forall y \in \RR^n)(\forall v \in \partial f(y)):\; K \ge \Vert v\Vert$. 
        \end{enumerate}
    \end{lemma}
    \begin{proof}
        We give a direct proof for the first result. 
        Let $x, y \in \RR^n$ be arbitrary.
        Choose $v_x \in \partial f(x)$ and $v_y \in \partial f(y)$ such that $\Vert v_x\Vert = \dist (\partial f(x) | \mathbf 0), \Vert v_y\Vert = \dist (\partial f(y) | \mathbf 0)$.
        This is possible because $\partial f(x)$ is closed for all $x \in \dom \partial f$. 
        Therefore: 
        \begin{align*}
            |f(x) - f(y) | &\le 
            \max(f(x) - f(y), f(y) - f(x))
            \\
            &\underset{(1)}\le \max(-\langle v_x, y - x\rangle, - \langle v_y, x - y\rangle)
            \\
            &\le \max(\Vert v_x\Vert, \Vert v_y\Vert)\Vert y - x\Vert
            \\
            &\le \left(
                \sup_{x \in \dom \partial f} \dist(\partial f(x) \;|\; \mathbf 0)
            \right)\Vert y - x\Vert. 
        \end{align*}
        At (1), we used the fact that $f(x) - f(y) \le - \langle v_x, y - x\rangle$ and, $f(y) - f(x) \le - \langle v_y, x - y\rangle$ which follows from the subgradient inequality of convex subgradient. 
        \par
        We now show the second result. 
        The following holds for all $x \in \RR^n$ and $y \in \RR^n$: 
        \begin{align*}
            f(x) - f(y) &\le \sup_{v \in \partial f(y)} \langle v, x - y\rangle
            \\
            &\underset{(2)}= f'(y; x - y)
            \\
            &= \lim_{\delta \searrow 0} \frac{f(y + \delta(x - y)) - f(y)}{\delta}
            \\
            &\underset{(3)}\le \lim_{\delta \searrow 0} \frac{\delta K_f\Vert x - y\Vert}{\delta}
            \\
            &= K_f\Vert x - y\Vert. 
        \end{align*}
        At (2), we used the max formula of Beck \cite[Theorem 3.26]{beck_first-order_2017}. 
        At (3), the inequality holds since $f$ is $K_f$ Lipschitz continuous. 
        Since this is true for all $x, y$, it implies that $|f(x) - f(y)|\le K_f \Vert x - y\Vert$. 

    \end{proof}

\bibliographystyle{siam}
\bibliography{references/refs.bib}
\end{document}